\documentclass[11pt,a4paper]{article}

\usepackage[english]{babel}
\usepackage[utf8]{inputenc}
\usepackage[T1]{fontenc}
\usepackage{lmodern} \normalfont 
\DeclareFontShape{T1}{lmr}{bx}{sc} { <-> ssub * cmr/bx/sc }{}
\usepackage{parskip}
\usepackage{microtype}	

\usepackage{amssymb}
\usepackage[intlimits]{amsmath}
\usepackage{amsthm}
\usepackage{mathtools}
\usepackage{mathrsfs}
\usepackage{accents} 
\usepackage{etoolbox}
\usepackage{physics}
\usepackage{siunitx}
\usepackage{dsfont} 
\begingroup
    \makeatletter
    \@for\theoremstyle:=definition,remark,plain\do{%
        \expandafter\g@addto@macro\csname th@\theoremstyle\endcsname{%
            \addtolength\thm@preskip\parskip
            }%
        }
\endgroup

\usepackage{tikz}
\usetikzlibrary{calc,er,positioning,plotmarks}
\usepackage{pgfplots}
\pgfplotsset{compat=newest}
\pgfplotsset{
        colormap={parula}{
            rgb255=(53,42,135)
            rgb255=(15,92,221)
            rgb255=(18,125,216)
            rgb255=(7,156,207)
            rgb255=(21,177,180)
            rgb255=(89,189,140)
            rgb255=(165,190,107)
            rgb255=(225,185,82)
            rgb255=(252,206,46)
            rgb255=(249,251,14)
        },
}
\usepgfplotslibrary{groupplots}
\usepgfplotslibrary{external} 
\usepackage{standalone}
\usepackage{color}
\usepackage{graphicx}
\usepackage[margin=10pt,font=small,labelfont=bf,labelsep=endash]{caption}
\usepackage{subcaption}

\usepackage{algorithm}
\usepackage{algorithmicx}
\usepackage{algpseudocode}

\usepackage[left=3cm,right=3cm,top=3cm,bottom=3cm]{geometry}
\usepackage{booktabs}
\usepackage{paralist}

\usepackage[noadjust]{cite}
\usepackage{multirow} 
\usepackage{enumitem} 

\usepackage[acronym,nomain,toc,nonumberlist,nopostdot,style=super]{glossaries}
\makeglossaries

\newacronymstyle{emphAcronym}
{%
  \GlsUseAcrEntryDispStyle{long-short}%
}%
{%
  \GlsUseAcrStyleDefs{long-short}%
}
\setacronymstyle{emphAcronym}

\let\oldbibliography\thebibliography
\renewcommand{\thebibliography}[1]{%
  \small
  \oldbibliography{#1}%
  \setlength{\itemsep}{0pt}%
}

\usepackage[colorlinks,citecolor=black,urlcolor=black]{hyperref}
\usepackage[nameinlink]{cleveref}

\numberwithin{equation}{section}
\numberwithin{table}{section}
\numberwithin{figure}{section}

\newtheorem{theorem}{Theorem}[section]

\newtheorem{prop}[theorem]{Proposition}
\newtheorem{lemma}[theorem]{Lemma}
\crefname{lemma}{lemma}{lemmata}
\Crefname{lemma}{Lemma}{Lemmata}

\crefname{corollary}{corollary}{corollaries}
\Crefname{corollary}{Corollary}{Corollaries}
\newtheorem{assumption}[theorem]{Assumption}
\crefname{assumption}{assumption}{assumptions}
\Crefname{assumption}{Assumption}{Assumptions}

\theoremstyle{definition}
\newtheorem{remark}[theorem]{Remark}
\AtEndEnvironment{remark}{\hfill\ensuremath{\diamondsuit}}
\theoremstyle{definition}

\newtheorem{example}[theorem]{Example}
\AtEndEnvironment{example}{\hfill\ensuremath{\bigcirc}}

\DeclareMathOperator{\Span}{span}
\DeclareMathOperator{\diag}{diag}
\DeclareMathOperator{\rel}{rel}

\newcommand{\RealNumbers}{\mathbb{R}}
\newcommand{\addSet}{\mathcal{A}}

\newcommand{\InitialCondition}{\Sol_{0}}
\newcommand{\InitialConditionROM}{\rom{\Sol}_{0}}
\newcommand{\TimeEnd}{T}

\newcommand{\StateDimensionReduced}{r}
\newcommand{\dimROM}{\StateDimensionReduced}
\newcommand{\dimFOM}{N}
\newcommand{\dimPathSpace}{q}
\newcommand{\param}{\mu}

\newcommand{\CoordinateMatrix}{\Phi}
\newcommand{\Coefficient}{\Coeff}

\newcommand{\PodOperator}{\mathcal{R}}
\newcommand{\PODoperator}{\PodOperator}
\newcommand{\Eigenvalue}{\lambda}
\newcommand{\AutoCorrelationFunction}{R}
\newcommand{\rom}[1]{\widetilde{#1}}

\newcommand{\PODmassMatrix}{M}
\newcommand{\PODmassMatrixState}{M_{\mathrm{\SolROM}}}
\newcommand{\PODmassMatrixPath}{M_{\mathrm{\Path}}}
\newcommand{\PODpathMatrix}{N}

\newcommand{\ShiftOperator}{\mathcal{S}}

\newcommand{\dt}{\,\mathrm{d}t}

\newcommand{\X}{\mathscr{X}}
\newcommand{\Y}{\mathscr{Y}}
\newcommand{\FrozenSolution}{v}
\newcommand{\FrozenSolROM}{\tilde{\FrozenSolution}}
\newcommand{\PhaseCondition}{\Psi}
\newcommand{\LTwo}[1]{L^{2}\left(#1\right)}

\newcommand{\LocSquareIntegrableFunctions}[1]{L_{\mathrm{loc}}^{2}\left(#1\right)}

\newcommand{\Sol}{z}
\newcommand{\SolROM}{\tilde{\Sol}}
\newcommand{\SolROMlifted}{\hat{\Sol}}

\newcommand{\ParDer}[2]{\partial_{#1}{#2}}
\newcommand{\RhsOp}{\mathcal{F}}
\newcommand{\F}{\RhsOp}
\newcommand{\FROM}{\rom{F}}
\newcommand{\FROMstate}{\FROM_{\mathrm{\SolROM}}}
\newcommand{\FROMpath}{\FROM_{\mathrm{\Path}}}

\newcommand{\Coeff}{\SolROM}
\newcommand{\Mode}{\varphi}
\newcommand{\Hom}{\Transformation}
\newcommand{\Transformation}{\mathcal{T}}

\newcommand{\Path}{p}
\newcommand{\PathSpace}{\mathcal{P}}

\newcommand{\ie}{i{.}e{.}}
\newcommand{\st}{s{.}t{.}}

\newcommand{\operatorNorm}[1]{{\left\vert\kern-0.25ex\left\vert\kern-0.25ex\left\vert #1 \right\vert\kern-0.25ex\right\vert\kern-0.25ex\right\vert}}

\newcommand{\ipX}[1]{\left\langle #1\right\rangle_{\mspace{-3mu}\X}}

\newacronym{MOR}{MOR}{model order reduction}
\newacronym{ROM}{ROM}{reduced order model}
\newacronym{FOM}{FOM}{full order model}
\newacronym{SVD}{SVD}{singular value decomposition}
\newacronym{POD}{POD}{proper orthogonal decomposition}
\newacronym{sPOD}{shifted POD}{shifted proper orthogonal decomposition}
\newacronym{EIM}{EIM}{empirical interpolation method}
\newacronym{DEIM}{DEIM}{discrete empirical interpolation method}
\newacronym{DAE}{DAE}{differential-algebraic equation}
\newacronym{MFEM}{MFEM}{moving finite element method}

\title{Projection-Based Model Reduction with Dynamically Transformed Modes}
\author{Felix Black\footnotemark[1]~\footnotemark[2] \and Philipp Schulze\footnotemark[1]~\footnotemark[2] \and Benjamin Unger\footnotemark[1]~\footnotemark[3]}
\date{\today}

\begin{document}

\maketitle
\renewcommand{\thefootnote}{\fnsymbol{footnote}}
\footnotetext[1]{Institut f\"ur Mathematik,
Technische Universität Berlin, Str.\ des 17.~Juni~136,
10623~Berlin,
Germany,
\texttt{\{black,pschulze,unger\}@math.tu-berlin.de}. }
\footnotetext[2]{
The work of these authors is supported by the DFG Collaborative Research Center 1029 \emph{Substantial efficiency increase in gas turbines through direct use of coupled unsteady combustion and flow dynamics},
project number 200291049.}
\footnotetext[3]{
The work of this author is supported by the DFG Collaborative Research Center 910 \emph{Control of self-organizing nonlinear systems: Theoretical methods and concepts of application}, project number 163436311.}
%
\begin{abstract}
 	We propose a new model reduction framework for problems that exhibit transport phenomena. 
     As in the moving finite element method (MFEM), our method employs time-dependent transformation operators and, especially, generalizes MFEM to arbitrary basis functions.
     The new framework is suitable to obtain a low-dimensional approximation with small errors even in situations where classical model order reduction techniques require much higher dimensions for a similar approximation quality. 
     Analogously to the MFEM framework, the reduced model is designed to minimize the residual, which is also the basis for an a-posteriori error bound. 
     Moreover, since the dependence of the transformation operators on the reduced state is nonlinear, the resulting reduced order model is obtained by projecting the original evolution equation onto a nonlinear manifold.
     Furthermore, for a special case, we show a connection between our approach and the method of freezing, which is also known as symmetry reduction.
     Besides the construction of the reduced order model, we also analyze the problem of finding optimal basis functions based on given data of the full order solution.
     Especially, we show that the corresponding minimization problem has a solution and reduces to the proper orthogonal decomposition of transformed data in a special case.
     Finally, we demonstrate the effectiveness of our method with several analytical and numerical examples. 
\vskip .3truecm

\vskip .3truecm

{\bf AMS(MOS) subject classification:} 35Q35, 65M15, 37L65
\end{abstract}


\section{Introduction}
Over the past three decades, \gls{MOR} has become an established tool to reduce the computational cost for obtaining high fidelity solutions of (partial) differential-algebraic equations that are required in parameter studies, controller design, and optimization. 
The key observation that is used in \gls{MOR} is that in many applications, the solution evolves in a low-dimensional manifold, which itself can be embedded approximately in a low-dimensional subspace. 
The projection of the dynamical system onto this low-dimensional subspace diminishes the computational cost while maintaining a high fidelity solution. 
A successful \gls{MOR} scheme can identify a suitable low-dimensional subspace and quantify the error of the solution of the resulting surrogate model with respect to the solution of the original dynamics. 
For an overview of such methods we refer to  \cite{BenCOW17,Ant05,QuaMN16,HesRS16,QuaR14} and the recent surveys \cite{BauBF14,BenGW15}. 
Most of these \gls{MOR} methods are formulated in a projection framework. 
In more detail, consider a separable Hilbert space $(\X, \ipX{\cdot,\cdot})$ with induced norm $\|\cdot\|_{\X}$ and an evolution equation of the form
\begin{align}
	\label{eq:FOMPDE}
	\dot{\Sol} \left( t
	\right) =
	\F
	\left(\Sol \left( t \right) \right), \quad
	\Sol(0) = \InitialCondition, \quad
	t \in (0,\TimeEnd),
\end{align}
where the operator 
$\F$ is defined on a dense subspace
$\Y\subseteq \X$,
\begin{align*} 
	\F\colon
	\Y \to \X,
	\quad y \mapsto \F(y).
\end{align*}
We assume $\InitialCondition\in\Y$ and call $\Sol$ a solution of \eqref{eq:FOMPDE} if $\Sol(t)\in\Y$ for all $t\geq 0$, $\Sol$ is continuous in $[0,\TimeEnd)$, differentiable in $(0,\TimeEnd)$, and \eqref{eq:FOMPDE} is satisfied. 

Standard projection-based \gls{MOR} is based on the identification of a suitable $\dimROM$-dimensional subspace $\mathcal{Y}\subseteq\Y$, conveniently described by a basis $\left(\Mode_1,\ldots,\Mode_{\dimROM}\right)$, and the assumption that the solution $\Sol$ of \eqref{eq:FOMPDE} is well-approximated within this space, i.e., there exist scalar functions $\Coefficient_i$ such that
\begin{equation}
	\label{eq:GalerkinAnsatz}
	\Sol(t) \approx \SolROMlifted(t) \vcentcolon= \sum_{i=1}^\dimROM \Coefficient_i(t)\Mode_{i} \qquad\text{for }t \in [0,\TimeEnd).
\end{equation}
Substituting $\SolROMlifted$ into \eqref{eq:FOMPDE} and requiring that the residual is orthogonal to the approximation space~$\mathcal{Y}$ results in the \gls{ROM}
\begin{equation}
	\label{eq:ROM}
	\sum_{j=1}^\dimROM \dot{\Coeff}_{j} \ipX{ \Mode_i,\Mode_j} = \ipX{ \Mode_i,\F\left(\sum_{j=1}^\dimROM \Coefficient_j \Mode_j \right)},\qquad i=1,\ldots,\dimROM.
\end{equation}
Note that the matrix $\PODmassMatrix\vcentcolon=[\ipX{ \Mode_i,\Mode_j}]_{i,j=1}^{\dimROM}$ is symmetric and nonsingular such that under reasonable assumptions on~$\F$ the \gls{ROM}~\eqref{eq:ROM} is uniquely solvable for any initial condition. For numerical reasons, the basis $\left(\Mode_1,\ldots,\Mode_{\dimROM}\right)$ is usually chosen to be orthonormal such that $\PODmassMatrix$ equals the $\dimROM$-dimensional identity matrix. The question that remains to be answered is how to choose the subspace $\mathcal{Y}$. The particular choice of this subspace is one of the distinguishing features of the different \gls{MOR} methods.

The best subspace of a given dimension --- in the sense that the worst-case approximation error is minimized --- is described by the Kolmogorov $n$-widths \cite{Kol36,Pin85} (or equivalently as shown in \cite{UngG19} by the Hankel singular values). 
For specific equations, it can be shown (cf.~\cite{MadPT02,MadPT02b,QuaMN16}) that the $n$-widths decay exponentially, enabling \gls{MOR} to succeed.
Unfortunately, there are also several problem classes, where the $n$-widths do not decay exponentially. 
This is typically the case if the dynamical system features strong convection or transport of a quantity within the spatial domain. 
For examples we refer to \cite{CagMS19} and \cite{GreU19}. 
For such problems, standard \gls{MOR} methods cannot produce a low-dimensional model that, at the same time, is very accurate. 
Several authors have observed this, yielding the emerging field of so-called \gls{MOR} for transport dominated phenomena. We give a detailed overview of the current state of the art in \cref{sec:literatureReview}.

The slow decay of the Kolmogorov $n$-widths is often related to sharp gradients of the solution that travel through the physical domain. 
The large gradients might develop over time (for instance, in Burgers' equation with large Reynolds number) or are enforced within the system by the initial condition. The \gls{MFEM} \cite{MilM81,GelDM81} accounts for these typically local effects by moving the discretization nodes in an automated fashion to the critical areas. Therefore the basis functions are transformed based on the current position of the respective nodes, and the node positions are considered as additional unknowns. The equation for the node position is derived from the necessary condition for minimizing the squared norm of the residual. In this paper we transfer this idea to the context of model order reduction by replacing the approximation \eqref{eq:GalerkinAnsatz} with
\begin{equation}
	\label{eq:multiframeAnsatz}
	\Sol(t) \approx \SolROMlifted(t) \vcentcolon=\sum_{i=1}^{\dimROM} \Coeff_{i}(t) \Transformation_i \left(\Path_i(t)\right)\Mode_i\qquad\text{for }t \in [0,\TimeEnd)
\end{equation}
with suitable transformation operators $\Transformation_i\colon \PathSpace_i \to\mathcal{B}(\X)$ for $i = 1,\ldots,\dimROM$ that are $\Y$-invariant, i.e., $\Transformation_i(\Path)\Y\subseteq\Y$ for all $\Path\in\PathSpace_i$. 
Here, $\mathcal{B}(\X)$ denotes the space of all linear and bounded operators on $\X$. For given paths $\Path_i$ the adaptive low-dimensional approximation manifold is thus given as the linear span of $\{\Transformation_i \left(\Path_i(t)\right)\Mode_i\}_{i=1}^r$. 
Since our main goal is to use the approximation~\eqref{eq:multiframeAnsatz} to derive a \gls{ROM}, we have to ensure that we can differentiate with respect to time. 
In particular, we have to choose suitable sets $\PathSpace_i$. 
Natural choices are to assume that the $\PathSpace_i$'s are Lie groups or real Banach spaces. 
For the main applications we have in mind, it suffices to choose $\PathSpace_i = \mathbb{R}^{\dimPathSpace_i}$ and for the ease of presentation we do not consider the more general cases described before.
We immediately arise at three questions:
\begin{compactenum}
	\item What are suitable families of transformation operators $\Transformation_i$?
	\item For given transformations $\Transformation_i$, how to determine $\Coeff_i$, $\Mode_i$, and $\Path_i$ such that the approximation error in \eqref{eq:multiframeAnsatz} is minimized?
	\item How to use \eqref{eq:multiframeAnsatz} in a projection framework such that the resulting \gls{ROM} can be evaluated efficiently?
\end{compactenum}

In practice, we often have knowledge about the solution behavior of the equations at hand and can thus design appropriate transformation operators. 
For instance, for wave-like phenomena, we can use translation operators that shift the basis functions or modes $\Mode_i$ in the spatial domain (see the forthcoming \Cref{ex:advectionEquationSPOD} for the advection equation and \Cref{rem:equivarianceSemigroup} for a more abstract discussion). 
Therefore, we assume for the remainder of the manuscript that the transformations $\Transformation_i$ are given such that we only have to deal with the second and the third question, which we address in \cref{sec:shiftedPOD} and \cref{sec:onlinePhase}, respectively. 
More precisely, we extend in \cref{sec:shiftedPOD} the residual minimization version of the \gls{sPOD} as presented in \cite{SchRM19} to the infinite-dimensional setting and, thereby, answer parts of the second question. 
The resulting adaptive basis is then used to obtain a \gls{ROM} via Galerkin projection. In most applications, the paths $\Path_i$ are unknown and have to be computed during the online phase along with the coefficient functions $\Coeff_i$, which renders the approximation \eqref{eq:multiframeAnsatz} nonlinear. 
Consequently, the \gls{ROM} that is obtained via Galerkin projection of \eqref{eq:FOMPDE} is an underdetermined nonlinear \gls{DAE} that needs to be completed with additional equations. This is discussed in detail in \cref{sec:onlinePhase}. Our main contributions are the following:
\begin{compactenum}
	\item In \cref{sec:literatureReview}, we give a detailed overview of the literature on model reduction methods with special emphasis on transport dominated problems or problems with slowly decaying Kolmogorov $n$-widths.
	\item We show (cf.\ \Cref{thm:sPODminimization}) that for given paths $\Path_i$, the optimization problem that minimizes the residual in \eqref{eq:multiframeAnsatz} has a solution. In the special case that all modes are transformed with the same operator and the same path we establish in \Cref{thm:sPODminimization:singleTransformation} that the minimization problem is equivalent to standard \gls{POD} (see~\cref{sec:POD}) with transformed data.
	\item Based on the approximation ansatz \eqref{eq:multiframeAnsatz}, we construct a \gls{ROM} that minimizes the residual, see \Cref{thm:continuousOptimality}, and thus extend the idea of the \gls{MFEM} to the model reduction context. The \gls{ROM} is certified in the sense that we provide an a-posteriori residual-based error bound in \Cref{thm:errorBound}. 
	\item For the special case that all modes are transformed with the same operator and the same path, we discuss in \cref{sec:symmetryReduction} the close connection to the symmetry reduction framework \cite{RowM00,BeyT04}. In particular, \Cref{thm:phaseCondition} details the connection of the phase condition (cf.~\eqref{eq:phaseConditions}) that is obtained from minimizing the residual and the phase condition that minimizes the temporal change of the reduced state.
\end{compactenum}

We close our discussion with three numerical examples. 
In the case of the advection-diffusion equation, which is discussed in \cref{subsec:advectionDiffusion}, our \gls{ROM} is capable of yielding a good approximation of the dynamics with only two modes. 
In addition, this example numerically justifies the approach to use the same transformation for several modes, as proposed in the forthcoming equation \eqref{eq:multframeAnsatzGroup}. This idea is not only advantageous for the actual implementation, but also for the theory; see \Cref{thm:wellPosedROM} and \Cref{rem:issueMultiframeApproach}.
In \cref{subsec:waveEquation}, we consider the linear wave equation and obtain an excellent agreement with the \gls{FOM} by using an approximation with only two modes.
Furthermore, we observe that the \gls{ROM} allows choosing significantly larger time steps than the full order model, which is an advantage of performing the model reduction on the time-continuous level.
Finally, we consider the viscous Burgers' equation in \cref{subsec:burgers} as an example of a nonlinear full order model.
In this context, we discuss how the evaluation of the reduced order model can be rendered independent from the full order dimension.
Furthermore, our approach outperforms the classical \gls{POD}-Galerkin approach and is able to yield a decent approximation of the dynamics with just seven modes.

\begin{remark}
	In \cite{LeeC19} the authors use a general nonlinear approximation of the form $\Sol(t) \approx g(\Coefficient_1(t),\ldots,\Coefficient_{\dimROM}(t))$, which certainly generalizes our approximation approach presented in \eqref{eq:multiframeAnsatz}. We believe however that our parametrization of the general nonlinear approximation offers some advantages. 
	First, in the general setting with a possibly infinite-dimensional Banach space $\X$ it seems rather difficult to prescribe a suitable nonlinear operator $g\colon\mathbb{R}^{\dimROM}\to\X$, even within a deep-learning framework as proposed in \cite{LeeC19}. 
	Working in the infinite-dimensional setting often allows using equivariance (see~\cref{sec:symmetryReduction} for further details), which may not be possible anymore after semi-discretization in space. 
	Moreover, the separation of amplitudes and transformed modes as in \eqref{eq:multiframeAnsatz}, gives a direct physical interpretation of the \gls{ROM}, where the transformed mode $\Transformation_i(\Path_i(t))\Mode_i$ may be related to a reference frame. 
	For additional numerical and theoretical benefits, this may be exploited even further by prescribing the same transformation operator with the same path for different modes $\Mode_i$, i.e., to replace \eqref{eq:multiframeAnsatz} with
	\begin{equation}
		\label{eq:multframeAnsatzGroup}
		\Sol(t) \approx \sum_{i=1}^{\hat{\dimROM}} \sum_{j=1}^{\dimROM_i} \Coefficient_{i,j}(t)\Transformation_i(\Path_i(t))\Mode_{i,j}.
	\end{equation}
	On the one hand, this approach reduces the computational cost required to solve the \gls{ROM}, since less path variables need to be computed and, on the other hand, provides more flexibility for theoretical investigations, see the forthcoming \Cref{rem:issueMultiframeApproach}.
\end{remark}

\section{State of the Art}
\label{sec:literatureReview}

As mentioned in the introduction, classical model order reduction techniques are usually not able to provide accurate low-dimensional models when applied to systems exhibiting the transport of structures with large gradients, such as shocks.
In the past years, there has been an increasing effort in the model reduction community to develop new methods that can solve this problem. 
The most relevant approaches can be roughly subdivided into three classes.
The first class of methods aims at describing the transport by appropriately chosen time-dependent coordinate transformations, which we hence refer to as \emph{reference frame methods}.
For instance, in the case of a simple advection problem, an apparent coordinate transformation is given by a time-dependent translation that describes the advective behavior. 
The second class incorporates a more general time-dependent update of the basis functions, which is not necessarily described by a coordinate transformation.
This update of the basis functions may be, e.g., realized by an $h$-refinement-like enrichment of the basis functions or by an equation describing the evolution of the basis functions.
The third class considers general nonlinear approximation ansatzes accompanied by a nonlinear Petrov--Galerkin projection to construct the \gls{ROM}.
The corresponding nonlinear mappings are usually constructed via techniques from machine learning.

\paragraph{Reference frame methods} A common approach among the methods using coordinate transformations consists of formulating the \gls{FOM} in a new coordinate system, which we refer to as the \emph{reference frame}.
The goal is to choose the coordinate system in such a way that the Kolmogorov $n$-widths for the transformed problem decay fast and thus enable standard \gls{MOR} methods to succeed.
The first developments in this direction are presented within the symmetry reduction framework, cf.~\cite{BeyT04,RowM00,RowKML03}.
The main idea is to approximate the solution by a composition of the action of a time-dependent group element and a so-called \emph{frozen solution}, cf.~\cref{sec:symmetryReduction}.
Ideally, the group action and the time-dependent element are chosen such that the frozen solution is almost constant over time, which supports a low-dimensional approximation.
The group action can, for example, be chosen based on physical considerations or from snapshot data of the full order solution.
For instance, in \cite{SonSK13}, the authors present a method to pre-process snapshot data to align them in such a way that the singular values decay fast and, thus, a low-dimensional description of the dynamics can be obtained.
To this end, they assume the snapshots to be almost identical up to the action of some underlying symmetry group.
Here, the coordinate transformation is described by the group action, and the task is to assign group elements to the snapshots optimally.
The proposed solution for this problem involves the recently introduced eigenvector alignment method \cite{Sin11}, which, in comparison to other methods, has the advantage that no template snapshot needs to be chosen.
Similar ideas to pre-process the snapshot data by some kind of alignment or calibration have been applied, for instance, in \cite{NonBRM19,SesS19}.
A different approach for determining a low-dimensional description of given snapshot data was presented in \cite{IolL14} and is based on solving optimal mass transfer problems.

The contributions mentioned so far mainly consider the task to find a low-dimensional approximation of snapshot data, but they do not construct a dynamic \gls{ROM} based on the low-dimensional description.
An approach that considers both tasks is, for instance, presented in \cite{MowS18}.
At the first step, the authors apply a symmetry reduction tool, called the method of slices \cite{RowM00}, to formulate the \gls{FOM} in the reference frame.
Afterward, they apply a standard model reduction scheme to the evolution equation in the reference frame.
Similar approaches based on the idea to formulate the \gls{FOM} in a different coordinate system and apply standard model reduction schemes afterward are discussed, e.g., in \cite{MojB17, OhlR13, RowM00, Tad19}.

Instead of first transforming the full order model and then reducing the transformed system, the authors in \cite{CagMS19} directly reduce the untransformed \gls{FOM} using an approximation ansatz that includes a coordinate transformation.
Their approach for the compression of snapshot data is similar to the ones mentioned above. However, they also present a method for constructing a \gls{ROM} based on the identified low-dimensional description of the snapshot data.
To this end, the authors first discretize in time and then substitute the approximation ansatz into the semi-discrete full order model.
They present an algorithm for updating the time-discrete states of the \gls{ROM} by minimizing the time-discrete residual.
The evaluation of the \gls{ROM} still scales with the full dimension, but they present an additional approximation of the \gls{ROM}, which allows achieving an efficient offline/online decomposition.
Further methods which are based on the idea to enhance the approximation ansatz via a coordinate transformation are discussed in \cite{CagCMA17, GlaMM98, KarBR19, NaiB18,RimPM20}.

All methods mentioned so far have in common that they apply a single coordinate transform in order to get a faster decay of the singular values or the Kolmogorov $n$-widths.
However, one coordinate transformation may not be sufficient to obtain a low-dimensional description in the case of multiple transport velocities in the system.
For instance, the analytical solution of the linear wave equation can be represented by two traveling waves, which cannot be described by two modes and a single transformation.
So far, there are only a few methods that consider the case of multiple transport velocities and treat them with different coordinate transformations.
For example, the \gls{sPOD} \cite{Rei18,ReiSSM18,SchRM19} considers an approximation ansatz, which allows to have different coordinate transformations for different sets of basis functions.
To compute a low-dimensional approximation of given snapshot data, the authors in \cite{ReiSSM18} propose a heuristic method that iteratively transforms the snapshots into the different reference frames and compresses the data in the respective frame using a singular value decomposition.
In \cite{Rei18}, the dominant modes are identified by solving an optimization problem that maximizes the leading singular values in each reference frame.
Also, in \cite{SchRM19}, a \gls{sPOD} approximation is computed by solving an optimization problem, which aims at minimizing the deviation of the approximation from the original snapshot data.
An alternative approach is considered in \cite{RimML18}, where the authors also present a method for identifying modes in several reference frames similar to the \gls{sPOD}.
However, instead of applying an iterative or optimization procedure, they extract the dominant modes one after another in a greedy fashion.
Thus, this method has a lower computational complexity in comparison to the \gls{sPOD}. However, it fails to give the minimum number of modes even for examples like the linear wave equation where the analytic solution is known.
Both the \gls{sPOD} and the transport reversal presented in \cite{RimML18} are methods for the identification of dominant modes only, and there is no dynamic \gls{ROM} constructed based on the determined modes.

In this paper, we close this gap by introducing a framework that allows constructing \glspl{ROM} based on transformed modes, which can, for instance, be computed by one of the contributions mentioned above.
Notably, our framework is not restricted to the case that all modes are transformed by the same coordinate transformation, but we consider the general ansatz \eqref{eq:multiframeAnsatz}, which allows incorporating several transformation operators.

\paragraph{Adaptive basis methods} 
The second class of \gls{MOR} methods for transport dominated systems is based on an online adaptation of the modes during the simulation of the \gls{ROM}.
For instance, in \cite{DihDH11}, the authors combine the reduced basis method with a segmentation of the time interval, i.e., they subdivide the total time interval of interest into subintervals and then compute a reduced basis on each of these subintervals.
The authors of \cite{DihDH11} propose an adaptive segmentation such that a prescribed error tolerance and a maximum number of basis functions per time interval is not exceeded.
Depending on the time interval segmentation, this may lead to a considerable reduction of the number of required basis functions in each time interval in comparison to a non-segmented approach.
In the online phase, they simulate the \gls{ROM} subsequently on each subinterval.
When the interface between two adjacent subintervals is reached, the initial value for the new subinterval is computed by an orthogonal projection of the current approximation of the full order state onto the span of the basis functions of the new subinterval.
Thus, the \gls{ROM} is a switched system, where the switching condition depends solely on time.
Furthermore, they provide an a-posteriori error estimator, which they use to drive a greedy parameter sampling and the adaptive time partitioning.

Differently from \cite{DihDH11}, the authors in \cite{EttC19} propose a scheme that adapts the reduced basis only in the online phase and only if the error estimator returns values which are too high or too low.
In the offline phase, they define a tree structure that represents hierarchical orthogonal decompositions of the underlying vector space of the \gls{FOM}.
Based on this tree structure, they are able to adaptively refine the basis if the error of the ROM is too high or to compress the basis if the error is smaller than a prescribed error threshold.
This refinement or compression of the basis corresponds to moving downwards or upwards in the tree structure, respectively.
As in \cite{DihDH11}, the obtained \gls{ROM} is a switched system, but in contrast to \cite{DihDH11}, the switching condition is state-dependent since the error estimators are based on the current state of the \gls{ROM}.
Thus, the switching times are a priori not known.
The work in \cite{EttC19} is based on the ideas presented in \cite{Car15} and extends them towards more general refinement trees and a more general and efficient basis compression scheme.
Another online-adaptive scheme is proposed in \cite{Peh18}, where the basis functions are regularly modified via a low-rank update, where the number of basis functions remains constant in contrast to \cite{EttC19}.
\begin{remark}
	Model reduction for switched systems is an active research area for itself, with several contributions within the last years. For an overview we mention \cite{BasPWL16,SchU18,GosPAF18,ScaA16a} and the references therein.
\end{remark}
In \cite{GerL14}, the authors present an approach that is also based on a reduced basis, which is adapted as time evolves.
However, in contrast to the works mentioned above, the rules for updating the basis are more problem-specific.
In concreto, they propose to use the eigenfunctions of a linear Schrödinger operator associated with the initial value of the \gls{FOM} as basis functions.
Then, the time evolution of the basis functions is performed in such a way that the basis functions remain eigenfunctions of a linear Schrödinger operator associated with the time-dependent \gls{FOM}.
Consequently, they obtain an additional evolution equation for the basis functions.
In contrast to the works mentioned above, the \gls{ROM} is not a switched system with time-discrete changes in the basis functions, but instead they obtain a time-continuous equation for the evolution of the basis functions.

In the last paragraphs, we have discussed model reduction methods, which are based on time-dependent basis functions, while there are different ways of establishing this time dependency.
Formally, also the reference frame methods fit into this rather general class of approaches since applying time-dependent coordinate transformations is a particular way of inducing a time dependency into the modes.
Nevertheless, we distinguish between these two classes since the main ideas are somewhat different.
The reference frame methods are motivated by physics and incorporate coordinate transformations to model the advection behavior present in the dynamics.
On the other hand, the adaptive basis methods propose rather general ways of updating the basis functions without attempting to model the advection explicitly.

\paragraph{Nonlinear projection methods}
The idea of the third class of methods is to approximate the solution via a general nonlinear approximation ansatz.
In \cite{LeeC19}, the authors use an autoencoder, which is a type of artificial neural network, to obtain a low-dimensional description of the \gls{FOM} solution.
Based on the snapshot data, a decoder and an encoder mapping are learned, where the decoder is a mapping from the reduced state space to the full state space and the encoder vice versa.
Especially, the lifting of the reduced state to an approximation of the full order state is performed by the decoder mapping, which thus describes the approximation ansatz.
The projection of the \gls{FOM} is carried out by substituting the approximation ansatz into the \gls{FOM} and then constructing the \gls{ROM} via minimization of the residual.
They propose two different approaches: One is based on minimizing the residual for the time-continuous \gls{FOM} while the other one considers the \gls{FOM} in the time-discrete setting.
The idea to use autoencoders for the purpose of model order reduction has been previously presented in \cite{Kas16} and \cite{HarM17}.

\begin{remark}
    For several hyperbolic problems, it is possible to derive an equivalent system of delay equations, see, for instance, \cite{Bra67, CooK68, Lop76}. Although these delay equations are still infinite-dimensional, they are -- from a computational perspective -- much easier to solve. From a model reduction perspective, this provides a different approach to the approximation of transport dominated phenomena by searching for a surrogate model that captures the transport by including a time-delay. Some first results in this direction are obtained in \cite{ScaA16,SchU16,SchUBG18,FosSU19-ppt,PonPS18}.
\end{remark}

\section{Notation and Preliminaries}

Throughout the paper we denote the Bochner space of  square integrable functions and locally square integrable functions in the time interval $[0,T]$ with values in a Banach space $\X$ by $\LTwo{0,T;\X}$ and $\LocSquareIntegrableFunctions{-\infty,\infty;\X}$, respectively. If $\X=\mathbb{R}$ we simply write $\LTwo{0,T}$. The Sobolov space of square integrable functions in a domain $\Omega \subseteq\mathbb{R}^d$ with square integrable derivative is denoted by $H^1(\Omega)$. If additionally periodic boundary conditions are assumed, we write $H^1_{\mathrm{per}}(\Omega)$. For an operator $\mathcal{A}$ we denote its adjoint operator by $\mathcal{A}^\star$. The induced operator norm of $\mathcal{A}$ is denoted by $\operatorNorm{\mathcal{A}}$ and the standard Euclidean norm on $\mathbb{R}^n$ is written as $\|\cdot\|_2$. The transpose of a matrix $M = [m_{ij}]\in\mathbb{R}^{n\times m}$ is denoted by $M^T$ and $\delta_{ij}$ is the Kronecker delta.

\subsection{Proper Orthogonal Decomposition}
\label{sec:POD}

To motivate what follows, we consider one classical way of constructing the \gls{MOR} projection basis $(\Mode_1,\ldots,\Mode_{\dimROM})$, namely \gls{POD}, see for instance \cite{HolLBR12,GubV17} and the references therein. Our starting point for the discussion is the approximation \eqref{eq:GalerkinAnsatz}. 
Suppose that we have access to a solution trajectory $\Sol$ of \eqref{eq:FOMPDE}, for instance via a numerical simulation. 
For a given dimension $\dimROM<\dim(\X)$, the approximation error in \eqref{eq:GalerkinAnsatz} is minimized via the optimization problem
\begin{equation}
	\label{eq:PODminimization}
	\left\{\begin{aligned}
		&\min \frac{1}{2}\int_0^\TimeEnd \left\lVert \Sol(t)-\sum_{j=1}^{\dimROM} \Coeff_{j}(t) \Mode_{j}\right\rVert_\X^2 \dt\\
		&\text{\st}\ \{\Mode_{j}\}_{j=1}^{\dimROM} \subseteq \Y \text{ and } \ipX{ \Mode_i, \Mode_j} = \delta_{ij},\; i,j=1,\ldots,\dimROM.
	\end{aligned}\right.
\end{equation}
For later referencing, we call the solution $\{\Mode_{j}\}_{j=1}^\dimROM$ of \eqref{eq:PODminimization} the \emph{dominant modes} for the solution trajectory $\Sol$. 
Notice that the optimization problem \eqref{eq:PODminimization} depends on the coefficient functions $\Coeff_{j}$ for $j=1,\ldots,\dimROM$ and in principle one should also minimize over the $\Coeff_{j}$'s. 
It is however well-known that the best approximation in a subspace is given by the orthogonal projection onto this subspace and hence we have $\Coeff_{j}(t) = \ipX{ \Sol(t), \Mode_{j}}$ for $j = 1,\ldots,\dimROM$.
Thus, we are interested in the optimization problem
\begin{equation}
	\label{eq:PODminimization2}
	\left\{\begin{aligned}
		&\min \frac{1}{2}\int_0^\TimeEnd \left\lVert \Sol(t)-\sum_{j=1}^{\dimROM} \ipX{ \Sol(t),\Mode_{j}} \Mode_{j}\right\rVert_\X^2 \dt\\
		&\text{\st}\ \{\Mode_{j}\}_{j=1}^\dimROM\subseteq \Y \text{ and } \ipX{ \Mode_i, \Mode_j} = \delta_{ij},\; i,j=1,\ldots,\dimROM.
	\end{aligned}\right.
\end{equation}

Following \cite{GubV17}, the optimization problem can be solved by computing the eigenvalues of the nonnegative, self-adjoint compact operator
\begin{equation}
	\label{eq:PODoperator}
	\PODoperator\colon \X\to\X,\qquad \PODoperator\Mode = \int_0^\TimeEnd \ipX{ \Sol(t),\Mode} \Sol(t) \dt.
\end{equation}
Observe that $\Sol\in \LTwo{0,\TimeEnd;\Y}$ implies (see also \cite[Lemma~1.24]{GubV17})  $\PODoperator\Mode\in \Y$ for all~$\Mode\in\X$. In this case, a solution of \eqref{eq:PODminimization2} can be obtained as follows.

\begin{theorem}[{\cite[Theorem~1.15]{GubV17}}]
	\label{thm:PODminimization}
	Let $\X$ be a separable real Hilbert space and suppose that $\Sol\in\LTwo{0,\TimeEnd;\X}$ is given. Then there exist nonnegative eigenvalues $\Eigenvalue_i$ and associated orthonormal eigenvectors $\Mode_i\in\X$ for $i\in\mathcal{I}$ with
	\begin{displaymath}
		\mathcal{I} \vcentcolon= \begin{cases}
			\{1,\ldots,\dim(\X)\}, & \text{if $\dim(\X)<\infty$},\\
			\mathbb{N}, & \text{otherwise},
		\end{cases}
	\end{displaymath}		
	 satisfying
	\begin{align}
		\label{eq:PODeigenvalueProblem}
		&&\PODoperator\Mode_i &= \Eigenvalue_i\Mode_i &&\text{for $i\in\mathcal{I}$},\\
		\text{and}&& \Eigenvalue_{i} \geq \Eigenvalue_{i+1} &\geq 0 &&\text{for $i<\dim(\X)$},
	\end{align}
	with $\PODoperator$ as defined in \eqref{eq:PODoperator}. 
	If in addition $\Sol\in\LTwo{0,\TimeEnd;\Y}$ with a dense subspace $\Y\subseteq\X$, then for any $\dimROM\in\mathcal{I}$ with $\lambda_{\dimROM} > 0$, the set of the $\dimROM$ leading eigenvectors $\{\Mode_i\}_{i=1}^{\dimROM}$ is a solution of \eqref{eq:PODminimization2}.
\end{theorem}

\begin{remark}
	\label{rem:uniquePODbasis}
	One can show (see for instance \cite{Vol01}) that \eqref{eq:PODeigenvalueProblem} equals the first-order necessary optimality condition for the minimization problem \eqref{eq:PODminimization2}. 
	In particular, the minimizer of \eqref{eq:PODminimization2} is unique if, and only if, the first $\dimROM+1$ eigenvalues of $\PODoperator$ are simple, that is $\Eigenvalue_1 > \ldots > \Eigenvalue_\dimROM > \Eigenvalue_{\dimROM+1}$.
\end{remark}

To illustrate \Cref{thm:PODminimization} we consider the following simple example with an advection equation, see also \cite{Ung13,HolLBR12}.
\begin{example}
	\label{ex:advectionEquationPOD}
	The one-dimensional linear advection equation with constant coefficients and periodic boundary condition is given by
	\begin{equation}
		\label{eq:advectionEquation}
		\left\{\begin{aligned}
			\ParDer{t}{\Sol}(t,\xi) + \ParDer{\xi}{\Sol}(t,\xi) &= 0, & (t,\xi)\in(0,1)\times(0,1),\\
			\Sol(t,0) &= \Sol(t,1), & t\in(0,1),\\
			\Sol(0,\xi) &= \InitialCondition(\xi), & \xi\in(0,1),
		\end{aligned}\right.
	\end{equation}
	with given initial value 
	\begin{displaymath}
		\InitialCondition\in \Y\vcentcolon= H^1_{\mathrm{per}}(0,1).
	\end{displaymath}
	For notational convenience, we consider $\InitialCondition$ as an element of $\LocSquareIntegrableFunctions{\mathbb{R}}$ via periodic continuation. 
	It is well-known that the solution of \eqref{eq:advectionEquation} is given by $\Sol(t,\xi) = \InitialCondition(\xi-t)$ for $(t,\xi)\in(0,1)\times(0,1)$. 
	To compute the \gls{POD} basis as in \Cref{thm:PODminimization} we set $\X = \LTwo{0,1}$ and observe that the operator in \eqref{eq:PODoperator} is given by
	\begin{displaymath}
		(\PODoperator\Mode)(\zeta) = \int_0^1 \AutoCorrelationFunction(\zeta,\xi)\Mode(\xi) \,\mathrm{d}\xi
	\end{displaymath}
	with auto-correlation function $\AutoCorrelationFunction(\zeta,\xi) = \int_0^1 \Sol(t,\zeta)\Sol(t,\xi)\dt = \int_0^1 \InitialCondition(t)\InitialCondition(t + (\zeta-\xi))\dt$. In particular, $\AutoCorrelationFunction(\zeta,\xi)$ depends only on the distance between $\zeta$ and $\xi$. Since $\AutoCorrelationFunction$ is periodic in the first argument, and thus also in the distance $\xi-\zeta$, we can consider its representation in the Fourier basis, i.e.,
	\begin{displaymath}
		\AutoCorrelationFunction(\zeta,\xi) = \sum_{k=-\infty}^\infty c_k \exp\left(2\pi\imath k (\xi-\zeta)\right)\qquad\text{with}\ c_k\in\mathbb{R}.
	\end{displaymath}
	Using this expression, we observe that the eigenvectors of $\PODoperator$, which correspond to the solution of the minimization problem \eqref{eq:PODminimization2} for the advection equation \eqref{eq:advectionEquation}, are given by the functions $\Mode_i(\xi) = \exp(2\pi\imath i\xi)$. 
	Note that the eigenvectors of $\PODoperator$ are independent of the initial value, which completely describes the solution of \eqref{eq:advectionEquation} and the initial value only influences the ordering of the dominant modes.
\end{example}

In practice, the solution of the \gls{FOM} \eqref{eq:FOMPDE} usually depends on an additional variable $\param$, i.e., $\Sol(t) = \Sol(t;\param)$, which may represent physical or geometry parameters or a control function. In any case, the dominant modes should reflect the dynamics for a large range of the additional variable $\param$, which we hereafter refer to as \emph{parameter}. 
The current state of the art is to sample the parameter space, i.e., to pick parameters $\param_j$ for $j=1,\ldots,M$ and solve \eqref{eq:FOMPDE} for each parameter value $\param_j$. 
The dominant modes can then be computed by solving \eqref{eq:PODminimization2} simultaneously for all parameters \cite{GubV17}, which is equivalent to concatenating the solution trajectories for different parameters and solve \eqref{eq:PODminimization2} based on the concatenated solution.
An alternative approach is to determine dominant modes for each parameter and then combine the different dominant modes. 
In general, the first approach provides a smaller set of dominant modes, while the second approach allows to pick the parameters iteratively, for instance by a greedy selection procedure. 
Let us mention that convergence rates for the greedy approach can be related to the decay of the Kolmogorov $n$-widths, see for instance \cite{Haa13,UngG19}.

\subsection{Galerkin Projection and Offline/Online Decomposition}\label{sec:offlineOnlineDecomposition}
Having identified a set of dominant modes $\{\Mode_{j}\}_{j=1}^{\dimROM}$ we substitute the Galerkin ansatz \eqref{eq:GalerkinAnsatz} into \eqref{eq:FOMPDE} and obtain at time $t>0$ the residual
\begin{displaymath}
	\sum_{i=1}^\dimROM \dot{\Coefficient}_i(t)\Mode_i - \F\left(\sum_{i=1}^\dimROM \Coefficient_i(t)\Mode_i\right).
\end{displaymath}
Note that the choice of the modes $\{\Mode_{j}\}_{j=1}^{\dimROM}$ fixes the initial condition and thus the coefficient functions $\Coefficient_i$ at time $t=0$. 
Thus, if we want to minimize the norm of the residual we can do so only by optimizing over the slope of the coefficient functions. 
The concept of minimizing the norm of the residual with respect to the slope of the coefficient functions is called \emph{continuous optimality} in \cite{CarBA17}. 
We deduce that the partial derivative with respect to $\dot{\Coefficient}_\ell$ is given by
\begin{equation*}
	\frac{\partial}{\partial \dot{\Coefficient}_\ell} \left\|\sum_{i=1}^\dimROM \dot{\Coefficient}_i\Mode_i - \F\left(\sum_{i=1}^\dimROM \Coefficient_i\Mode_i\right)\right\|_\X^2 = 2\sum_{i=1}^r \dot{\Coefficient}_i \ipX{ \Mode_\ell,\Mode_i} - 2\ipX{ \Mode_\ell, \F\left(\sum_{i=1}^\dimROM \Coefficient_i\Mode_i\right)}.
\end{equation*}
As a consequence the \gls{ROM} \eqref{eq:ROM}, or equivalently
\begin{equation}
	\label{eq:ROM2}
	\PODmassMatrix\dot{\SolROM}(t) = \FROM(\SolROM(t)),\quad 
	\SolROM(0) = \InitialConditionROM, \quad t\in(0,\TimeEnd)
\end{equation}
with $\PODmassMatrix\vcentcolon=[\ipX{ \Mode_i,\Mode_j}]_{i,j=1}^\dimROM\in\mathbb{R}^{\dimROM\times\dimROM}$, reduced state $\SolROM(t)\in\mathbb{R}^\dimROM$, $\FROM\colon\mathbb{R}^{\dimROM}\to\mathbb{R}^{\dimROM}$, and initial value $\InitialConditionROM\in\mathbb{R}^{\dimROM}$ defined as
\begin{displaymath}
	\SolROM(t) \vcentcolon= \begin{bmatrix}
		\Coeff_1(t)\\\vdots\\\Coeff_{\dimROM}(t)
	\end{bmatrix},\qquad
	\FROM(\SolROM)\vcentcolon= \begin{bmatrix}
		\ipX{ \Mode_1, \F\left(\sum_{i=1}^\dimROM \Coefficient_i \Mode_{i}\right)}\\
		\vdots\\
		\ipX{ \Mode_\dimROM, \F\left(\sum_{i=1}^\dimROM \Coefficient_i \Mode_{i}\right)}
	\end{bmatrix},\qquad\text{and}\qquad
	\InitialConditionROM \vcentcolon= \begin{bmatrix}
		\ipX{ \Mode_1, \InitialCondition}\\
		\vdots\\
		\ipX{ \Mode_{\dimROM}, \InitialCondition}
	\end{bmatrix},
\end{displaymath}
satisfies the following result, see also \cite{CarBA17}.

\begin{lemma}[Continuous optimality]
	\label{lem:continuousOptimality}
	The \gls{ROM} \eqref{eq:ROM2} is continuously optimal in the sense that if $\SolROM$ is a solution of \eqref{eq:ROM2}, then for each $t>0$, the velocity $\dot{\SolROM}(t)$ is the unique minimizer of 
	\begin{displaymath}
		\min_{\alpha=[\alpha_i]_{i=1}^r} \left\|\sum_{i=1}^r \alpha_i\Mode_i - \F\left(\sum_{i=1}^r \Coefficient_i(t)\Mode_i\right)\right\|_\X.
	\end{displaymath}
\end{lemma}

Although the \gls{ROM} is formally stated in $\mathbb{R}^{\dimROM}$, it still depends on the evaluation of $\F$ in the original space $\Y$ and thus it might still be computationally intractable to solve \eqref{eq:ROM2} efficiently. However, in many cases we can precompute all quantities that depend on $\Y$, which allows a fast evaluation of \eqref{eq:ROM2}. This process is called \emph{efficient offline/online decomposition} in the \gls{MOR} literature. For instance for a linear operator $\mathcal{A}\colon\Y\to\X$ we have
\begin{displaymath}
	\begin{bmatrix}
		\ipX{ \Mode_1,\mathcal{A}(\sum_{i=1}^\dimROM \Coefficient_{i}\Mode_{i})}\\
		\vdots\\
		\ipX{ \Mode_{\dimROM}, \mathcal{A}(\sum_{i=1}^\dimROM \Coefficient_{i}\Mode_{i})}\\
	\end{bmatrix} = \begin{bmatrix}
		\ipX{ \Mode_1,\mathcal{A}\Mode_1} & \dots & \ipX{ \Mode_1, \mathcal{A}\Mode_{\dimROM}}\\
		\vdots & & \vdots\\
		\ipX{ \Mode_{\dimROM}, \mathcal{A}\Mode_1 } & \dots & \ipX{ \Mode_{\dimROM}, \mathcal{A}\Mode_{\dimROM}}\\
	\end{bmatrix}\begin{bmatrix}
		\Coefficient_1\\\vdots\\\Coefficient_{\dimROM}
	\end{bmatrix} =\vcentcolon \rom{\mathcal{A}}\SolROM
\end{displaymath}
with $\rom{\mathcal{A}}\in\mathbb{R}^{\dimROM\times\dimROM}$. 
Notice that the precomputation is possible for all polynomial structures, see for instance \cite{KraW19} and the references therein. 
Indeed, let us present the details for a quadratic polynomial, exemplified by a linear operator $\rom{\mathcal{N}}\colon (\Y\otimes\Y)\to\X$, where $\Y\otimes\Y$ denotes the tensor product of $\Y$ with itself. In this case, we have
\begin{displaymath}
	\mathcal{N}\left(\left(\sum_{i=1}^{\dimROM} \Coefficient_i\Mode_i\right)\otimes \left(\sum_{\ell=1}^{\dimROM} \Coefficient_\ell\Mode_\ell\right)\right) = \sum_{i=1}^{\dimROM}\sum_{\ell=1}^{\dimROM} \Coefficient_i\Coefficient_\ell \mathcal{N}(\Mode_i\otimes \Mode_\ell).
\end{displaymath}
Defining
\begin{displaymath}
	\rom{\mathcal{N}} \vcentcolon= \begin{bmatrix}
		\ipX{ \Mode_1,\mathcal{N}(\Mode_1\otimes \Mode_1)} & \ipX{ \Mode_1,\mathcal{N}(\Mode_1\otimes \Mode_2)} & \dots & 
		\ipX{ \Mode_1, \mathcal{N}(\Mode_{\dimROM}\otimes \Mode_{\dimROM}) }\\
		\vdots & \vdots & & \vdots\\
		\ipX{ \Mode_{\dimROM}, \mathcal{N}(\Mode_1\otimes \Mode_1)} & \ipX{ \Mode_{\dimROM}, \mathcal{N}(\Mode_1\otimes \Mode_2)} & \dots & 
		\ipX{ \Mode_{\dimROM}, \mathcal{N}(\Mode_{\dimROM}\otimes \Mode_{\dimROM}) }
	\end{bmatrix}\in\mathbb{R}^{\dimROM\times \dimROM^2}
\end{displaymath}
implies
\begin{displaymath}
	\begin{bmatrix}
		\ipX{ \Mode_1, \mathcal{N}\left(\left(\sum_{i=1}^{\dimROM} \Coefficient_i\Mode_i\right)\otimes \left(\sum_{\ell=1}^{\dimROM} \Coefficient_\ell\Mode_\ell\right)\right) }\\
		\vdots\\
		\ipX{ \Mode_{\dimROM}, \mathcal{N}\left(\left(\sum_{i=1}^{\dimROM} \Coefficient_i\Mode_i\right)\otimes \left(\sum_{\ell=1}^{\dimROM} \Coefficient_\ell\Mode_\ell\right)\right) }
	\end{bmatrix} = \rom{\mathcal{N}}(\SolROM\otimes\SolROM),
\end{displaymath}
where $\rom{\mathcal{N}}$ can be precomputed in the offline phase and, thus, the right-hand side can be computed independently of the original space $\Y$. 
Recall that, for each time instance, $\SolROM$ is an $r$-dimensional real vector and thus the tensor product $\SolROM\otimes\SolROM$ reduces (up to an isomorphism) to the standard Kronecker product. The procedure extends directly for a general polynomial but for the sake of notation we omit the details.

\begin{remark}
	If the Hilbert space $\X$ is finite-dimensional it is isomorphic to $\mathbb{R}^{\dimFOM}$ (with standard inner product) such that without loss of generality we may assume $\X = \mathbb{R}^{\dimFOM}$. In this case, the basis vectors $\Mode_{j}$ ($j=1,\ldots,\dimROM$) form a matrix
	\begin{displaymath}
		\CoordinateMatrix \vcentcolon= \begin{bmatrix}
			\Mode_1 & \dots & \Mode_{\dimROM}
		\end{bmatrix}\in\mathbb{R}^{\dimFOM\times\dimROM}.
	\end{displaymath}
	In this case, the polynomial operators are associated with matrices, i.e., $\mathcal{A}\in\mathbb{R}^{\dimFOM\times\dimFOM}$ and $\mathcal{N}\in\mathbb{R}^{\dimFOM\times\dimFOM^2}$ and the reduced analogues are given by
	\begin{displaymath}
		\rom{\mathcal{A}} = \CoordinateMatrix^T\mathcal{A}\CoordinateMatrix\qquad\text{and}\qquad
		\rom{\mathcal{N}} = \CoordinateMatrix^T\mathcal{N}(\CoordinateMatrix\otimes\CoordinateMatrix),
	\end{displaymath}
	where again $\otimes$ denotes the Kronecker product.
\end{remark}

Albeit many nonlinear systems can be rewritten as polynomial systems by introducing additional states \cite{Gu11}, it may not be possible to reduce the computational complexity to a satisfactory level with the approach presented above. 
To remedy this problem a standard approach is to further approximate the nonlinear function, for instance via the \gls{EIM} \cite{BarMNP04} or the \gls{DEIM} \cite{ChaS11}. 
Although the extension of these methods to our methodology presented in the forthcoming \cref{sec:onlinePhase} is certainly an interesting aspect we consider this a second step and postpone the extension to a future work.

\begin{remark}[Parameter separability]
	\label{rem:paramSeparability}
	If the right-hand side in \eqref{eq:FOMPDE} depends on a parameter $\param$ and is separable with respect to this parameter, that is $\F(\Sol,\param) = \sum_{k=1}^K \theta_k(\param)\F_k(\Sol)$ with suitable scalar-valued functions $\theta_k$, then this structure is retained in the \gls{ROM} by setting $\FROM(\SolROM,\param) = \sum_{k=1}^K \theta_k(\param) \FROM_k(\SolROM)$. 
	This facilitates the efficient usage of the \gls{ROM} in a many-query context, where the \gls{ROM} has to be evaluated for many different parameter values.
\end{remark}

\section{Identification of Dominant Modes}
\label{sec:shiftedPOD}

Similarly as in \cref{sec:POD}, we aim for identifying the dominant modes of the system, which capture most of the dynamics.
However, instead of considering a linear Galerkin approach as in \eqref{eq:GalerkinAnsatz}, we use the more general ansatz \eqref{eq:multiframeAnsatz} here.
Thus, similarly to \cref{sec:POD}, we assume that we are given $\Sol\in\LTwo{0,T;\X}$, which we want to approximate via \eqref{eq:multiframeAnsatz}, \ie, we want to solve the minimization problem
\begin{equation}
	\label{eq:sPODminimization}
	\left\{\begin{aligned}
		&\min\frac{1}{2}\int\limits_0^\TimeEnd \left\lVert z(t)-\sum_{i=1}^{\dimROM} \Coeff_{i}(t) \Transformation_i \left(\Path_i(t)\right)\Mode_i\right\rVert_\X^2 \dt,\\
		&\text{\st}\ \Mode\in\Y^{\dimROM},\, \|\Mode_i\|_\X = 1,\, \Coefficient_i\in\LTwo{0,T},\, \Path_i\in\LTwo{0,T;\PathSpace_i}\ \text{for $i=1,\ldots,\dimROM$}.
	\end{aligned}\right.
\end{equation}
Here, we assume that the mappings $\Transformation_i\colon \PathSpace_i\rightarrow \mathcal{B}(\X)$ for $i=1,\ldots,\dimROM$ are given and satisfy the following assumption. 

\begin{assumption}
	\label{ass:offline} 
	For each $\Mode_i\in\Y$ and each $i\in\lbrace 1,\ldots,\dimROM\rbrace$, the mappings $\Transformation_i(\cdot)\Mode_i\colon\PathSpace_i\to\Y$ are continuous. Moreover, there exists a constant $\overline{c}>0$ such that 
	\begin{equation}
		\label{eq:boundednessOfTrafo}
		\operatorNorm{\Transformation_i(\eta)} \leq \overline{c},\quad\forall\,\eta\in\PathSpace_i
		\end{equation}
	for $i=1,\ldots,\dimROM$, where $\operatorNorm{\cdot}$ denotes the induced operator norm.
\end{assumption}

\begin{lemma}
	Let the mappings $\Transformation_i$ satisfy \Cref{ass:offline} and assume that $\Sol\in\LTwo{0,T;\X}$, $\Mode\in\Y^{\dimROM}$, $\SolROM\in\LTwo{0,T;\mathbb{R}^{\dimROM}}$, and $\Path\in\LTwo{0,T;\PathSpace}$ are given.
	Then the integal in \eqref{eq:sPODminimization} is defined.
\end{lemma}

\begin{proof}
	We define for $i=1,\ldots,\dimROM$ the mapping
	\begin{displaymath}
		\alpha_i\colon(0,T)\times \PathSpace_i\to \X,\qquad (t,\eta_i)\mapsto \Coefficient_i(t)\Transformation_i(\eta_i)\Mode_i
	\end{displaymath}
	and the associated Nemytskij operator $\mathcal{A}_i(\Path_i)(t) = \alpha_i(t,\Path_i(t))$ for almost all $t\in(0,T)$. 
	\Cref{ass:offline} implies that $\alpha_i(t,\cdot)\colon\PathSpace_i\to\X$ is continuous for almost all $t\in(0,T)$. 
	By assumption, $\SolROM$ is measurable and thus $\alpha_i(\cdot,\eta_i)$ is measurable for all $\eta_i\in\PathSpace_i$. In particular, $\alpha_i$ satisfies the Carath\'eodory condition and thus $\mathcal{A}_i(\Path_i)$ is measurable \cite{GolKT92}. We conclude the proof by observing
	\begin{align*}
		\|\mathcal{A}_i(\Path_i)\|_{\LTwo{0,T;\X}}^2 = \int_0^T |\Coefficient_i(t)|^2\|\Transformation_i(\Path_i(t))\Mode_i\|_\X^2 \dt \leq \overline{c}^2\|\Coefficient_i\|_{\LTwo{0,T}}^2
	\end{align*}
	and thus $\mathcal{A}_i(\Path_i)\in\LTwo{0,T;\X}$.
\end{proof}

Before we discuss existence of a minimizer of \eqref{eq:sPODminimization} let us illustrate the usefulness of the ansatz \eqref{eq:multiframeAnsatz} by revisiting the linear advection equation discussed in \Cref{ex:advectionEquationPOD}.
\begin{example}
	\label{ex:advectionEquationSPOD}
	Recall that the solution $\Sol$ of the linear advection equation in \Cref{ex:advectionEquationPOD} is given by a shift of the initial condition, \ie, $\Sol(t,\xi)=\InitialCondition(\xi-t)$ for all $(t,\xi)\in(0,1)\times(0,1)$.
	Defining for $\eta\in\mathbb{R}$ the shift operator
	\begin{align*}
		\ShiftOperator(\eta)&\colon \LTwo{0,1} \rightarrow \LTwo{0,1}, & \ShiftOperator(\eta)f  &\vcentcolon= f(\cdot-\eta)
	\end{align*}
	via periodic continuation, we observe that the solution of the advection equation can be written as 
	\begin{displaymath}
		\Sol(t,\xi)=\ShiftOperator(t)\InitialCondition(\xi)\qquad\text{for all $(t,\xi)\in(0,1)\times(0,1)$}.
	\end{displaymath}
	Thus, a minimizer of \eqref{eq:sPODminimization} is given by the choice $\dimROM=1$, $\Hom_1=\ShiftOperator$, $\Mode_1 = \InitialCondition/\|\InitialCondition\|_\X$, $\Coeff_1(t)\equiv \|\InitialCondition\|_\X$, and $\Path_1(t)=t$ for $t\in (0,1)$. 
	Thus, the solution can be described without approximation error with just one mode when using the ansatz \eqref{eq:multiframeAnsatz}.
	Furthermore, while the dominant modes determined via \gls{POD} are independent from the initial condition (cf.~\Cref{ex:advectionEquationPOD}), here, the dominant mode is given by the (normalized) initial condition itself, which in turn fully describes the solution.
	Especially, it is possible to construct initial conditions which result in a need of arbitrarily many \gls{POD} modes to capture the solution, cf.~\cite{CagMS19}, while using \eqref{eq:multiframeAnsatz} only one mode is needed regardless of which initial condition is chosen.
\end{example}

Let us emphasize that in contrast to the \gls{POD} minimization problem \eqref{eq:PODminimization} discussed in \cref{sec:POD}, we only require the modes to be normalized but not necessarily to form an orthonormal set. The proof of \Cref{thm:PODminimization} relies heavily on the fact that the modes are orthogonal. Mimicking this proof would require that $\Transformation_i(\Path_i(t))\Mode_i$ is orthogonal to $\Transformation_j(\Path_j(t))\Mode_j$ for all $i\neq j$ and all $t\in[0,T]$. The next example highlights that in general, this is not a reasonable assumption. Instead, we only require the modes to be normalized in~$\X$.

\begin{example}[Wave equation]
	\label{ex:waveEquation}
	We consider the linear acoustic wave equation in $\Omega \vcentcolon= (0,1)$ with periodic boundary conditions for the density $\rho$ and the velocity $v$ given by
	\begin{equation}
		\label{eq:linWaveEq}
		\left\{\begin{aligned}
			\ParDer{t}{\rho}(t,\xi) + \ParDer{\xi}{v}(t,\xi) &= 0, & (t,\xi)\in(0,1)\times\Omega,\\
			\ParDer{t}{v}(t,\xi) + \ParDer{\xi}{\rho}(t,\xi) &= 0, & (t,\xi)\in(0,1)\times\Omega,\\
			\rho(t,0) &= \rho(t,1), & t\in(0,1),\\
			v(t,0) &= v(t,1), & t\in(0,1),\\
			\rho(0,\xi) &= \rho_0(\xi), & \xi\in\Omega,\\
			v(0,\xi) &= 0, & \xi\in\Omega,
		\end{aligned}\right.
	\end{equation}
	with given initial value 
	\begin{equation*}
		z_0=
		\begin{bmatrix}
			\rho_0\\
			0
		\end{bmatrix}
		\in \Y\vcentcolon= \left(H^1_{\mathrm{per}}(\Omega)\right)^2.
	\end{equation*}
	Similar to \Cref{ex:advectionEquationPOD}, we consider $\InitialCondition$ as an element of $\LocSquareIntegrableFunctions{\RealNumbers;\mathbb{R}^2}$ via periodic continuation. 
	The analytic solution can be expressed as
	\begin{equation}
		\label{eq:linWaveSolution}
		\begin{bmatrix}
			\rho(t,\xi)\\
			v(t,\xi)
		\end{bmatrix}
		=
		\begin{bmatrix}
			1\\
			1
		\end{bmatrix}
		q_{+}(\xi-t)+
		\begin{bmatrix}
			1\\
			-1
		\end{bmatrix}
		q_{-}(\xi+t),
	\end{equation}
	where $q_{+},q_{-}\in \LocSquareIntegrableFunctions{\mathbb{R}}$ are functions determined via the initial value and the boundary conditions, cf.\ \cite{HirR07}.
	In the case of a homogeneous initial condition for $v$, the values of $q_{+}$ and $q_{-}$ in $\Omega$ are determined via
	\begin{equation*}
			q_{+}(\xi)=q_{-}(\xi)=\frac{1}{2}\rho_0(\xi),\quad \xi\in \Omega,
	\end{equation*}
	and are periodically extended to $\LocSquareIntegrableFunctions{\mathbb{R}}$. Now let us assume that we are only interested in a low-dimensional approximation of the density which is given by
	\begin{equation*}
		\rho(t,\xi) = \frac{1}{2}\left(\rho_0(\xi-t)+\rho_0(\xi+t)\right)=\frac{1}{2}\left(\ShiftOperator(t)\rho_0(\xi)+\ShiftOperator(-t)\rho_0(\xi)\right)=\vcentcolon z(t,\xi)
	\end{equation*}
	for $(t,\xi)\in(0,1)\times\Omega$ with the shift operator as defined in \Cref{ex:advectionEquationSPOD}.
	Thus, a minimizer of \eqref{eq:sPODminimization} is given by the choice $\dimROM=2$, $ \Hom_1=\Hom_2=\ShiftOperator$, $\Mode_1 = \Mode_2 = \rho_0/\|\rho_0\|_\X$, and
	\begin{align*}
		\Coeff_1(t)=\Coeff_2(t)&\equiv \frac{\|\rho_0\|_\X}{2}, & \Path_1(t)&=t, & \Path_2(t)&=-t \quad\text{ for } t\in (0,1),
	\end{align*}
	and we conclude that the solution can be described without approximation error with just two modes with \eqref{eq:multiframeAnsatz}.
	Especially, we observe that the transformed modes $\Hom_1(\Path_1(t))\Mode_1$ and $\Hom_2(\Path_2(t))\Mode_2$ become linearly dependent for $t=0$ and $t=1$.
	Thus, even a minimizer of the cost functional may lead to transformed modes which become linearly dependent.
	This observation indicates that it is not reasonable to enforce orthogonality of the transformed modes in contrast to the POD minimization problem addressed in \cref{sec:POD}.
	In this example, we may still obtain linearly independent modes as long as $\rho_0\neq 0$, by adding the velocity data, \ie, by considering $z=[\rho\,\; v]^T$ instead of $z=\rho$.
\end{example}

By assumption, $\Y$ is a dense subspace of $\X$. Since we only require the modes to be normalized in $\X$, it is not clear that \eqref{eq:sPODminimization} possesses a minimizer in $\Y^{\dimROM}$. Instead, we assume that $\Y$ itself is a reflexive Banach space with norm $\|\cdot\|_\Y$. Moreover, we assume that $\Y$ is compactly embedded in $\X$ (cf.~\cite[Def.~21.13]{Zei90a}) and propose to restrict the admissible set of \eqref{eq:sPODminimization} by imposing a bound on the modes with respect to $\|\cdot\|_\Y$. Another drawback of not enforcing that the modes are linearly independent is that the coefficient functions $\Coefficient_i$ may become unbounded. To prevent this, we further restrict the admissible set by imposing a bound on $\SolROM$ in the $L^2$ norm. Finally, we simplify the minimization problem \eqref{eq:sPODminimization} by assuming that the paths are known a priori or have been determined in a pre-processing step, see for instance \cite{MenBALK19,ReiSSM18,RimML18,SchRM19}. In summary, we assume that the cost functional is defined as
\begin{equation}
	\label{eq:sPODminimization:cost}
	J\colon \LTwo{0,T;\mathbb{R}^{\dimROM}}\times\X^{\dimROM} \to \mathbb{R}, \qquad (\SolROM,\Mode)\mapsto 
	\frac{1}{2}\int_0^T\left\|\Sol(t) - \sum_{i=1}^{\dimROM} \Coefficient_i(t)\Transformation_i(\Path_i(t))\Mode_i\right\|_\X^2 \dt,
\end{equation}
where we use the notation $\SolROM = (\SolROM_1,\ldots,\SolROM_{\dimROM})$ and $\Mode = (\Mode_1,\ldots,\Mode_{\dimROM})$. The admissible set is given by
\begin{equation}
	\label{eq:sPODminimzation:admissibleSet}
	\addSet \vcentcolon= \left\{
		(\SolROM,\Mode)\in \LTwo{0,T;\mathbb{R}^{\dimROM}}\times\X^{\dimROM}\,\left|\, \begin{gathered}
			\Mode \in \Y^{\dimROM}, \max \{\|\Mode_i\|_{\Y},\|\Coefficient_i\|_{\LTwo{0,T;\mathbb{R}}}\} \leq C,\\
			\text{and } \|\Mode_i\|_{\X} = 1 \text{ for } i=1,\ldots,\dimROM
		\end{gathered} \right.
	\right\}
\end{equation}
with given constant $C>0$ that is large enough. 
Instead of the minimization problem \eqref{eq:sPODminimization} we thus consider 
\begin{equation}
	\label{eq:sPODminimizationRestriction}
	\min J(\SolROM,\Mode)\qquad\text{\st}\ (\SolROM,\Mode)\in\addSet.
\end{equation}

\begin{example}
	\label{ex:compactEmbedding}
	For a bounded Lipschitz domain $\Omega\subseteq\mathbb{R}^d$ set $\X = \LTwo{\Omega}$ and $\Y = H^1(\Omega)$. Then $\Y$ is compactly embedded in $\X$, see for instance \cite[Thm.~9.16]{Bre11} or \cite[Thm.~21.A]{Zei90a}. Note that in this case, we can replace $\|\Mode_i\|_\Y\leq C$ by $\|\partial_\xi \Mode_i\|_\X \leq C$ in \eqref{eq:sPODminimzation:admissibleSet}.
\end{example}

\begin{theorem}
	\label{thm:sPODminimization}
	Let $(\Y,\|\cdot\|_{\Y})$ be a reflexive Banach space which is compactly embedded in $\X$. Moreover, assume that paths $\Path_i\in\LTwo{0,T;\PathSpace_i}$ are given and the transformation operators satisfy \Cref{ass:offline}. Then \eqref{eq:sPODminimizationRestriction} has a solution.
\end{theorem}

\begin{proof}
	Let $J^\star\geq0$ denote the infimum of $J$ over $\addSet$ and let $(\SolROM^k,\Mode^k)$ denote a sequence in $\addSet$ with $\lim_{k\to\infty} J(\SolROM^k,\Mode^k) = J^\star$. 
	Since $(\SolROM^k,\Mode^k)$ is bounded in $\LTwo{0,T,\mathbb{R}^{\dimROM}}\times \Y^{\dimROM}$, the Eberlein-\u{S}muljian theorem \cite[Thm.~21.D]{Zei90a} ensures that $(\SolROM^k,\Mode^k)$ possesses a weakly convergent subsequence $(\SolROM^{k_n},\Mode^{k_n})$ with limit $(\SolROM^\star,\Mode^\star)$, i.e., $(\SolROM^{k_n},\Mode^{k_n}) \rightharpoonup (\SolROM^\star,\Mode^\star)$ in $\LTwo{0,T;\mathbb{R}^{\dimROM}}\times\Y^{\dimROM}$ for $n\to\infty$. 
	Since $\Y$ is compactly embedded into $\X$ we conclude strong convergence in $\X$ \cite[Prop.~21.35]{Zei90a}, i.e., $\Mode^{k_n}\to \Mode^\star$ in $\X^r$. 
	This immediately implies $\|\Mode_i^\star\|_\X = 1$ for $i=1,\ldots,\dimROM$. Define the bilinear mapping
	\begin{displaymath}
		\beta\colon \LTwo{0,T;\mathbb{R}^{\dimROM}} \times \X^{\dimROM}\to \LTwo{0,T;\X},\qquad (\SolROM,\Mode) \mapsto \sum_{i=1}^\dimROM \Coefficient_i(\cdot)\Transformation_i(\Path_i(\cdot))\Mode_i.
	\end{displaymath}
	For $f\in\LTwo{0,T;\X}$, $\SolROM\in\LTwo{0,T;\mathbb{R}^{\dimROM}}$, and $\Mode\in \X^{\dimROM}$ we compute
	\begin{align*}
		\langle f,\beta(\SolROM,\Mode)\rangle_{\LTwo{0,T;\X}}
		&= \sum_{i=1}^{\dimROM} \int_0^T \Coefficient_i(t) \ipX{ f(t),\Transformation_i(\Path_i(t))\Mode_i} \dt.
	\end{align*}
	Since $\Mode_i^{k_n}\to \Mode_i^\star$ for $n\to\infty$ implies $\ipX{ f(t),\Transformation_i(\Path_i(t))\Mode_i^{k_n}} \to \ipX{ f(t),\Transformation_i(\Path_i(t))\Mode_i^\star}$ for $n\to\infty$ and almost all $t\in(0,T)$, we use \cite[Prop.~21.23(j)]{Zei90a} to infer $\beta(\SolROM^{k_n},\Mode^{k_n}) \rightharpoonup \beta(\SolROM^\star,\Mode^\star)$. 
	The claim now follows from the fact that the norm is weakly sequentially lower semi-continuous \cite[Prop.~21.23(c)]{Zei90a} and thus
	\begin{displaymath}
		J^\star \leq J(\SolROM^\star,\Mode^\star) \leq \liminf_{n\to\infty} J(\SolROM^{k_n},\Mode^{k_n}) = \lim_{k\to\infty} J(\SolROM^k,\Mode^k) = J^\star.\qedhere
	\end{displaymath}
\end{proof}

\begin{remark}
	Instead of restricting the admissible set in \eqref{eq:sPODminimzation:admissibleSet} to bounded coefficients and modes, we can alternatively regularize the cost functional. In more detail, we consider
	\begin{equation}
		\label{eq:sPODminimizationRegularization}
		\left\{\begin{aligned}
		&\min J(\SolROM,\Mode) + \frac{\gamma_1}{2}\|\SolROM\|_{\LTwo{0,T;\RealNumbers^{\dimROM}}}^2 + \frac{\gamma_2}{2}\|\Mode\|_{\Y^{\dimROM}}^2\\
		&\text{\st}\ \Mode\in \Y^{\dimROM} \text{ and } \|\Mode_i\|_\X = 1 \text{ for } i=1,\ldots,\dimROM
	\end{aligned}\right.
	\end{equation}
	with given regularization parameters $\gamma_1,\gamma_2>0$ instead of \eqref{eq:sPODminimizationRestriction}. Note that $\gamma_1,\gamma_2>0$ implies that a minimizing sequence $(\SolROM^k,\Mode^k)$ is bounded and thus one can show existence of a minimizer of \eqref{eq:sPODminimizationRegularization} as in the proof of \Cref{thm:sPODminimization}. In general, we expect that \eqref{eq:sPODminimizationRegularization} is favorable compared to \eqref{eq:sPODminimizationRestriction} from a numerical point of view. This is subject to further investigation.
\end{remark}

Let us emphasize that even in the case that all transformation operators are given by the identity, it is not clear that the minimizer of \eqref{eq:sPODminimizationRestriction} is unique, see \Cref{rem:uniquePODbasis}. In particular the uniqueness depends on the data $\Sol$ and thus without further restrictions on $\Sol$ we cannot expect to establish that the minimizer of \eqref{eq:sPODminimizationRestriction} is unique. 

In order to numerically solve the minimization problem \eqref{eq:sPODminimizationRestriction}, it needs to be discretized in space and time.
	In \cite{SchRM19}, the authors present a method which first discretizes the problem and afterwards solves the fully discrete minimization problem numerically.
	Related approaches are presented in \cite{Rei18,ReiSSM18,RimML18}, where the transformed modes are also identified based on fully discrete problems.
	These approaches have in common that the problem to be solved is already formulated based on a given discretization in space and time.
	
In the case that all modes are transformed by the same operator, i.e., $J$ is given by
\begin{equation}
	\label{eq:sPODminimization:costFunctional:singleShift}
	J(\SolROM,\Mode) \vcentcolon= \frac{1}{2}\int_0^T \left\| \Sol(t)-\Transformation(\Path(t))\sum_{i=1}^{\dimROM} \Coefficient_i(t)\Mode_i\right\|_{\X}^2 \dt,
\end{equation}
we can use the following observation to compute a minimizer of \eqref{eq:sPODminimizationRestriction}. For the special case that the transformation $\Transformation$ is given by the shift operator, this was recognized for instance in \cite{CagMS19,ReiSSM18,SchRM19}.

\begin{theorem}
	\label{thm:sPODminimization:singleTransformation}
	For given data $\Sol\in\LTwo{0,T;\Y}$ and a given path $\Path\in\LTwo{0,T;\PathSpace}$, consider the minimization problem \eqref{eq:sPODminimizationRestriction} with $J$ as defined in \eqref{eq:sPODminimization:costFunctional:singleShift}. 
	Suppose that $\Transformation$ is isometric and satisfies \Cref{ass:offline}. Let $\Mode^\star \vcentcolon= (\Mode_1^\star,\ldots,\Mode_{\dimROM}^\star)$ denote a solution of the \gls{POD} minimization problem \eqref{eq:PODminimization2} with transformed data $\Transformation^*(\Path)\Sol$ with corresponding eigenvalues $\lambda_1\geq\ldots\geq\lambda_{\dimROM} > 0$ as defined in \Cref{thm:PODminimization}. 
	Define $\SolROM^\star = (\Coefficient_1^\star,\ldots,\Coefficient_{\dimROM}^\star)$ via $\SolROM_i^\star \vcentcolon= \ipX{ \Sol,\Mode_i^\star}$ for $i=1,\ldots,\dimROM$. 
	If $C$ in \eqref{eq:sPODminimzation:admissibleSet} satisfies
	\begin{displaymath}
		\max\left\{\frac{1}{2\lambda_{\dimROM}}\left(\|\Sol\|_{\LTwo{0,T,\X}}^2 + \|\Sol\|_{\LTwo{0,T,\Y}}^2\right),\|\Sol\|_{\LTwo{0,T,\X}}\right\} \leq C,
	\end{displaymath}
	then $(\SolROM^\star,\Mode^\star)$ is a minimizer of \eqref{eq:sPODminimizationRestriction}.
\end{theorem}

\begin{proof}
	Since $\Transformation$ is isometric, we have
	\begin{displaymath}
		J(\SolROM,\Mode) = \frac{1}{2}\int_0^T \left\| \Transformation^*(\Path(t))\Sol(t)-\sum_{i=1}^\dimROM \Coefficient_i(t) \Mode_i \right\|_{\X}^2 \dt.
	\end{displaymath}
	It is easy to see that we can substitute the condition $\|\Mode_i\|_\X = 1$ for $i=1,\ldots,\dimROM$ in the admissible set \eqref{eq:sPODminimzation:admissibleSet} by the condition $\ipX{ \Mode_i,\Mode_j} = \delta_{ij}$ for $i,j=1,\ldots,\dimROM$ without changing the minimum. It thus remains to show that $(\SolROM^\star,\Mode^\star)$ is an element of the admissible set defined in \eqref{eq:sPODminimzation:admissibleSet}. We immediately obtain
	\begin{displaymath}
		\|\Coefficient_i^\star\|_{\LTwo{0,T}}^2 \leq \int_0^T \|\Sol(t)\|_{\X}^2\|\Mode_i^\star\|_{\X}^2 \dt = \|\Sol\|_{\LTwo{0,T;\X}}^2 \leq C.
	\end{displaymath}
	For the estimate of the modes we use the operator $\mathcal{R}$ defined in \eqref{eq:PODoperator} and Young's inequality to obtain
	\begin{align*}
		\|\Mode_i^\star\|_{\Y} &= \frac{1}{\lambda_i}\|\mathcal{R}\Mode_i^\star\|_{\Y} \leq \frac{1}{\lambda_i}\int_0^T |\Coefficient_i^\star(t)|\,\|\Sol(t)\|_{\Y}\dt \leq \frac{1}{2\lambda_i}\left( \|\Coefficient_i^\star\|_{\LTwo{0,T}}^2 + \|\Sol\|_{\LTwo{0,T;\Y}}^2\right)\\
		&\leq \frac{1}{2\lambda_i}\left(\|\Sol\|_{\LTwo{0,T;\X}}^2 + \|\Sol\|_{\LTwo{0,T;\Y}}^2\right) \leq C.\qedhere
	\end{align*}
\end{proof}

\section{Reduced Order Model with Transformed Modes}
\label{sec:onlinePhase}

Suppose now that we are given suitable transformation operators $\Transformation_i$ and have identified a set of dominant modes $\Mode_i$, for instance via the procedure described in \cref{sec:shiftedPOD}. Then we are able to construct a \gls{ROM} for \eqref{eq:FOMPDE} via Galerkin projection, i.e., by substituting the approximation \eqref{eq:multiframeAnsatz} in \eqref{eq:FOMPDE} and formally project the resulting equations onto the time-dependent approximation space
\begin{align}
	\label{eq:ProjectionSubspace}
    \Span \left\{ \Transformation_j(\Path_j(t))\Mode_j \mid j=1,\ldots,\dimROM\right\}.
\end{align}
Since the abstract differential equation \eqref{eq:FOMPDE} involves a differentiation with respect to time, we have to assume that the transformation operators are continuously differentiable. Indeed, we only require that the transformation applied to the respective mode is continuously differentiable and thus make the following assumption. 

\begin{assumption}
	\label{ass:TransformationDifferentiability}
	The mappings $\Transformation_i(\cdot)\Mode_i\colon \PathSpace_i\to\Y$ are continuously differentiable.
\end{assumption}

\begin{example}
	\label{ex:shiftOperator}
	The shift operator $\Transformation(\Path)\Sol = \Sol(\cdot - \Path)$ from \Cref{ex:advectionEquationSPOD} with periodic embedding into $\LTwo{0,1}$ is a strongly continuous semigroup \cite[Ex.~I.5.4]{EngN00}. In particular, semigroup theory implies that $\Transformation(\cdot)\Mode$ is continuously differentiable for all
	\begin{displaymath}
		\Mode\in D(\mathcal{A}) = H^1_{\mathrm{per}}(0,1)
	\end{displaymath}
	see for instance \cite[Cha.\,1,~Thm.\,2.4]{Paz83}.
	We conclude that the shift operator satisfies \Cref{ass:TransformationDifferentiability}, whenever $\Y \subseteq D(\mathcal{A})$.
\end{example}

By abuse of notation we denote the derivative of $\Transformation_i(\cdot)\Mode_i$ at $\Path_i\in\PathSpace_i$ by $\left[\Transformation_i'(\Path_i)\Mode_i\right]\in \mathcal{L}(\PathSpace_i,\X)$.
Recall that for the sake of notation we assume $\PathSpace_i = \mathbb{R}^{\dimPathSpace_i}$ and thus $\PathSpace \vcentcolon = \PathSpace_1 \times \dots \times \PathSpace_{\dimROM} = \mathbb{R}^{\dimPathSpace}$ with $\dimPathSpace \vcentcolon= \sum_{i=1}^\dimROM \dimPathSpace_i$. The Galerkin projection of \eqref{eq:FOMPDE} onto \eqref{eq:ProjectionSubspace} is then given by
\begin{equation}
	\label{eq:ROMtransformed}
	\PODmassMatrixState(\Path(t))\dot{\SolROM}(t) + \PODpathMatrix(\Path(t))D(\SolROM(t))\dot{\Path}(t) = \FROMstate(\Path(t),\SolROM(t))
\end{equation}
with state and path vectors
\begin{align}
	\label{eq:ROMtransformedState}
	\SolROM(t) &\vcentcolon= \begin{bmatrix}
		\Coefficient_1(t) &
		\cdots &
		\Coefficient_{\dimROM}(t)
	\end{bmatrix}^T \in \mathbb{R}^{\dimROM}, & 
	\Path(t) &\vcentcolon= \begin{bmatrix}
		\Path_1^T(t) & \cdots & \Path_{\dimROM}^T(t)
	\end{bmatrix}^T \in \mathbb{R}^{\dimPathSpace},
\end{align}
mass matrix $\PODmassMatrixState(\Path) \vcentcolon= [\ipX{ \Transformation_i(\Path_i)\Mode_i,\Transformation_j(\Path_j)\Mode_j}]_{i,j=1}^{\dimROM}\in\mathbb{R}^{\dimROM\times\dimROM}$, correlation block matrix 
\begin{displaymath}
	\PODpathMatrix(\Path) \vcentcolon= \begin{bmatrix}\ipX{ \Transformation_i(\Path_i)\Mode_i,[\Transformation_j'(\Path_j)\Mode_j]e_1} & \cdots & \ipX{ \Transformation_i(\Path_i)\Mode_i,[\Transformation_j'(\Path_j)\Mode_j]e_{\dimPathSpace_j}} \end{bmatrix}_{i,j=1}^{\dimROM}\in\mathbb{R}^{\dimROM\times\dimPathSpace},
\end{displaymath}
diagonal matrix $D(\SolROM) \vcentcolon= \diag(\Coefficient_1 I_{\dimPathSpace_1},\ldots,\Coefficient_{\dimROM}I_{\dimPathSpace_{\dimROM}})\in\mathbb{R}^{\dimPathSpace\times\dimPathSpace}$,
and right-hand side
\begin{align*}
	\FROMstate(\Path,\SolROM) \vcentcolon= \begin{bmatrix}
		\ipX{ \Transformation_1(\Path_1)\Mode_1, \F(\sum_{i=1}^{\dimROM} \Coefficient_i\Transformation_i(\Path_i)\Mode_i} \\
		\vdots \\
		\ipX{ \Transformation_{\dimROM}(\Path_{\dimROM})\Mode_{\dimROM}, \F(\sum_{i=1}^{\dimROM} \Coefficient_i\Transformation_i(\Path_i)\Mode_i}
	\end{bmatrix}.
\end{align*}
Hereby, $e_i$ denotes the $i$th unit vector of suitable dimension, such that $(e_1,\ldots,e_{\dimPathSpace_i})$ forms a basis of $\PathSpace_i = \mathbb{R}^{\dimPathSpace_i}$. 

As in \cref{sec:offlineOnlineDecomposition}, the right-hand side $\FROMstate$ still depends on the original space $\Y$ and requires further simplifications. For instance, if $\F$ is given by a quadratic polynomial of the form $\F(\Sol) = \mathcal{N}(\Sol\otimes\Sol)$ with linear operator $\mathcal{N}\colon \Y\otimes\Y\to\X$ we can (for given $\Path$) precompute the quantities $\ipX{ \Transformation_j(\Path_j)\Mode_j,\mathcal{N}(\Transformation_i(\Path_i)\Mode_i\otimes \Transformation_\ell(\Path_\ell)\Mode_\ell)}$ for $i,j,\ell=1,\ldots,\dimROM$. A further simplification is possible if $\F$ is \emph{equivariant} with respect to the transformation operators $\Transformation_i$ (see the upcoming \Cref{ass:equivarianceGroupAction} and the discussion thereafter for further details). Let us mention that parameter separability (see~\Cref{rem:paramSeparability}) is easily retained in the \gls{ROM} \eqref{eq:ROMtransformed}.

\begin{remark}
	In contrast to \gls{POD} we cannot ensure that $\PODmassMatrixState(\Path)$ is nonsingular for every $\Path\in\PathSpace$, since some of the modes $\Transformation_i(\Path_i)\Mode_i$ may become linearly dependent (see \Cref{ex:waveEquation}). 
	This may happen either at single time points or at a complete time-interval. 
	In the latter case this implies that some of the modes are redundant and can be removed during this interval. 
	In any case, whenever $\PODmassMatrixState(\Path)$ becomes singular we have to restart the computation of the reduced model.
\end{remark}

It is clear that \eqref{eq:ROMtransformed} is not sufficient to compute $\SolROM$ and $\Path$ and hence can be understood as underdetermined \gls{DAE}, cf.\,\cite{KunM06}. To complete the underdetermined \gls{DAE} \eqref{eq:ROMtransformed} we have to add additional equations
\begin{equation}
	\label{eq:phaseConditions}
	\PhaseCondition(\Path,\dot{\Path},\SolROM,\dot{\SolROM}) = 0
\end{equation}
and consider the coupled system \eqref{eq:ROMtransformed} and \eqref{eq:phaseConditions}. 
In the literature, these equations are called \emph{phase conditions} \cite{BeyT04, OhlR13} or \emph{reconstruction equations} \cite{RowKML03,RowM00} and are used to determine the path $\Path(t)$ along the solution $\SolROM$. 
Although several choices for $\PhaseCondition$ are proposed in \cite{BeyT04}, it is not clear a-priori, which phase condition benefits the model the most. 
Since our \gls{ROM} is obtained via Galerkin projection, which satisfies the continuous optimality principle (see \Cref{lem:continuousOptimality} and \cite{CarBA17,LeeC19}), we propose to construct the phase condition also via continuous optimality. 
More precisely we define
\begin{equation}
	\label{eq:sPODconstraint}
	\PhaseCondition_{\mathrm{Res}}(\Path,\dot{\Path},\SolROM,\dot{\SolROM}) \vcentcolon= D(\SolROM)^T\left(\PODpathMatrix(\Path)^T\dot{\SolROM} + \PODmassMatrixPath(\Path)D(\SolROM)\dot{\Path} - \FROMpath(\Path,\SolROM)\right)
\end{equation}
with block mass matrix
\begin{equation}
	\label{eq:PhaseConstraingMatrices}
	\begin{aligned}
		\PODmassMatrixPath(\Path) &\vcentcolon= \begin{bmatrix}\ipX{ [\Transformation_i'(\Path_i)\Mode_i]e_1,[\Transformation_j'(\Path_j)\Mode_j]e_1} & \cdots & \ipX{ [\Transformation_i'(\Path_i)\Mode_i]e_1,[\Transformation_j'(\Path_j)\Mode_j]e_{\dimPathSpace_j}} \\
		\vdots & & \vdots\\
		\ipX{ [\Transformation_i'(\Path_i)\Mode_i]e_{\dimPathSpace_i},[\Transformation_j'(\Path_j)\Mode_j]e_1} & \cdots & \ipX{ [\Transformation_i'(\Path_i)\Mode_i]e_{\dimPathSpace_i},[\Transformation_j'(\Path_j)\Mode_j]e_{\dimPathSpace_j}}\end{bmatrix}_{i,j=1}^{\dimROM}\hspace*{-2em} \in \mathbb{R}^{\dimPathSpace\times\dimPathSpace}
	\end{aligned}
\end{equation}
and reduced right-hand side
\begin{displaymath}
		\FROMpath(\Path,\SolROM) \vcentcolon= \begin{bmatrix} \ipX{ [\Transformation_i'(\Path_i)\Mode_i]e_1, \F\left(\sum_{j=1}^{\dimROM}\Coefficient_j\Transformation_j(\Path_j)\Mode_j\right)}\\
		\vdots\\
		\ipX{ [\Transformation_i'(\Path_i)\Mode_i]e_{\dimPathSpace_i}, \F\left(\sum_{j=1}^{\dimROM}\Coefficient_j\Transformation_j(\Path_j)\Mode_j\right)}
		 \end{bmatrix}_{i=1}^{\dimROM}\hspace*{-1.5em} \in \mathbb{R}^{\dimPathSpace}.
\end{displaymath}
The coupled \gls{ROM} for the reduced state $\SolROM$ and the transformation path $\Path$ is thus given by
\begin{subequations}
	\label{eq:ROMcoupled}
	\begin{align}
		\label{eq:ROMcoupled:ROM}\PODmassMatrixState(\Path(t))\dot{\SolROM}(t) + \PODpathMatrix(\Path(t))D(\SolROM(t))\dot{\Path}(t) &= \FROMstate(\Path(t),\SolROM(t)),\\
		\label{eq:ROMcoupled:PhaseConstraint}D(\SolROM(t))^T\PODpathMatrix(\Path(t))^T\dot{\SolROM}(t) + D(\SolROM(t))^T\PODmassMatrixPath(\Path(t))D(\SolROM(t))\dot{\Path}(t) &= D(\SolROM(t))^T\FROMpath(\Path(t),\SolROM(t)),
	\end{align}
\end{subequations}
or equivalently in matrix notation
\begin{equation}
	\begin{bmatrix}
		I_{\dimROM} & 0\\
		0 & D(\SolROM)^T
	\end{bmatrix}
	\begin{bmatrix}
		\PODmassMatrixState(\Path) & \PODpathMatrix(\Path) \\
		\PODpathMatrix(\Path)^T & \PODmassMatrixPath(\Path)
	\end{bmatrix}\begin{bmatrix}
		I_{\dimROM} & 0\\
		0 & D(\SolROM)
	\end{bmatrix}\begin{bmatrix}
		\dot{\SolROM}\\
		\dot{\Path}
	\end{bmatrix} = \begin{bmatrix}
		\FROMstate(\Path,\SolROM)\\
		D(\SolROM)^T\FROMpath(\Path,\SolROM)
	\end{bmatrix}.
\end{equation}
\begin{remark}
	The phase condition \eqref{eq:ROMcoupled:PhaseConstraint} can be obtained from \eqref{eq:FOMPDE} by substituting the ansatz \eqref{eq:multiframeAnsatz} and enforcing that the residual is orthogonal to
	\begin{equation}
		\label{eq:PhaseConditionProjectionSpace}
		\Span\left\{ \Coefficient_i[\Transformation_i'(\Path_i)\Mode_i]e_j \mid i=1,\ldots,\dimROM, j=1,\ldots,\dimPathSpace_i\right\}.
	\end{equation}
	Thus the phase condition \eqref{eq:ROMcoupled:PhaseConstraint} can be obtained via projection onto the space in \eqref{eq:PhaseConditionProjectionSpace}.
\end{remark}
For $t>0$ we define $\mathscr{R}\colon (0,T]\times \mathbb{R}^{\dimROM}\times \PathSpace \to \X$ via
\begin{equation*}
	\mathscr{R}(t,x,\eta) = \sum_{i=1}^\dimROM x_i \Transformation_i(\Path_i(t))\Mode_i + \sum_{i=1}^r \Coefficient_i(t) \left[\Transformation_i'(\Path_i(t))\Mode_i\right]\eta_i - \F\left(\sum_{i=1}^{\dimROM}\Coefficient_i(t)\Transformation_i(\Path_i(t))\Mode_i\right)
\end{equation*}
such that the residual that is obtained at time $t>0$ by substituting the ansatz~\eqref{eq:multiframeAnsatz} into the evolution equation~\eqref{eq:FOMPDE} is given by $\mathscr{R}(t,\dot{\SolROM}(t),\dot{\Path}(t))$.

\begin{theorem}[Continuous optimality]
	\label{thm:continuousOptimality}
	The \gls{ROM} \eqref{eq:ROMcoupled} is continuously optimal in the sense that if $(\SolROM,\Path)$ is a solution of \eqref{eq:ROMcoupled}, then for each $t>0$, the pair $(\dot{\SolROM}(t),\dot{\Path}(t))$ is a minimizer of the norm of $\mathscr{R}$, i.e.
	\begin{displaymath}
		\|\mathscr{R}(t,\dot{\SolROM}(t),\dot{\Path}(t))\|_\X \leq \|\mathscr{R}(t,x,\eta)\|_\X\qquad\text{for all $(x,\eta)\in\mathbb{R}^{\dimROM}\times \PathSpace$}.
	\end{displaymath}
\end{theorem}

\begin{proof}
	Let $t>0$. We first notice that $\|\mathscr{R}(t,x,\eta)\|_\X^2$ is convex in $(x,\eta)$ and hence the first-order necessary optimality condition is also sufficient. The partial derivatives with respect to the first variable are given by
	\begin{multline*}
		\frac{\partial}{\partial x_\ell} \|\mathscr{R}(t,x,\eta)\|_\X^2 = 2\sum_{i=1}^{\dimROM} x_i \ipX{ \Transformation_\ell(\Path_\ell(t))\Mode_\ell,\Transformation_i(\Path_i(t))\Mode_i}\\
		 + 2\sum_{i=1}^{\dimROM} \Coefficient_i(t) \ipX{ \Transformation_\ell(\Path_\ell(t))\Mode_\ell,\left[\Transformation_i'(\Path_i(t))\Mode_i\right]\eta_i} - 2\ipX{ \Transformation_\ell(\Path_\ell(t))\Mode_\ell, \F\left(\sum_{i=1}^{\dimROM} \Coefficient_i(t)\Transformation_i(\Path_i(t))\Mode_i\right)}
	\end{multline*}
	for $\ell=1,\ldots,\dimROM$. 
	The partial derivatives with respect to the second variable constitute linear mappings $\frac{\partial}{\partial \eta_\ell} \|\mathscr{R}(t,x,\eta)\|_\X^2 \colon \PathSpace_\ell \to \mathbb{R}$ given by
	\begin{multline*}
		\frac{\partial}{\partial \eta_\ell} \|\mathscr{R}(t,x,\eta)\|_\X^2(\lambda_\ell) = 2 \sum_{i=1}^{\dimROM} \Coefficient_\ell(t)\Coefficient_i(t) \ipX{ \left[\Transformation_\ell'(\Path_\ell(t))\Mode_\ell\right]\lambda_\ell, \left[\Transformation_i'(\Path_i(t))\Mode_i\right]\eta_i} \\
		+ 2 \sum_{i=1}^{\dimROM} x_i\Coefficient_\ell(t) \ipX{ \left[\Transformation_\ell'(\Path_\ell(t))\Mode_\ell\right]\lambda_\ell,\Transformation_i(\Path_i(t))\Mode_i}\\
		- 2\Coefficient_\ell(t)\ipX{ \left[\Transformation_\ell'(\Path_\ell(t))\Mode_\ell\right]\lambda_\ell, \F\left(\sum_{i=1}^{\dimROM} \Coefficient_i(t)\Transformation_i(\Path_i(t))\Mode_i\right)}.
	\end{multline*}
	Choosing the standard basis $(e_1,\ldots,e_{\dimPathSpace_\ell})$ for $\PathSpace_\ell$ and using the notation above implies that the first-order necessary condition is given by
	\begin{equation}
		\label{eq:ContinousOptimalityNecessaryCondition}
		\begin{aligned}
			\PODmassMatrixState(\Path(t))x+ \PODpathMatrix(\Path(t))D(\SolROM(t))\eta &= \FROMstate(\Path(t),\SolROM(t)),\\
			D(\SolROM(t))^T\PODpathMatrix(\Path(t))^Tx + D(\SolROM(t))^T\PODmassMatrixPath(\Path(t))D(\SolROM(t))\eta &= D(\SolROM(t))^T\FROMpath(\Path(t),\SolROM(t)).
		\end{aligned}
	\end{equation}	
	Since $(\SolROM,\Path)$ is a solution of \eqref{eq:ROMcoupled} we conclude that $(\dot\SolROM,\dot{\Path})$ is a solution of \eqref{eq:ContinousOptimalityNecessaryCondition} and thus a minimizer of $\|\mathscr{R}(t,x,\eta)\|_\X$.
\end{proof}

\begin{remark}
	The proof of \Cref{thm:continuousOptimality} shows that instead of using the standard basis of $\PathSpace_i = \mathbb{R}^{\dimPathSpace_i}$ it is possible to use any basis of $\PathSpace_i$ for the construction of the \gls{ROM} \eqref{eq:ROMcoupled}.
\end{remark}

If all transformations are chosen constant, then it is easy to see that the phase condition \eqref{eq:phaseConditions} is satisfied for any $(\SolROM,\Path)$ and hence the \gls{ROM} \eqref{eq:ROMcoupled} may not have a unique solution. We immediately conclude that the minimizer in \Cref{thm:continuousOptimality} may not be unique. On the other hand, by virtue of \Cref{ass:TransformationDifferentiability}, the matrix
\begin{equation}
	\label{eq:ROMcoupled:massMatrix}
	M(\Path)\vcentcolon= \begin{bmatrix}
		\PODmassMatrixState(\Path) & \PODpathMatrix(\Path) \\
		\PODpathMatrix(\Path)^T & \PODmassMatrixPath(\Path)
	\end{bmatrix}\in\mathbb{R}^{(\dimROM+\dimPathSpace)\times(\dimROM+\dimPathSpace)}
\end{equation}
is continuous with respect to $\Path$ and thus, if we assume that $M(\Path(0))$ is nonsingular, then there exists a neighborhood $\mathcal{U}\subseteq\mathbb{R}^{\dimPathSpace}$ around $\Path(0)$ such that $M(\Path(0))$ is nonsingular for all $\Path\in\mathcal{U}$. We conclude that $M(\Path)^{-1}$ is continuous for all $\Path\in\mathcal{U}$. As a direct consequence we have proven the following result.

\begin{prop}
	\label{thm:wellPosedROM}
	Let $(\SolROM_{(0)},\Path_{(0)})$ denote the initial value for \eqref{eq:ROMcoupled}, i.e.,
	\begin{equation}
		\label{eq:ROMcoupled:IC}
		\SolROM(0) = \SolROM_{(0)}\qquad\text{and}\qquad
		\Path(0) = \Path_{(0)}.
	\end{equation}
	Assume that \smash{$M(\Path_{(0)})$} in \eqref{eq:ROMcoupled:massMatrix} is nonsingular, $e_i^T\SolROM_{(0)} \neq 0$ for all $i=1,\ldots,\dimROM$, and the transformation operators satisfy \Cref{ass:TransformationDifferentiability}. If $\F$ is continuous, then there exists $\widetilde{T}>0$ such that the \gls{ROM} \eqref{eq:ROMcoupled} has a (classical) solution on $[0,\widetilde{T})$. If the transformation operators and $\F$ are sufficiently smooth, then the solution is unique.
\end{prop}

\begin{remark}
	\label{rem:issueMultiframeApproach}
	The approximation ansatz \eqref{eq:multiframeAnsatz} suffers from the fact that at each time $t\in[0,T)$ with $\Coefficient_i(t)=0$ we cannot expect to determine any information on $\Path_i(t)$. This drawback results in the rather restrictive assumption $e_i^T\SolROM_{(0)}\neq 0$ for all $i=1,\ldots,\dimROM$ in \Cref{thm:wellPosedROM}. We can mitigate this restriction by enforcing the same transformation and the same path for a couple of modes as proposed in \eqref{eq:multframeAnsatzGroup}. In this case, we only have to ensure that for each of the transformations one single coefficient of the initial value is nonzero. This means that the initial condition has to contribute to every reference frame that we are interested in.
\end{remark}

\begin{remark}
	\label{rem:MFEMregularization}
	In principle, the \gls{MFEM} may also suffer from a possible degenerate mass matrix. To circumvent this issue, a regularization is proposed in \cite{MilM81} to prevent nodes to move arbitrarily. In our setting, this corresponds to adding a regularization term for the path variable, respectively its derivative.
\end{remark}

In order to apply \Cref{thm:wellPosedROM}, respectively \Cref{rem:issueMultiframeApproach}, we have to discuss how to choose the initial values $\SolROM_{(0)}$ and $\Path_{(0)}$. Following our general approximation \eqref{eq:multiframeAnsatz} we thus have to find a minimizer for the optimization problem
\begin{equation}
	\label{eq:ICminimization}
	\min_{\Path_{(0)}, \SolROM_{(0)}} J_{\mathrm{IV}} \vcentcolon= \left\|\InitialCondition - \sum_{i=1}^{\dimROM} \Coefficient_{(0),i}\Transformation_i(\Path_{(0),i})\Mode_i\right\|_\X^2.
\end{equation}
The first-order optimality condition is given by
\begin{subequations}
	\label{eq:ICfirstOrder}
\begin{align}
	\label{eq:ICfirstOrder:a}\PODmassMatrixState(\Path_{(0)})\SolROM_{(0)} &= \left[\ipX{ \InitialCondition, \Transformation_i(\Path_{(0),i})\Mode_i}\right]_{i=1}^{\dimROM} =\vcentcolon b_\mathrm{z}(\Path_{(0)}),\\	
	\label{eq:ICfirstOrder:b}D(\SolROM_{(0)})^T\PODpathMatrix(\Path_{(0)})^T\SolROM_{(0)} &= D(\SolROM_{(0)})^T\begin{bmatrix}
		\ipX{ \InitialCondition,[\Transformation'_i(\Path_{(0),i})\Mode_i]e_1}\\
		\vdots\\
		\ipX{ \InitialCondition,[\Transformation'_i(\Path_{(0),i})\Mode_i]e_{\dimPathSpace_i}}
	\end{bmatrix}_{i=1}^r \!\!\!\!\!\!=\vcentcolon D(\SolROM_{(0)})^T b_{\mathrm{p}}(\Path_{(0)}).
\end{align}
\end{subequations}
We immediately notice that if $\PODmassMatrixState(\Path_{(0)})$ is singular, then in general we cannot expect a solution of \eqref{eq:ICfirstOrder:a}. On the other hand, if $\PODmassMatrixState(\Path_{(0)})$ is nonsingular, then we can solve the first equation for $\SolROM(0)$ and it is easy to see that in this case $\Path_{(0)}$ has to satisfy
\begin{equation}
	\label{eq:ICfirstOrder:Path}
	D(\SolROM_{(0)})^T \PODpathMatrix(\Path_{(0)})^T\PODmassMatrixState(\Path_{(0)})^{-1} b_{\mathrm{z}}(\Path_{(0)}) - D(\SolROM_{(0)})^T b_{\mathrm{p}}(\Path_{(0)}) = 0.
\end{equation}
In order to apply the inverse function theorem, we need additional smoothness of the mappings in \Cref{ass:TransformationDifferentiability}. In the context of semigroups, this imposes stronger restrictions on the modes $\Mode_i$. 
Instead, we simply assume that we have initial values $\SolROM_{(0)}$ and $\Path_{(0)}$ available, such that the approximation error $J_{\mathrm{IV}}(\SolROM_{(0)},\Path_{(0)})$ is sufficiently small. If we pick $\Path_{(0)}$ such that $\Transformation_i(\Path_{(0),i}) = \mathrm{Id}_\X$ is the identity on $\X$, then \eqref{eq:ICfirstOrder:a} simply describes the projection of the initial condition on the dominant modes.

For the remainder of this section we analyze the special case, where the right-hand side of \eqref{eq:FOMPDE} is given by a linear operator $\mathcal{A}\colon D(\mathcal{A})\to\X$, i.e., $\F(\Sol) = \mathcal{A}
\Sol$. Assuming additionally that $\mathcal{A}$ is the generator of a strongly continuous semigroup \cite[Cha.~1, Def.~2.1]{Paz83} allows us immediately to establish a simple a-posteriori error bound. Note that in this case $D(\mathcal{A})$ is a dense subspace of $\X$ such that we can choose $\Y \vcentcolon= D(\mathcal{A})$. 

\begin{theorem}[A-posteriori error bound]
	\label{thm:errorBound}
	Let $\Sol$ denote the solution of the \gls{FOM} \eqref{eq:FOMPDE} with initial value $\InitialCondition\in D(\mathcal{A}) = \Y$, linear right-hand side $\F(\Sol) = \mathcal{A}\Sol$ and suppose that $\mathcal{A}$ is the generator of a strongly continuous semigroup $\{S(t)\}_{t\geq 0}$. Assume that the assumptions from \Cref{thm:wellPosedROM} are satisfied and that the mappings $t\mapsto \Transformation_i(\Path_i(t))\Mode_i$ are twice continuously differentiable. For the unique solution $(\SolROM,\Path)$ of the \gls{ROM} \eqref{eq:ROMcoupled} with initial condition \eqref{eq:ROMcoupled:IC} define the error $\varepsilon \vcentcolon= \Sol - \sum_{i=1}^{\dimROM} \SolROM_i\Transformation_i(\Path_i)\Mode_i$.
	Then there exist constants $\widetilde{T},\widetilde{C},\omega\geq 0$ independent of the modes $\Mode_i$ and the transformation operators $\Transformation_i$ such that
	\begin{equation}
		\label{eq:posterioriErrorBound}
		\|\varepsilon(t)\|_\X \leq \widetilde{C}\mathrm{e}^{\omega t}\left(J_{\mathrm{IV}}(\SolROM_{(0)},\Path_{(0)}) + t\|\mathscr{R}(\cdot,\dot{\SolROM},\dot{\Path})\|_{L^\infty(0,t;\X)}\right)
	\end{equation}
	for $t\in [0,\widetilde{T})$.
\end{theorem}

\begin{proof}
	For the proof we first note that~\Cref{thm:wellPosedROM} implies the existence of $\widetilde{T}$ such that~\eqref{eq:ROMcoupled} possesses a unique solution for $t\in[0,\widetilde{T})$ and observe a standard error residual relation given by the (abstract) differential equation
	\begin{equation}
		\label{eq:errorResidualRelation}
		\begin{aligned}
		\dot{\varepsilon}(t) &= \dot{\Sol}(t) - \sum_{i=1}^{\dimROM}\dot{\Coefficient}_i(t)\Transformation_i(\Path_i(t))\Mode_i - \sum_{i=1}^{\dimROM}\Coefficient_i(t)[\Transformation_i'(\Path_i(t))\Mode_i]\dot{\Path}_i(t)\\
			&= \mathcal{A}\varepsilon(t) - \mathscr{R}(t,\dot{\SolROM}(t),\dot{\Path}(t))
		\end{aligned}
	\end{equation}
	together with the initial condition
	\begin{displaymath}
		\varepsilon(0) = \varepsilon_0 \vcentcolon= \InitialCondition-\sum_{i=1}^{\dimROM}\Coefficient_{(0),i}\Transformation_i(\Path_{(0),i})\Mode_i\in D(\mathcal{A}).
	\end{displaymath}
	Since $\{S(t)\}_{t\geq 0}$ is a strongly continuous semigroup, there exist (cf.~\cite[Cha.~1, Thm~2.2]{Paz83}) constants~$\omega\geq 0$ and~$\widetilde{C}\geq 1$ with $\operatorNorm{S(t)} \leq \widetilde{C}\mathrm{e}^{\omega t}$ for all $t\geq 0$. 
	Since $(\SolROM,\Path)$ is a solution of \eqref{eq:ROMcoupled}, we infer that the residual $\mathscr{R}$ is continuously differentiable and thus \cite[Cha.~4, Cor.~2.5]{Paz83} ensures that the (classical) solution of \eqref{eq:errorResidualRelation} is given by 
	\begin{displaymath}
		\varepsilon(t) = S(t)\varepsilon_0 - \int_0^t S(t-\tau)\mathscr{R}(\tau,\dot{\SolROM}(\tau),\dot{\Path}(\tau))\,\mathrm{d}\tau.
	\end{displaymath}
	We conclude the proof by estimating the integral with the supremum norm and the fact that $\operatorNorm{S(\tau)} \leq \widetilde{C}\mathrm{e}^{\omega \tau} \leq \widetilde{C}\mathrm{e}^{\omega t}$ for all $0\leq \tau \leq t$. 
\end{proof}

\begin{remark}
	In many applications, we expect that $\mathcal{A}$ is a semigroup of contractions or even an analytic semigroup with negative spectral abscissa. In either case, we can remove the exponential factor in \eqref{eq:posterioriErrorBound} and hence obtain a linear growth factor in the time variable. Moreover, we can modify the result to include the $L^2$-norm of the residual instead of the $L^\infty$-norm by using Young's inequality in last step of the proof. Further possible modifications are the incorporation of the time-discretization error \cite{HaaO08,JanNP13}, using a weighted norm \cite{Gre05}, or a space-time formulation \cite{UrbP14}. Moreover, the error bound might be extended to nonlinear systems by using similar techniques as in \cite{TadPQ15}. Modifications of the error bound are considered future work.
\end{remark}

Assume again that $\mathcal{A}$ is the generator of a strongly continuous semigroup $\{S(t)\}_{t\geq 0}$. 
For the moment, let us further assume that for some $\Path_i(t)$ we have $\Transformation_i(\Path_i(t)) = S(t)$. 
Using $\tfrac{\mathrm{d}}{\mathrm{d}t} S(t)\Sol = \mathcal{A}S(t)\Sol = S(t)\mathcal{A}\Sol$ for all $\Sol\in D(\mathcal{A})$ (cf.~\cite[Cha.\,1, Thm.\,2.4]{Paz83}), we observe $\FROMstate(\Path,\SolROM) = \PODpathMatrix(\Path)D(\SolROM)\dot{\Path}$ and thus \eqref{eq:ROMtransformed} is given by
\begin{equation}
	\label{eq:ROMtransformationSemigroup}
		\PODmassMatrix(\Path(t))\dot{\SolROM}(t) = 0.
\end{equation}
Clearly, $\SolROM(t) = \SolROM_0$ for all $t\geq 0$ is a solution of \eqref{eq:ROMtransformationSemigroup} that is particularly easy to compute. 

\begin{example}
	\label{ex:advectionEquationROM}
	Let us reconsider the advection equation (cf.~\Cref{ex:advectionEquationPOD}). For the \gls{ROM} \eqref{eq:ROMcoupled} we use the shift operator $\Transformation(\Path)\Sol = \Sol(\cdot - \Path)$ with the usual embedding into $\X = \LTwo{0,1}$, i.e., $\Transformation_i(\Path_i) \vcentcolon =\Transformation(\Path_i)$. It is well-known that $\Transformation$ is a semi-group and $-\partial_\xi\colon\Y\to\X$ is its generator \cite[Sec.~II.2.10]{EngN00}. For the right-hand side in \eqref{eq:ROMcoupled:PhaseConstraint} we obtain
	\begin{align*}
		\FROMpath(\Path,\SolROM) &= \begin{bmatrix}
			\ipX{ \partial_\xi \Transformation(\Path_1)\Mode_1, \partial_\xi \left(\sum_{j=1}^{\dimROM} \Coefficient_j\Transformation(\Path_j)\Mode_j\right)} \\
			\vdots\\
			\ipX{ \partial_\xi \Transformation(\Path_{\dimROM})\Mode_{\dimROM}, \partial_\xi \left(\sum_{j=1}^{\dimROM} \Coefficient_j\Transformation(\Path_j)\Mode_j\right)} 
		\end{bmatrix}\\
		&= \begin{bmatrix}
			\ipX{ \partial_\xi \Transformation(\Path_1)\Mode_1,\partial_\xi \Transformation(\Path_1)\Mode_1} & \cdots & \ipX{ \partial_\xi \Transformation(\Path_1)\Mode_1,\partial_\xi \Transformation(\Path_{\dimROM})\Mode_{\dimROM}}\\
			\vdots & & \vdots\\
			\ipX{ \partial_\xi \Transformation(\Path_{\dimROM})\Mode_{\dimROM},\partial_\xi \Transformation(\Path_1)\Mode_1} & \cdots & \ipX{ \partial_\xi \Transformation(\Path_{\dimROM})\Mode_{\dimROM},\partial_\xi \Transformation(\Path_{\dimROM})\Mode_{\dimROM}}
		\end{bmatrix}\SolROM\\
		&= \PODmassMatrixPath(\Path) D(\SolROM)e,
	\end{align*}
	with $e \vcentcolon= \smash{\begin{bmatrix}
			1 & \cdots & 1
		\end{bmatrix}^T}\in\mathbb{R}^{\dimROM}$. Similarly, $\FROMstate(\Path,\SolROM) = \PODpathMatrix(\Path)D(\SolROM)e$. Substituting \eqref{eq:ROMcoupled:ROM} into \eqref{eq:ROMcoupled:PhaseConstraint} thus implies
	\begin{displaymath}
		D(\SolROM)^T\left(\PODpathMatrix(\Path)^T \PODmassMatrixState^{-1}(\Path)\PODpathMatrix(\Path) - \PODmassMatrixPath(\Path)\right) D(\SolROM)(\dot{\Path}-e) = 0.
	\end{displaymath}
	Note that $\PODpathMatrix(\Path)^T \PODmassMatrixState^{-1}(\Path)\PODpathMatrix(\Path) - \PODmassMatrixPath(\Path)$ is the Schur complement of $\PODmassMatrixState(\Path)$ in $M(\Path)$ defined in \eqref{eq:ROMcoupled:massMatrix}. In particular, the assumptions of \Cref{thm:wellPosedROM} imply $\dot{\Path} =e$ and hence the initial condition $\Path(0) = 0$ implies that $\Transformation(\Path_i)$ equals the semigroup associated with the advection equation \eqref{eq:advectionEquation}.
\end{example}

\section{Connection with Symmetry Reduction}
\label{sec:symmetryReduction}
\allowdisplaybreaks

In this section we investigate how the methodology outlined in \cref{sec:onlinePhase} compares to the symmetry reduction framework as described in \cite{BeyT04,RowM00,RowKML03,OhlR13}. To this end we focus on the special case that all modes are transformed with the same transformation $\Hom$ and the same path $\Path(t)\in\PathSpace=\RealNumbers^q$. 
More precisely, we assume $\hat{\dimROM} = 1$ in \eqref{eq:multframeAnsatzGroup}, i.e., we consider the approximation.
\begin{equation}
	\label{eq:referenceFrameApproach}
	\Sol(t) \approx \sum_{i=1}^{\dimROM} \Coefficient_i(t)\Transformation(\Path(t))\Mode_i.
\end{equation}
In this case, the matrices in the \gls{ROM} \eqref{eq:ROMcoupled}, which we recall here
\begin{subequations}
	\label{eq:ROMcoupledOneTrafo}
	\begin{align}
		\label{eq:ROMcoupled:ROMOneTrafo}\PODmassMatrixState(\Path(t))\dot{\SolROM}(t) + \PODpathMatrix(\Path(t))D(\SolROM(t))\dot{\Path}(t) &= \FROMstate(\Path(t),\SolROM(t)),\\
		\label{eq:ROMcoupled:PhaseConstraintOneTrafo}D(\SolROM(t))^T\PODpathMatrix(\Path(t))^T\dot{\SolROM}(t) + D(\SolROM(t))^T\PODmassMatrixPath(\Path(t))D(\SolROM(t))\dot{\Path}(t) &= D(\SolROM(t))^T\FROMpath(\Path(t),\SolROM(t)).
	\end{align}
\end{subequations}
simplify as follows: The matrices $\PODmassMatrixState$ and $\FROMstate$ are defined as in \cref{sec:onlinePhase} but with $\Transformation(\Path)$ instead of $\Transformation_i(\Path_i)$, $D(\SolROM) \vcentcolon= \Coefficient \otimes I_{q}\in\mathbb{R}^{\dimROM\dimPathSpace\times\dimPathSpace}$, where $\otimes$ denotes the Kronecker product,
\begin{align*}
	\PODpathMatrix(\Path) &\vcentcolon= \begin{bmatrix}\ipX{ \Hom(\Path)\Mode_i,[\Hom'(\Path)\Mode_j]e_1} & \cdots & \ipX{ \Hom(\Path)\Mode_i,[\Hom'(\Path)\Mode_j]e_{\dimPathSpace}} \end{bmatrix}_{i,j=1}^{\dimROM}\in\mathbb{R}^{\dimROM\times\dimROM\dimPathSpace},\\
	\PODmassMatrixPath(\Path) &\vcentcolon= 
	\begin{bmatrix}
		\ipX{ [\Hom'(\Path)\Mode_i]e_1,[\Hom'(\Path)\Mode_j]e_1} & \cdots & \ipX{ [\Hom'(\Path)\Mode_i]e_1,[\Hom'(\Path)\Mode_j]e_{\dimPathSpace}} \\
		\vdots & & \vdots\\
		\ipX{ [\Hom'(\Path)\Mode_i]e_{\dimPathSpace},[\Hom'(\Path)\Mode_j]e_1} & \cdots & \ipX{ [\Hom'(\Path)\Mode_i]e_{\dimPathSpace},[\Hom'(\Path)\Mode_j]e_{\dimPathSpace}}
	\end{bmatrix}
	_{i,j=1}^{\dimROM}\hspace*{-2em} \in \mathbb{R}^{\dimROM\dimPathSpace\times\dimROM\dimPathSpace},\\
	\FROMpath(\Path,\SolROM) &\vcentcolon= 
	\begin{bmatrix} 
		\ipX{ [\Hom'(\Path)\Mode_i]e_1, \F\left(\sum_{j=1}^{\dimROM}\Coefficient_j\Hom(\Path)\Mode_j\right)}\\
		\vdots\\
		\ipX{ [\Hom'(\Path)\Mode_i]e_{\dimPathSpace}, \F\left(\sum_{j=1}^{\dimROM}\Coefficient_j\Hom(\Path)\Mode_j\right)}
	\end{bmatrix}
	_{i=1}^{\dimROM}\hspace*{-1.5em} \in \mathbb{R}^{\dimROM\dimPathSpace}.
\end{align*}
Recall that \eqref{eq:ROMcoupled:PhaseConstraintOneTrafo} is the phase condition that corresponds to minimizing the residual with respect to $\dot{\Path}$ and is therefor referred to as $\PhaseCondition_{\mathrm{Res}}$.

Another common assumption in the symmetry reduction framework is that $\Transformation$ is a group action and the right-hand side $\F$ is equivariant with respect to $\Transformation$, cf.~\Cref{ass:equivarianceGroupAction}.
Furthermore, we assume in the following that $\Hom(\eta)$ is isometric for all $\eta\in\PathSpace$.

\begin{assumption}
	\label{ass:equivarianceGroupAction}
	The right-hand side $\F$ is \emph{equivariant} with respect to $\Transformation$, i.e.,
	\begin{equation}
		\label{eq:equivariance}
		\F(\Transformation(\Path)\Mode) = \Transformation(\Path)\F(\Mode)\qquad\text{for all $\Mode\in\Y$ and $p\in\PathSpace$}.
	\end{equation}
	Moreover the mapping $\Transformation\colon \PathSpace\times \X \to \X$ is a \emph{group action}, i.e.
	\begin{equation}
		\label{eq:groupAction}
		\Transformation(0) = \mathrm{Id}_{\X}\qquad\text{and}\qquad
		\Transformation(\tilde{\Path})\Transformation(\Path)\Mode = \Transformation(\tilde{\Path}+\Path)\Mode\qquad\text{for all $\Mode\in\X$ and $p,\tilde p\in\PathSpace$}.
	\end{equation}
\end{assumption}

\begin{remark}
	\label{rem:equivarianceSemigroup}
	The properties of $\F$ and $\Hom$ stated in \Cref{ass:equivarianceGroupAction} can also be motivated by considering them from the perspective of the semigroup theory.
	In \cref{sec:onlinePhase} we demonstrated that if the right-hand side $\F$ is linear and the corresponding linear operator is the generator of a strongly continuous semigroup, choosing the transformation $\Transformation$ to be the action of this semigroup leads to a particularly simple \gls{ROM}.
	However, in practice we do not expect to have direct access to the semigroup $\{S(t)\}_{t\geq 0}$ or, more generally speaking, to the flow of the differential equation.
	Still, this discussion provides a good starting point for the construction of the transformation operator $\Transformation$, which should reflect the characteristic features of the expected solution behavior as encoded in $\{S(t)\}_{t\geq 0}$. 
	In fact, by \Cref{ass:equivarianceGroupAction} we inherit two important properties of this semigroup.
\end{remark}

\begin{remark}
	\label{rem:pathIndependentROM}
	It is important to note that due to the isometry property of $\Hom(\Path(t))$ and due to the equivariance assumption \eqref{eq:equivariance}, the matrix $\PODmassMatrixState$ and the right-hand side $\FROMstate$ do not depend on the path $\Path$ anymore.
	If additionally, $\Transformation(-\Path)[\Transformation'(\Path)\phi]$ and $\ipX{ [\Transformation'(\Path)\phi]v,[\Transformation'(\Path)\psi]w}$ do not depend on $\Path$ for all $\phi,\psi\in\Y$ and $v,w\in\PathSpace=\RealNumbers^q$, then also $\PODpathMatrix$, $\FROMpath$, and $\PODmassMatrixPath$ are independent from the path. 
	For instance, the shift operator as discussed in \Cref{ex:advectionEquationROM} satisfies these properties.
	In this case, the coefficient matrices $\PODmassMatrixState$, $\PODpathMatrix$, and $\PODmassMatrixPath$ can be precomputed in the offline phase.
	If additionally $\F$ is linear, or more generally, a polynomial mapping, the corresponding reduced operator can also be precomputed in the offline phase and, thus, the \gls{ROM} can be evaluated efficiently without requiring computations that scale with the dimension of the \gls{FOM}.
\end{remark}

Using a single transformation $\Hom$ establishes a \emph{reference frame}, c.f.\ \cref{sec:literatureReview}. 
In more detail assume that we are given a smooth path $\Path$ and that we can split the dynamics as
\begin{equation}
	\label{eq:dynamicSplitting}
	\Sol(t) = \Transformation(\Path(t))\FrozenSolution(t),
\end{equation}
where we refer to $\FrozenSolution$ as the \emph{frozen solution}. 
Especially, if \eqref{eq:groupAction} from \Cref{ass:equivarianceGroupAction} is satisfied, the frozen solution is given by $\FrozenSolution(t) = \Transformation(-\Path(t))\Sol(t)$.
Substituting \eqref{eq:dynamicSplitting} into the evolution equation \eqref{eq:FOMPDE} yields
\begin{align*}
	\Transformation(\Path(t))\dot{\FrozenSolution}(t) + [\Transformation'(\Path)\FrozenSolution(t)]\dot{\Path}(t) &= \F(\Transformation(\Path(t))\FrozenSolution(t)).
\end{align*}
If \Cref{ass:equivarianceGroupAction} is satisfied, we can employ the identities $\F(\Transformation(\Path(t))\FrozenSolution(t))=\Transformation(\Path(t))\F(\FrozenSolution(t))$ and $\Transformation(-\Path(t))\Transformation(\Path(t))=\mathrm{Id}_{\X}$ to arrive at the reference frame equation
\begin{equation}
	\label{eq:referencePDE}
	\dot{\FrozenSolution}(t) = \F(\FrozenSolution(t)) - \Transformation(-\Path(t))[\Transformation'(\Path)\FrozenSolution(t)]\dot{\Path}(t).
\end{equation}
Especially, given a continuously differentiable path $\Path$ the assumptions imply that $\FrozenSolution$ is a solution of \eqref{eq:referencePDE} if and only if $\Sol$ is a solution of \eqref{eq:FOMPDE}, see \cite[Thm.\,2.6]{BeyT04} for further details. 
Based on the reference frame equation \eqref{eq:referencePDE}, we can construct a \gls{ROM} via Galerkin projection onto the span of a suitable orthonormal basis $(\Mode_1,\ldots,\Mode_{\dimROM})$ and obtain 
\begin{align}
	\label{eq:referenceFrameROM}
	\dot\FrozenSolROM_i(t) = \ipX{ \Mode_i,\F\left(\sum_{j=1}^{\dimROM} \FrozenSolROM_j(t)\Mode_j\right)} - \sum_{j=1}^{\dimROM} \ipX{ \Mode_i,\Transformation(-\Path(t))[\Transformation'(\Path)\Mode_j]\dot{\Path}(t)} \FrozenSolROM_j(t)
\end{align}
for $i=1,\ldots,\dimROM$. Note that the approximation is given by $\FrozenSolution(t) \approx \sum_{i=1}^{\dimROM} \FrozenSolROM_i(t)\Mode_i$ with associated residual
\begin{equation}
	\label{eq:residual}
	\mathscr{R}(t) = \sum_{i=1}^{\dimROM} \dot{\FrozenSolROM}(t)\Mode_i - \F\left(\sum_{i=1}^{\dimROM} \FrozenSolROM_i(t)\Mode_i\right) + \sum_{i=1}^{\dimROM} \FrozenSolROM_i(t)\Transformation(-\Path(t))[\Transformation'(\Path(t))\Mode_i]\dot{\Path}(t).
\end{equation}
For convenience we introduce the reduced frozen state $\FrozenSolROM \vcentcolon= \smash{[\FrozenSolROM_1\;\, \cdots\;\, \FrozenSolROM_{\dimROM}]^T}$. 
Using the notation from \eqref{eq:ROMcoupledOneTrafo}, we can write the \gls{ROM} \eqref{eq:referenceFrameROM} of the reference frame equation as
\begin{equation}
	\label{eq:referenceFrameROMMatrix}
	\dot\FrozenSolROM(t) = \FROMstate(\Path(t),\FrozenSolROM(t))-\PODpathMatrix(\Path(t))D(\FrozenSolROM(t))\dot{\Path}(t).
\end{equation}
Thus, we immediately arrive at the following relation between the symmetry reduction \gls{ROM} \eqref{eq:referenceFrameROM} and the continuously optimal \gls{ROM} \eqref{eq:ROMcoupledOneTrafo}.

\begin{lemma}\label{lem:equivalenceSingleFrame}
	Let $(\Mode_1,\ldots,\Mode_{\dimROM})$ be an orthonormal basis of an $\dimROM$-dimensional subspace of $\Y$ and consider the single-frame approximation \eqref{eq:referenceFrameApproach}. 
	If the transformation operator $\Transformation$ is an isometry and satisfies \Cref{ass:TransformationDifferentiability,ass:equivarianceGroupAction}, then \eqref{eq:ROMcoupled:ROMOneTrafo} and \eqref{eq:referenceFrameROM} are equivalent in the sense that for every continuously differentiable path $\Path$, any solution of \eqref{eq:ROMcoupled:ROMOneTrafo} is a solution of \eqref{eq:referenceFrameROM} and vice versa.
\end{lemma}

\Cref{lem:equivalenceSingleFrame} establishes the equivalence between the \gls{ROM} obtained by symmetry reduction and the one obtained by our framework in the case that the same path is chosen for both \glspl{ROM}.
In our framework, the path is fixed by adding the phase condition \eqref{eq:ROMcoupled:PhaseConstraintOneTrafo}, which ensures to minimize the residual.
In the symmetry reduction framework, different phase conditions have been proposed in the literature.
For instance, in \cite{BeyT04} the authors propose to derive a phase condition for the reference frame equation \eqref{eq:referencePDE} by minimizing the temporal change of the frozen solution, \ie, by minimizing $\frac{1}{2}\lVert \dot\FrozenSolution(t)\rVert_\X^2$ over $\dot{\Path}(t)$.
This idea of choosing the path is usually referred to as \emph{freezing}.
The associated phase condition $\PhaseCondition_{\mathrm{freezeFOM}}(\Path,\dot{\Path},\FrozenSolution)$\footnote{The authors of \cite{BeyT04} denote the phase condition $\PhaseCondition_{\mathrm{freezeFOM}}$ as $\PhaseCondition_{\min}$. Since all phase conditions discussed in our exposition are based on a minimization problem, we use a different name here.} is given by the first-order necessary optimality condition 
\begin{equation}
	\label{eq:phaseConditionPDE}
	[\Hom'(\Path(t))\FrozenSolution(t)]^*[\Hom'(\Path(t))\FrozenSolution(t)]\dot{\Path}(t) = [\Hom'(\Path(t))\FrozenSolution(t)]^*\F(\Hom(\Path(t))\FrozenSolution(t)),
\end{equation}
where we use that $\Hom(\Path(t))$ is isometric as well as \Cref{ass:equivarianceGroupAction}. 
In \cite{BeyT04}, the authors propose to discretize the phase condition in space by replacing the occurring operators and the frozen solution $\FrozenSolution$ by their finite difference approximations.
In the context of model reduction, this strategy corresponds to enforcing
\begin{equation*}
	\PhaseCondition_{\mathrm{freeze}}(\Path,\dot{\Path},\SolROM) \vcentcolon= \PhaseCondition_{\mathrm{freezeFOM}}\left(\Path,\dot{\Path},\sum_{i=1}^{\dimROM}\FrozenSolROM_i\Mode_i\right) = 0
\end{equation*}
or equivalently,
\begin{equation}
	\label{eq:phaseConditionROM1}
	D(\FrozenSolROM(t))^T\PODmassMatrixPath(\Path(t))D(\FrozenSolROM(t))\dot{\Path}(t) = D(\FrozenSolROM(t))^T\FROMpath(\Path(t),\FrozenSolROM(t)).
\end{equation}
It is important to note that in general, \eqref{eq:phaseConditionROM1} is not equivalent to the first-order necessary optimality condition for minimizing the temporal change of the reduced state, i.e., by minimizing $\frac{1}{2}\lVert \dot\FrozenSolROM(t)\rVert_2$ over $\dot{\Path}(t)$. This  optimality condition, which we call~$\rom{\PhaseCondition}_{\mathrm{freeze}}$, is given by
\begin{equation}
	\label{eq:phaseConditionROM2}
	D(\FrozenSolROM(t))^T\PODpathMatrix(\Path(t))^T\PODpathMatrix(\Path(t))D(\FrozenSolROM(t))\dot{\Path}(t) = D(\FrozenSolROM(t))^T\PODpathMatrix(\Path(t))^T\FROMstate(\Path(t),\FrozenSolROM(t)).
\end{equation}
The relation between the three different phase conditions $\PhaseCondition_{\mathrm{Res}}$, $\PhaseCondition_{\mathrm{freeze}}$, and $\rom{\PhaseCondition}_{\mathrm{freeze}}$, defined in \eqref{eq:ROMcoupled:PhaseConstraintOneTrafo}, \eqref{eq:phaseConditionROM1}, and \eqref{eq:phaseConditionROM2}, respectively, is provided in the next result.

\begin{theorem}
	\label{thm:phaseCondition}
	Let the assumptions of \Cref{lem:equivalenceSingleFrame} be satisfied. Then, for each $t>0$, the  phase condition $\PhaseCondition_{\mathrm{freeze}}$ given in \eqref{eq:phaseConditionROM1} is the necessary first-order optimality condition for the optimization problem
	\begin{displaymath}
		\min_{\dot{\Path}(t)} \frac{1}{2}\lVert \mathscr{R}(t)\rVert_\X^2+\frac{1}{2}\lVert\dot{\FrozenSolROM}(t)\rVert_2^2,
	\end{displaymath}
	where $\mathscr{R}$ denotes the residual defined in \eqref{eq:residual} and $\dot{\FrozenSolROM}$ is given by \eqref{eq:referenceFrameROMMatrix}.
\end{theorem}

\begin{proof}
	Using \eqref{eq:referenceFrameROMMatrix}, we calculate for $t>0$
	\begin{align*}
		J_{\mathrm{freeze}}(\dot{\Path}(t)) 
		&\vcentcolon= \frac{1}{2} \|\mathscr{R}(t)\|_\X^2 + \frac{1}{2}\|\dot{\FrozenSolROM}(t)\|_2^2 
		= \frac{1}{2} \| \mathscr{R}(t)\|_\X^2 + \frac{1}{2}\left\|\sum_{i=1}^\dimROM\dot{\FrozenSolROM}_i(t)\Mode_i\right\|_\X^2\\
		&= \frac{1}{2}\|\mathscr{R}(t)\|_\X^2 + \frac{1}{2}\left\|\sum_{i=1}^\dimROM \dot{\FrozenSolROM}_i(t)\Mode_i\right\|_\X^2 - \ipX{\mathscr{R}(t), \sum_{i=1}^\dimROM \dot{\FrozenSolROM}_i(t)\Mode_i}\\
		&= \frac{1}{2}\left\| \mathscr{R}(t) - \sum_{i=1}^\dimROM \dot{\FrozenSolROM}_i(t)\Mode_i\right\|_\X^2\\
		&= \frac{1}{2}\left\| \sum_{i=1}^\dimROM \FrozenSolROM_i(t)[\Transformation'(\Path(t))\Mode_i]\dot{\Path}(t) - \F\left(\sum_{j=1}^\dimROM \FrozenSolROM_j(t)\Transformation(\Path(t))\Mode_j\right)\right\|_\X^2,
	\end{align*}
	where the second equality follows from the fact that the Galerkin \gls{ROM} \eqref{eq:referenceFrameROMMatrix} enforces the residual to be orthogonal to $\Span\{\Mode_1,\ldots,\Mode_\dimROM\}$.
	The necessary first-order optimality condition for this minimization problem is obtained by differentiating $J_{\mathrm{freeze}}$ with respect to $\dot{\Path}(t)$ and setting the derivative to zero. 
	The resulting equation is \eqref{eq:phaseConditionROM1} and, thus, the claim follows.
\end{proof}

The different phase conditions together with their associated optimization problems are summarized in \Cref{tab:phaseConditions}. Note that \eqref{eq:ROMcoupled:ROMOneTrafo} implies that \eqref{eq:ROMcoupled:PhaseConstraintOneTrafo} is equivalent to 
\begin{equation}
	\label{eq:phaseConditionROMpsiRes}
	\begin{aligned}
		&D(\SolROM(t))^T\left(\PODmassMatrixPath(\Path(t))-\PODpathMatrix(\Path(t))^T\PODpathMatrix(\Path(t))\right)D(\SolROM(t))\dot{\Path}(t) \\
		&= D(\SolROM(t))^T\left(\FROMpath(\Path(t),\SolROM(t))-\PODpathMatrix(\Path(t))^T\FROMstate(\Path(t),\SolROM(t))\right),
	\end{aligned}
\end{equation}
where we used that $\Transformation(\Path(t))$ is an isometry. In particular, we have $\PhaseCondition_{\mathrm{Res}} = \PhaseCondition_{\mathrm{freeze}} -  \rom{\PhaseCondition}_{\mathrm{freeze}}$.

\begin{table}[ht]
	\centering
	\caption{Phase conditions}
	\label{tab:phaseConditions}
	
	\begin{tabular}{ccc}
		\toprule
		\textbf{Phase condition} & \textbf{Equation} & \textbf{Optimization problem}\\\midrule
		$\PhaseCondition_{\mathrm{Res}} = 0$ & \eqref{eq:ROMcoupled:PhaseConstraintOneTrafo} & $\min_{\dot{\Path}} \frac{1}{2} \|\mathscr{R}\|^2$ \\[.3em]
		$\PhaseCondition_{\mathrm{freeze}} = 0$ & \eqref{eq:phaseConditionROM1} &  $\min_{\dot{\Path}} \frac{1}{2}\|\mathscr{R}\|^2 + \frac{1}{2}\|\dot{\FrozenSolROM}\|^2$ \\[.3em]
		$\rom{\PhaseCondition}_{\mathrm{freeze}} = 0$ & \eqref{eq:phaseConditionROM2} & $\min_{\dot{\Path}} \frac{1}{2}\|\dot{\FrozenSolROM}\|^2$\\\bottomrule
	\end{tabular}
\end{table}

\section{Numerical Examples}
\label{sec:examples}

In this section we illustrate the \gls{ROM} \eqref{eq:ROMcoupled} derived in \cref{sec:onlinePhase} for several test cases. In \cref{subsec:advectionDiffusion}, we discuss the one-dimensional advection-diffusion equation with periodic boundary condition and illustrate the advantages of a restriction of the transformations as outlined in \Cref{rem:issueMultiframeApproach}. A numerical study with non-periodic boundary conditions is performed in \cref{subsec:ADE_nonPeriodic}. For an example that exhibits more than one transport we discuss the linear wave equation in \cref{subsec:waveEquation}. We conclude our numerical results with the nonlinear (viscous) Burgers' equation (cf.\ \cref{subsec:burgers}) and demonstrate the advantages in comparison to standard \gls{POD}. For all our examples with periodic boundary conditions we use the shift operator discussed in \Cref{ex:shiftOperator} for $\Transformation_i$. In particular, \Cref{ass:TransformationDifferentiability} is satisfied provided that the data is sufficiently smooth. Moreover, this immediately implies $\PathSpace_i = \mathbb{R}$, i.e., $\dimPathSpace_i = 1$. For the case of non-periodic boundary conditions the transformation operator has to be modified, which is discussed in \cref{subsec:ADE_nonPeriodic}.

As is common in the model-order reduction literature, we refer to the solution of the spatially discretized equation \eqref{eq:FOMPDE} as the \emph{truth solution} and, by abuse of notation, denote it by~$\Sol$. 
For all examples except for the one in \cref{subsec:ADE_nonPeriodic}, we use a $6$-th order central finite difference scheme (see for instance \cite{For88}) on a grid with $200$ equidistant points for the spatial discretization.
The stencils for the finite-difference matrices $\partial_{\xi} \approx D_{1}$ and $\partial_{\xi}^{2} \approx D_{2}$ approximating the spatial derivatives are given by
\begin{displaymath}
    \left[ -\frac{1}{60}, \frac{3}{20}, -\frac{3}{4}, 0, \frac{3}{4}, -\frac{3}{20}, \frac{1}{60} \right]\qquad\text{and}\qquad
    \left[ \frac{1}{90}, -\frac{3}{20}, \frac{3}{2}, -\frac{49}{18}, \frac{3}{2}, -\frac{3}{20}, \frac{1}{90} \right],
\end{displaymath}
respectively. 
For the time integration, we use the implicit trapezoidal rule with constant stepsize of $\tau=\num{5e-3}$ and we use the \textsc{Matlab} function \texttt{fsolve} with standard tolerances to solve the resulting implicit system of equations. 
The $\LTwo{0,1; \X}$ norm is approximated with the trapezoidal rule, and for a truth solution $\Sol$ with approximation $\SolROMlifted$ we denote by
\begin{displaymath}
    e_{\rel} \vcentcolon= \frac{\Vert \Sol - \SolROMlifted \Vert_{\LTwo{0,1; \X}}}{\Vert \Sol \Vert_{\LTwo{0,1; \X}}}
\end{displaymath}
the relative error. 
The computation of the dominant modes for our approach, i.e., the solution of the minimization problem \eqref{eq:sPODminimization} is performed with the algorithm described in~\cite{ReiSSM18}. 
Consequently, we denote the \gls{ROM} \eqref{eq:ROMcoupled} as \emph{\gls{sPOD}}.

The shift operator's application may require the evaluation of the modes at points not captured by the spatial grid.
In the numerical experiments, we compute these evaluations utilizing polynomial interpolation with cubic Lagrange polynomials.
Furthermore, we also need to evaluate the spatial derivatives of the shifted modes in the online phase, cf.~\Cref{ex:advectionEquationROM}.
To this end, we use the same discretization of the space derivative as in the simulation of the \gls{FOM}.

\paragraph{Code availability} The \textsc{Matlab} source code for the numerical examples can be obtained from the doi \href{https://doi.org/10.5281/zenodo.3902924}{10.5281/zenodo.3902924}.

\subsection{Advection-Diffusion Equation}
\label{subsec:advectionDiffusion}

The one-dimensional advection-diffusion equation with periodic boundary conditions is given by
	\begin{equation}
		\label{eq:advectionDiffusionEquation}
		\left\{\begin{aligned}
			\ParDer{t}{\Sol}(t,\xi) + c \ParDer{\xi}{\Sol}(t,\xi) - \mu \partial_{\xi}^{2} \Sol \left(t,\xi\right) &= 0, & (t,\xi)\in(0,1)\times(0,1),\\
			\Sol(t,0) &= \Sol(t,1), & t\in(0,1),\\
			\Sol(0,\xi) &= \exp\left(-\left(\tfrac{\xi-0.5}{0.1}\right)^2\right), & \xi\in(0,1).
		\end{aligned}\right.
	\end{equation}
The parameters $c$ and $\mu$ denote the transport velocity and the diffusion coefficient, respectively, and are chosen as $c = 1$ and $\mu = \num{0.002}$. 
Choosing $\dimROM = 2$ modes, the algorithm described in \cite{ReiSSM18}, yields a relative offline error of $\num{8.6e-3}$. 

For the online phase (cf.~\cref{sec:onlinePhase}) we compare two different \glspl{ROM}: the first one uses only one path $\Path$ for both modes as described in \eqref{eq:referenceFrameApproach}, i.e., the approximate solution $\SolROMlifted_1$ of \eqref{eq:advectionDiffusionEquation} is given by $\SolROMlifted_1(t) \vcentcolon= \sum_{i=1}^{\dimROM} \Coefficient_i(t)\Transformation(\Path)\Mode_i$. 
The second \gls{ROM} is constructed as proposed in \eqref{eq:ROMcoupled} and the corresponding approximation is denoted by $\SolROMlifted_2$. The relative errors are
\begin{displaymath}
    e_{\rel, \SolROMlifted_1} \approx \num{4.3e-3} \qquad \text{and} \qquad e_{\rel, \SolROMlifted_2} \approx \num{3.6e-2}
\end{displaymath}
and the evolution of the relative error with respect to time is presented in  \Cref{fig:advectionDiffusionEquation:relativeError}. To understand the larger error in the second approximation, we consider the coefficient functions for the second mode for both approaches, which are presented in \Cref{fig:advectionDiffusionEquation:secondCoefficient}. We observe that at time $t \approx 0.45$ the coefficient functions are almost equal to zero. Following the discussion in \Cref{rem:issueMultiframeApproach}, the solution of the \gls{ROM} \eqref{eq:ROMcoupled} in a neighborhood of such a point is ill-conditioned, which results in the weaker approximation quality of $\SolROMlifted_2$.

\begin{figure}
	\centering
	\begin{subfigure}[t]{.45\linewidth}
    	\includegraphics{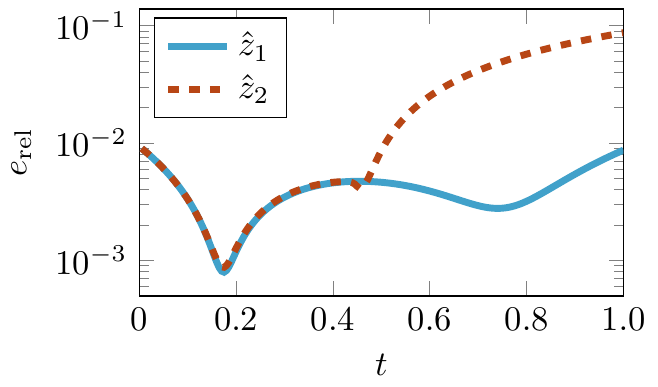}
    	\caption{Time evolution of the relative error for the \glspl{ROM} $\SolROMlifted_1$ and $\SolROMlifted_2$.}
    	\label{fig:advectionDiffusionEquation:relativeError}
    \end{subfigure}\quad
    \begin{subfigure}[t]{.50\linewidth}
    	\includegraphics{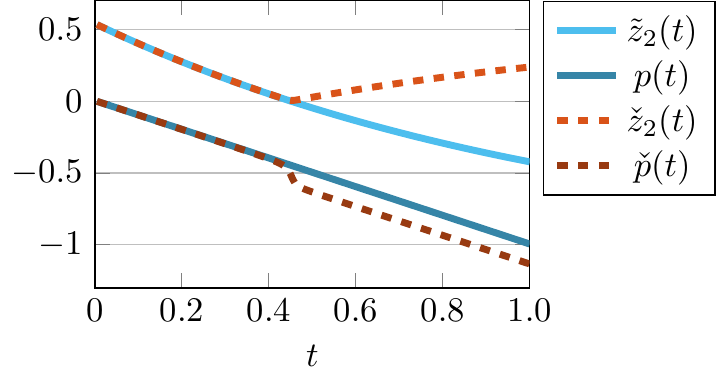}
    	\caption{Evolution of the coefficient for the second mode and the corresponding path for $\SolROMlifted_1$ and $\SolROMlifted_2$.}
    	\label{fig:advectionDiffusionEquation:secondCoefficient}
    \end{subfigure}
    \caption{Advection-diffusion equation -- Comparison of the \gls{ROM} with a single transformation for all modes ($\SolROMlifted_1$, solid blue) and the \gls{ROM} with one transformation per mode ($\SolROMlifted_2$, dashed red).}
    \label{fig:advectionDiffusionEquation}
\end{figure}

The \gls{ROM} is not only able to reproduce the \gls{FOM} with the same set of parameters, but can also accurately predict the behavior of the \gls{FOM} for different parameters. The relative error of the \gls{sPOD} approximation with $\dimROM=2$ for different values of the transport velocity is presented in the left image in \Cref{fig:prediction}. 
Note that \gls{POD} requires $\dimROM=11$ modes to obtain a similar accuracy as the \gls{sPOD} albeit with a larger computational time (cf.~right image in \Cref{fig:prediction}). 
Here, the computational time is the median time of $5$ simulation runs. 
A similar computational time as the \gls{sPOD} is achieved with \gls{POD} with $\dimROM=3$ modes although with a larger relative error. For the comparison of the computational time we use the \textsc{MATLAB} solver \texttt{ode45} (see also the forthcoming discussion in \cref{subsec:waveEquation}) and the \gls{sPOD} approach is implemented utilizing equivariance, see \Cref{rem:pathIndependentROM}. Nevertheless, the numerical simulation code for the \gls{FOM}, \gls{POD}, and \gls{sPOD} are not optimized for performance, which is the main reason for omitting the computational times in the following examples. 
An efficient implementation of our framework is subject to further research.

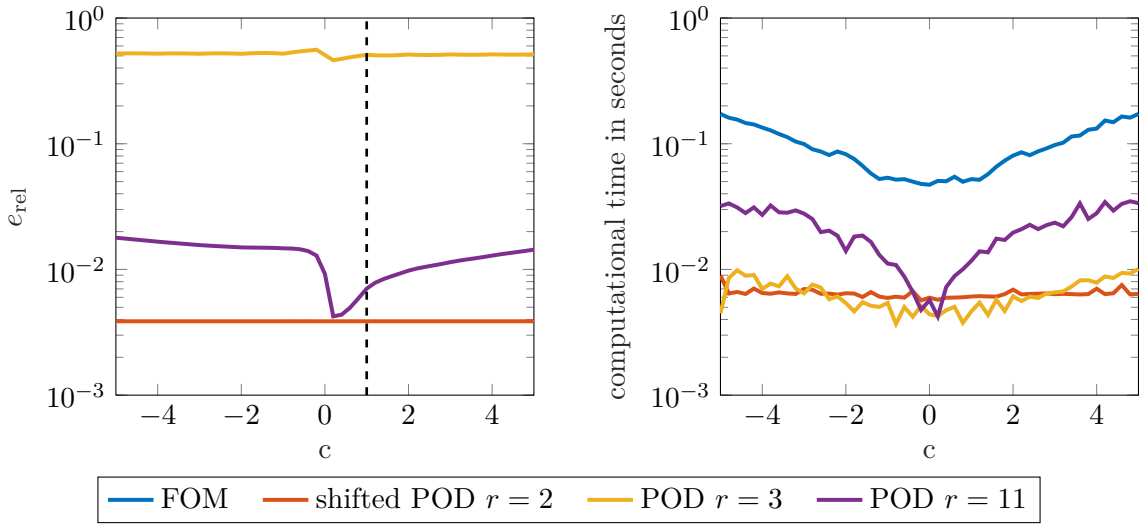
\begin{figure}
	\captionsetup[subfigure]{width=0.9\textwidth}
	\centering
	\newcommand{\tikzLineWidth}{1.5pt}
%
%
\definecolor{mycolor1}{rgb}{0.00000,0.44700,0.74100}%
\definecolor{mycolor2}{rgb}{0.85000,0.32500,0.09800}%
\definecolor{mycolor3}{rgb}{0.92900,0.69400,0.12500}%
\definecolor{mycolor4}{rgb}{0.49400,0.18400,0.55600}%
\begin{tikzpicture}

\begin{axis}[%
width=5.5cm,
height=5cm,
at={(1.733in,1.037in)},
scale only axis,
xmin=-5,
xmax=5,
xlabel style={font=\color{white!15!black}},
xlabel={c},
ymode=log,
ymin=0.001,
ymax=1,
yminorticks=true,
ylabel style={font=\color{white!15!black}},
ylabel={$e_{\rel}$},
axis background/.style={fill=white},
legend style={legend cell align=left, align=left, draw=white!15!black}
]
\addplot [color=mycolor2, line width=\tikzLineWidth]
  table[row sep=crcr]{%
-5	0.00386967847604406\\
-4.8	0.00386966103431156\\
-4.6	0.00386964784785877\\
-4.4	0.00386963529621725\\
-4.2	0.00386963049742059\\
-4	0.00386964119043856\\
-3.8	0.00386966788209981\\
-3.6	0.00386969953776035\\
-3.4	0.0038697133773787\\
-3.2	0.00386970524048692\\
-3	0.00386968779707408\\
-2.8	0.00386969069471797\\
-2.6	0.00386971250853129\\
-2.4	0.00386972907431153\\
-2.2	0.00386975260785472\\
-2	0.00386976957075442\\
-1.8	0.00386978723131165\\
-1.6	0.00386973747845119\\
-1.4	0.00386964357177608\\
-1.2	0.00386962386655272\\
-1	0.00386965225790253\\
-0.8	0.00387045572124401\\
-0.6	0.0038701257003808\\
-0.399999999999999	0.00387027472350457\\
-0.199999999999999	0.00386950182949443\\
0	0.00386949016592376\\
0.199999999999999	0.00386950086562298\\
0.399999999999999	0.00387028256845312\\
0.6	0.00387012576697636\\
0.8	0.00387046180185104\\
1	0.00386965659716536\\
1.2	0.00386962337174073\\
1.4	0.00386964299966729\\
1.6	0.00386973657646748\\
1.8	0.00386978617769173\\
2	0.00386976836274144\\
2.2	0.00386975154447394\\
2.4	0.00386972802477724\\
2.6	0.00386971147757052\\
2.8	0.00386968967715818\\
3	0.00386968676800098\\
3.2	0.00386970419873851\\
3.4	0.00386971233382129\\
3.6	0.00386969849685306\\
3.8	0.00386966683631058\\
4	0.00386964015743337\\
4.2	0.00386962946042555\\
4.4	0.00386963423617974\\
4.6	0.00386964677290076\\
4.8	0.00386965994480861\\
5	0.00386967736038646\\
};

\addplot [color=mycolor3, line width=\tikzLineWidth]
  table[row sep=crcr]{%
-5	0.521749450394667\\
-4.8	0.523306003510603\\
-4.6	0.523752459591373\\
-4.4	0.523141987603984\\
-4.2	0.521964881698926\\
-4	0.520420426112923\\
-3.8	0.522511686076284\\
-3.6	0.523312707576786\\
-3.4	0.52278687296483\\
-3.2	0.521470192072975\\
-3	0.519444954994802\\
-2.8	0.522410539948458\\
-2.6	0.52383084116909\\
-2.4	0.523487034035321\\
-2.2	0.521951058566322\\
-2	0.519042191914825\\
-1.8	0.523846419902207\\
-1.6	0.526749595798606\\
-1.4	0.527047986228229\\
-1.2	0.525256392788406\\
-1	0.520254354848909\\
-0.8	0.531703300812305\\
-0.6	0.542441022201044\\
-0.399999999999999	0.550896691952164\\
-0.199999999999999	0.559269888888441\\
0	0.50770647596969\\
0.199999999999999	0.461613611624179\\
0.399999999999999	0.473705905101852\\
0.6	0.487294610704181\\
0.8	0.498968521174149\\
1	0.510603249838266\\
1.2	0.504987755046588\\
1.4	0.503094936893621\\
1.6	0.504540706044411\\
1.8	0.507642834269075\\
2	0.512574427003936\\
2.2	0.509457227044136\\
2.4	0.507968675724116\\
2.6	0.508419603107708\\
2.8	0.510103241103494\\
3	0.51332771359839\\
3.2	0.511294959612827\\
3.4	0.510151421773639\\
3.6	0.510334924066261\\
3.8	0.511485522141596\\
4	0.513945602529443\\
4.2	0.512532050779049\\
4.4	0.511633496695581\\
4.6	0.511723064950559\\
4.8	0.512611616282318\\
5	0.514645750586387\\
};

\addplot [color=mycolor4, line width=\tikzLineWidth]
  table[row sep=crcr]{%
-5	0.0178788070891656\\
-4.8	0.0176260376920596\\
-4.6	0.0173753291234193\\
-4.4	0.0171276458659947\\
-4.2	0.0168823634873666\\
-4	0.0166319562960081\\
-3.8	0.0164215182577895\\
-3.6	0.0162169697334891\\
-3.4	0.0160175005114719\\
-3.2	0.0158225734914234\\
-3	0.0156220101791785\\
-2.8	0.0154783386222416\\
-2.6	0.0153463431624339\\
-2.4	0.0152225492721226\\
-2.2	0.0151057131929689\\
-2	0.0149795043826727\\
-1.8	0.0149339254838024\\
-1.6	0.0149045682724666\\
-1.4	0.0148763130527432\\
-1.2	0.0148344128473114\\
-1	0.0147258464225205\\
-0.8	0.0146712420634004\\
-0.6	0.0144887034996947\\
-0.399999999999999	0.0139988331293679\\
-0.199999999999999	0.0128881731673737\\
0	0.00921466695656504\\
0.199999999999999	0.00423458845551911\\
0.399999999999999	0.00436660868013633\\
0.6	0.0049560641423865\\
0.8	0.00587939279662032\\
1	0.00704320291030501\\
1.2	0.0078255167406588\\
1.4	0.00837729662674771\\
1.6	0.00884237861433302\\
1.8	0.00929142389458409\\
2	0.00978349760318817\\
2.2	0.0101712331951821\\
2.4	0.0104791308031152\\
2.6	0.0107690558528379\\
2.8	0.0110733041429464\\
3	0.0114225230660041\\
3.2	0.0117358580419074\\
3.4	0.0120116705646205\\
3.6	0.0122782088152387\\
3.8	0.0125633420873211\\
4	0.0128863505622512\\
4.2	0.0131964013120657\\
4.4	0.0134777158157394\\
4.6	0.0137551363218074\\
4.8	0.0140482921119708\\
5	0.0143740217515739\\
};

\addplot [color=black, dashed, line width=1]
  table[row sep=crcr]{%
1 0.001\\
1 1.0\\
};

\end{axis}
\end{tikzpicture}%
\hfill
%
%
\definecolor{mycolor1}{rgb}{0.00000,0.44700,0.74100}%
\definecolor{mycolor2}{rgb}{0.85000,0.32500,0.09800}%
\definecolor{mycolor3}{rgb}{0.92900,0.69400,0.12500}%
\definecolor{mycolor4}{rgb}{0.49400,0.18400,0.55600}%
\begin{tikzpicture}

\begin{axis}[%
width=5.5cm,
height=5cm,
at={(1.733in,1.068in)},
scale only axis,
xmin=-5,
xmax=5,
xlabel style={font=\color{white!15!black}},
xlabel={c},
ymode=log,
ymin=0.001,
ymax=1,
yminorticks=true,
ylabel style={font=\color{white!15!black}},
ylabel={computational time in seconds},
axis background/.style={fill=white},
legend style={
	at={(0.5,-0.1)},
	anchor=north,
	/tikz/column 2/.style={
                column sep=10pt,
            },
    /tikz/column 4/.style={
                column sep=10pt,
            },
    /tikz/column 6/.style={
                column sep=10pt,
            },
},
legend cell align={left},
legend columns=4,
legend to name=legPrediction, 
]
\addplot [color=mycolor1, line width=\tikzLineWidth]
  table[row sep=crcr]{%
-5	0.172618\\
-4.8	0.161033\\
-4.6	0.155636\\
-4.4	0.146133\\
-4.2	0.142342\\
-4	0.134449\\
-3.8	0.127821\\
-3.6	0.119526\\
-3.4	0.112887\\
-3.2	0.104006\\
-3	0.099496\\
-2.8	0.090383\\
-2.6	0.086727\\
-2.4	0.081257\\
-2.2	0.086764\\
-2	0.082532\\
-1.8	0.075644\\
-1.6	0.06664\\
-1.4	0.058015\\
-1.2	0.052387\\
-1	0.053661\\
-0.8	0.051796\\
-0.6	0.052261\\
-0.399999999999999	0.04997\\
-0.199999999999999	0.047901\\
0	0.047168\\
0.199999999999999	0.0506\\
0.399999999999999	0.050294\\
0.6	0.054641\\
0.8	0.0498\\
1	0.052272\\
1.2	0.051698\\
1.4	0.05728\\
1.6	0.065644\\
1.8	0.073397\\
2	0.08044\\
2.2	0.085467\\
2.4	0.081067\\
2.6	0.086982\\
2.8	0.091806\\
3	0.097732\\
3.2	0.101952\\
3.4	0.11422\\
3.6	0.1163\\
3.8	0.129208\\
4	0.131952\\
4.2	0.153055\\
4.4	0.148311\\
4.6	0.164791\\
4.8	0.161293\\
5	0.172637\\
};
\addlegendentry{\textcolor{black}{FOM}}

\addplot [color=mycolor2, line width=\tikzLineWidth]
  table[row sep=crcr]{%
-5	0.008965\\
-4.8	0.006437\\
-4.6	0.006623\\
-4.4	0.00638\\
-4.2	0.007045\\
-4	0.006495\\
-3.8	0.006408\\
-3.6	0.006569\\
-3.4	0.006447\\
-3.2	0.006385\\
-3	0.006988\\
-2.8	0.006893\\
-2.6	0.006401\\
-2.4	0.006425\\
-2.2	0.006302\\
-2	0.006468\\
-1.8	0.006428\\
-1.6	0.006033\\
-1.4	0.006607\\
-1.2	0.00608\\
-1	0.005933\\
-0.8	0.006072\\
-0.6	0.005962\\
-0.399999999999999	0.006731\\
-0.199999999999999	0.005677\\
0	0.005967\\
0.199999999999999	0.00574\\
0.399999999999999	0.005951\\
0.6	0.005975\\
0.8	0.006005\\
1	0.006089\\
1.2	0.006132\\
1.4	0.006096\\
1.6	0.006075\\
1.8	0.006366\\
2	0.006909\\
2.2	0.006331\\
2.4	0.006385\\
2.6	0.006419\\
2.8	0.006451\\
3	0.006354\\
3.2	0.006397\\
3.4	0.006362\\
3.6	0.006326\\
3.8	0.006364\\
4	0.007029\\
4.2	0.006359\\
4.4	0.006464\\
4.6	0.007517\\
4.8	0.006328\\
5	0.006399\\
};
\addlegendentry{\textcolor{black}{shifted POD $\dimROM=2$}}

\addplot [color=mycolor3, line width=\tikzLineWidth]
  table[row sep=crcr]{%
-5	0.004466\\
-4.8	0.008541\\
-4.6	0.009853\\
-4.4	0.008907\\
-4.2	0.009035\\
-4	0.007005\\
-3.8	0.007725\\
-3.6	0.007337\\
-3.4	0.008835\\
-3.2	0.007047\\
-3	0.006455\\
-2.8	0.007591\\
-2.6	0.00717\\
-2.4	0.005793\\
-2.2	0.006108\\
-2	0.005405\\
-1.8	0.004634\\
-1.6	0.005491\\
-1.4	0.005138\\
-1.2	0.005033\\
-1	0.005416\\
-0.8	0.003685\\
-0.6	0.005023\\
-0.399999999999999	0.004205\\
-0.199999999999999	0.005125\\
0	0.004398\\
0.199999999999999	0.004255\\
0.399999999999999	0.004742\\
0.6	0.005044\\
0.8	0.003763\\
1	0.004635\\
1.2	0.005452\\
1.4	0.004345\\
1.6	0.005738\\
1.8	0.004663\\
2	0.006111\\
2.2	0.005605\\
2.4	0.006062\\
2.6	0.005926\\
2.8	0.006443\\
3	0.006572\\
3.2	0.006684\\
3.4	0.007306\\
3.6	0.008182\\
3.8	0.008207\\
4	0.007851\\
4.2	0.008794\\
4.4	0.008552\\
4.6	0.009388\\
4.8	0.009286\\
5	0.010082\\
};
\addlegendentry{\textcolor{black}{POD $\dimROM=3$}}

\addplot [color=mycolor4, line width=\tikzLineWidth]
  table[row sep=crcr]{%
-5	0.031924\\
-4.8	0.033513\\
-4.6	0.031134\\
-4.4	0.028079\\
-4.2	0.031155\\
-4	0.027166\\
-3.8	0.032272\\
-3.6	0.028499\\
-3.4	0.028247\\
-3.2	0.0295\\
-3	0.027808\\
-2.8	0.025169\\
-2.6	0.01982\\
-2.4	0.020346\\
-2.2	0.018561\\
-2	0.014026\\
-1.8	0.018267\\
-1.6	0.018599\\
-1.4	0.016554\\
-1.2	0.013049\\
-1	0.011134\\
-0.8	0.01085\\
-0.6	0.008759\\
-0.399999999999999	0.006632\\
-0.199999999999999	0.004764\\
0	0.005711\\
0.199999999999999	0.004229\\
0.399999999999999	0.007241\\
0.6	0.008862\\
0.8	0.010042\\
1	0.01169\\
1.2	0.013951\\
1.4	0.013663\\
1.6	0.017583\\
1.8	0.017147\\
2	0.019639\\
2.2	0.020887\\
2.4	0.022663\\
2.6	0.020927\\
2.8	0.022507\\
3	0.023568\\
3.2	0.022041\\
3.4	0.026071\\
3.6	0.033519\\
3.8	0.025213\\
4	0.027987\\
4.2	0.034364\\
4.4	0.029358\\
4.6	0.033226\\
4.8	0.034875\\
5	0.0337\\
};
\addlegendentry{\textcolor{black}{POD $\dimROM=11$}}

\end{axis}
\end{tikzpicture}%
\\[.2em]
    \ref{legPrediction}
    \caption{Advection-diffusion equation -- Approximation quality of the \gls{sPOD} reduced model with $\dimROM=2$ and \gls{POD} reduced model with $\dimROM\in\{3,11\}$ for different values of the transport velocity $c$. The dashed line at $c=1$ denotes the parameter that is used to construct the \glspl{ROM}. Left image: relative online error. Right image: median computational time in seconds.}
    \label{fig:prediction}
\end{figure}

\subsection{Non-Periodic Boundary Conditions}
\label{subsec:ADE_nonPeriodic}
To illustrate that the presented framework is not restricted to problems with periodic boundary conditions, we consider the advection-diffusion equation with Dirichlet-Neumann boundary conditions of the form
\begin{equation*}
	\left\{\begin{aligned}
		\ParDer{t}{\Sol}(t,\xi) + c \ParDer{\xi}{\Sol}(t,\xi) - \mu \partial_{\xi}^{2} \Sol \left(t,\xi\right) &= 0, & (t,\xi)\in(0,1)\times(0,1.5),\\
		\Sol(t,0) &= \frac12\exp\left(-\left(\tfrac{t-0.2}{0.03}\right)^2\right), & t\in(0,1),\\
		\partial_\xi \Sol(t,1) &= 0, & t\in(0,1),\\
		\Sol(0,\xi) &= \frac12\exp\left(-\left(\tfrac{\xi-0.5}{0.02}\right)^2\right), & \xi\in(0,1).
	\end{aligned}\right.
\end{equation*}
In the following we use the values $c=1$ and $\mu = 0.001$.
For the spatial discretization we employ a standard central second-order finite difference scheme with mesh width $\Delta \xi = \num{1.25e-3}$.
The solution of the \gls{FOM} and the solution of the corresponding \gls{ROM} with our framework with $\dimROM=4$ and a single transformation operator for all modes is depicted in \Cref{fig:ADEFOM}.

\begin{figure}
	\centering
	\includegraphics{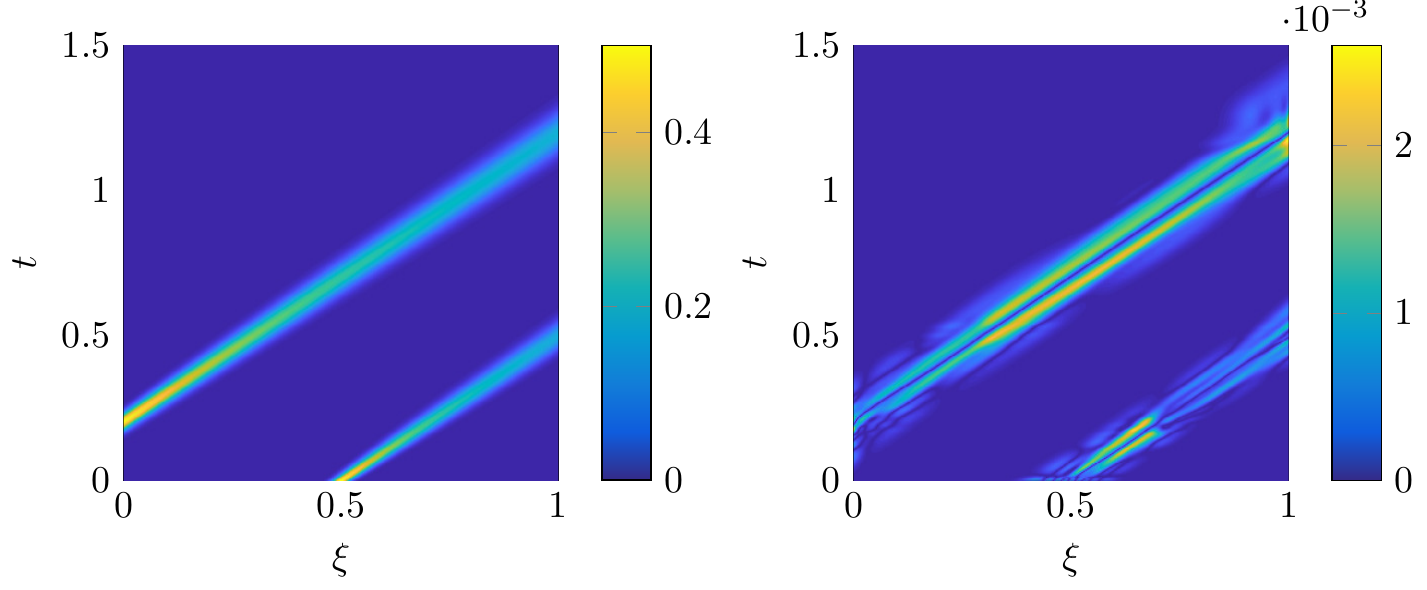}
	\caption{Advection diffusion equation with non-periodic boundary conditions -- \gls{FOM} (left) and absolute error between the \gls{FOM} and the shifted POD \gls{ROM} (right) with $r=4$.}
	\label{fig:ADEFOM}
\end{figure}

Note that a proper treatment of the non-periodic boundary conditions require a modification of the (periodic) shift operator introduced in \Cref{ex:shiftOperator}. To mimic the behavior of the upstream boundary, we propose the following strategy: Instead of using modes on the original domain $\Omega \vcentcolon= [0,1]$, we define the modes on a \emph{virtual} domain $\Omega_{\mathrm{virt}}$ satisfying
\begin{displaymath}
	\{\xi-\Path(t) \mid \xi\in \Omega, t\in [0,T)\} \subseteq \Omega_{\mathrm{virt}}.
\end{displaymath}
The modified shift operator can then be defined as
\begin{equation}
	\label{eq:modifiedShiftOperator}
	\left( \mathcal{T}(\eta)\Mode_i\right)(\cdot) \vcentcolon= \left.\Mode_i(\cdot-\eta)\right|_{\Omega},
\end{equation}
where $\left.\Mode_i(\cdot-\eta)\right|_{\Omega}$ denotes the restriction of $\Mode_i(\cdot-\eta)$ to $\Omega$.
If the path $\Path$ is known, this virtual domain can, for instance, be chosen as
\begin{equation*}	
	\Omega_{\mathrm{virt}}=[-\sup\limits_{t\in [0,T]}\Path(t),\;1-\inf\limits_{t\in [0,T]}\Path(t)].
\end{equation*}
With the parameters defined above we thus obtain $\Omega_{\mathrm{virt}} = [-1.5,1]$. 
At this point, we emphasize that the shift operator~\eqref{eq:modifiedShiftOperator} does formally not fit into the framework presented in this paper. 
In particular, \Cref{ass:offline,ass:TransformationDifferentiability} are not satisfied. 
Nevertheless, the right image in \Cref{fig:ADEFOM} details that our approach with $\dimROM=4$ modes can represent the solution of the \gls{FOM} accurately. It is important to note that the approximation quality of the \gls{sPOD} strongly depends on the accuracy of the \gls{FOM}. If the discretization of the \gls{FOM} is not sufficiently fine, then the online and the offline error of \gls{sPOD} may stagnate. The specific case where the \gls{FOM} is computed with the same spatial and temporal step size is presented in the left plot of \Cref{fig:errorDecay}. In contrast, the approximation with \gls{POD} is less sensitive, albeit with a larger error. A detailed analysis of the impact of the discretization error is subject to further research.

\begin{figure}
	\centering
	\newcommand{\tikzLineWidth}{1.5pt}
%
%
\definecolor{mycolor1}{rgb}{0.00000,0.44700,0.74100}%
\definecolor{mycolor2}{rgb}{0.85000,0.32500,0.09800}%
\definecolor{mycolor3}{rgb}{0.92900,0.69400,0.12500}%
\definecolor{mycolor4}{rgb}{0.49400,0.18400,0.55600}%
\begin{tikzpicture}

\begin{axis}[%
width=3.5cm,
height=5cm,
at={(2.6in,1.239in)},
scale only axis,
xmin=1,
xmax=5,
xlabel style={font=\color{white!15!black}},
xlabel={r},
ymode=log,
ymin=0.0004,
ymax=2,
yminorticks=true,
xtick={1,2,3,4,5},
ylabel style={font=\color{white!15!black}},
ylabel={$e_{\rel}$},
axis background/.style={fill=white},
legend style={
	at={(0.5,-0.1)},
	anchor=north,
	/tikz/column 2/.style={
                column sep=10pt,
            },
},
legend columns=2,
legend cell align={left},
legend to name=legPODDiscretization, 
]


\addplot [color=mycolor1, line width=\tikzLineWidth]
  table[row sep=crcr]{%
1	0.1733937164433\\
2	0.097265400101354\\
3	0.10470033693215\\
4	0.101587155043874\\
5	0.0934144765252241\\
};
\addlegendentry{\textcolor{black}{$\Delta \xi = \Delta t = \num{5.0e-3}$}}

\addplot [color=mycolor2, line width=\tikzLineWidth]
  table[row sep=crcr]{%
1	0.164011983829888\\
2	0.0399181278536028\\
3	0.0307054741286651\\
4	0.0294321934010227\\
5	0.0292226832081255\\
};
\addlegendentry{\textcolor{black}{$\Delta \xi = \Delta t = \num{2.5e-3}$}}

\addplot [color=mycolor3, line width=\tikzLineWidth]
  table[row sep=crcr]{%
1	0.164851840145943\\
2	0.0309552200567041\\
3	0.0118545381552422\\
4	0.00648324625152186\\
5	0.00631515777844174\\
};
\addlegendentry{\textcolor{black}{$\Delta \xi = \Delta t = \num{1.25e-3}$}}

\addplot [color=mycolor4, line width=\tikzLineWidth]
  table[row sep=crcr]{%
1	0.176292836769453\\
2	0.0321653342815226\\
3	0.0113568210336846\\
4	0.00458820012288318\\
5	0.00424445827103984\\
};
\addlegendentry{\textcolor{black}{$\Delta \xi = \Delta t = \num{6.25e-4}$}}


\addplot [color=mycolor1, dashed, line width=\tikzLineWidth]
  table[row sep=crcr]{%
1	0.170049815100066\\
2	0.03861940958073\\
3	0.0121644295594069\\
4	0.00615899421411685\\
5	0.00643046711291943\\
};

\addplot [color=mycolor2, dashed, line width=\tikzLineWidth]
  table[row sep=crcr]{%
1	0.163175136405693\\
2	0.0298693531377606\\
3	0.00709528186321321\\
4	0.00118632433667488\\
5	0.000730337970134829\\
};

\addplot [color=mycolor3, dashed, line width=\tikzLineWidth]
  table[row sep=crcr]{%
1	0.163172570809012\\
2	0.0296030028332028\\
3	0.00687234050519471\\
4	0.00106494850897536\\
5	0.000440999326123829\\
};

\addplot [color=mycolor4, dashed, line width=\tikzLineWidth]
  table[row sep=crcr]{%
1	0.163323227084295\\
2	0.0296605881007445\\
3	0.00689073835791871\\
4	0.0010954071736604\\
5	0.000485847157369381\\
};

\end{axis}
\end{tikzpicture}%
%
%
\definecolor{mycolor1}{rgb}{0.00000,0.44700,0.74100}%
\definecolor{mycolor2}{rgb}{0.85000,0.32500,0.09800}%
\definecolor{mycolor3}{rgb}{0.92900,0.69400,0.12500}%
\definecolor{mycolor4}{rgb}{0.49400,0.18400,0.55600}%
\definecolor{mycolor5}{rgb}{0.46600,0.67400,0.18800}%
\definecolor{mycolor6}{rgb}{0.30100,0.74500,0.93300}%
\definecolor{mycolor7}{rgb}{0.63500,0.07800,0.18400}%
\begin{tikzpicture}

\begin{axis}[%
width=7.5cm,
height=5cm,
at={(2.6in,1.239in)},
scale only axis,
xmin=1,
xmax=35,
xlabel style={font=\color{white!15!black}},
xlabel={r},
ymode=log,
ymin=0.0004,
ymax=2,
yminorticks=true,
xtick={1,5,10,15,20,25,30,35},
ylabel style={font=\color{white!15!black}},
axis background/.style={fill=white},
legend style={legend cell align=left,align=left,draw=white!15!black}
]
\addplot [color=mycolor1,dashed,line width=\tikzLineWidth]
  table[row sep=crcr]{%
1	0.909415380122657\\
2	0.829310286735522\\
3	0.753975597759961\\
4	0.685807350312545\\
5	0.619219035400633\\
6	0.564537327926101\\
7	0.507427237371515\\
8	0.461437614286522\\
9	0.418251806645735\\
10	0.370774598856815\\
11	0.328175579770978\\
12	0.293199746449194\\
13	0.256580344534409\\
14	0.224618889898036\\
15	0.197460477787164\\
16	0.168573482335001\\
17	0.147834838710071\\
18	0.127502016343999\\
19	0.105210835714001\\
20	0.0917113903364226\\
21	0.0764438229946202\\
22	0.0628459578756402\\
23	0.0528807816751659\\
24	0.0429161460070848\\
25	0.0357067544059695\\
26	0.0295066025065821\\
27	0.0228767275031773\\
28	0.0190289117203847\\
29	0.0152801408983398\\
30	0.0117070128187686\\
31	0.00965355730936967\\
32	0.00726335899813329\\
33	0.00579770206992276\\
34	0.00455035842524842\\
35	0.0033029917284139\\
};

\addplot [color=mycolor1,solid,line width=\tikzLineWidth]
  table[row sep=crcr]{%
1	0.96190285368787\\
2	0.97157669088934\\
3	0.827754032246713\\
4	1.02746574494472\\
5	0.693445565870332\\
6	0.731941384246162\\
7	0.576593200089677\\
8	0.554989348196029\\
9	0.601961934455632\\
10	0.482779664700689\\
11	0.414029470057212\\
12	0.44410747626943\\
13	0.371560037154013\\
14	0.29561681809536\\
15	0.310399296924709\\
16	0.251576347809057\\
17	0.224061066124014\\
18	0.240055605591094\\
19	0.15777763632189\\
20	0.169248816815875\\
21	0.129773820647934\\
22	0.0898598624384707\\
23	0.0924091263410172\\
24	0.0691993382080776\\
25	0.0519731776572646\\
26	0.0593165581517272\\
27	0.0344341568889288\\
28	0.0338849816893005\\
29	0.0259421073883007\\
30	0.0166849882653995\\
31	0.0162975936854726\\
32	0.0116814802825311\\
33	0.00766744866508789\\
34	0.00821802689141089\\
35	0.0049840367097226\\
};

\addplot [color=mycolor2,dashed,line width=\tikzLineWidth]
  table[row sep=crcr]{%
1	0.91041432888373\\
2	0.830413549029162\\
3	0.755327274413465\\
4	0.687213049484557\\
5	0.620850033880685\\
6	0.566245019577174\\
7	0.509131747805039\\
8	0.463302782634509\\
9	0.42005429708483\\
10	0.372597677584885\\
11	0.329961326039619\\
12	0.294936463637615\\
13	0.258321468324071\\
14	0.226236647678954\\
15	0.199179873553559\\
16	0.170409989947784\\
17	0.14964753638295\\
18	0.129343115985359\\
19	0.107191697364085\\
20	0.0935735201525664\\
21	0.0784582389466675\\
22	0.0648465941642508\\
23	0.0547910565123253\\
24	0.0447195920430649\\
25	0.0375301188126047\\
26	0.0311784118924273\\
27	0.0244147284029046\\
28	0.0203934854397404\\
29	0.0165166356595484\\
30	0.012912955209081\\
31	0.0107163981047004\\
32	0.00810788604919748\\
33	0.00666713193628053\\
34	0.00526779280363734\\
35	0.00387957301479616\\
};

\addplot [color=mycolor2,solid,line width=\tikzLineWidth]
  table[row sep=crcr]{%
1	0.958135795334978\\
2	0.971681271210583\\
3	0.826340668944206\\
4	1.02490547565869\\
5	0.689788487306831\\
6	0.731531993355914\\
7	0.569077536814598\\
8	0.543346186995404\\
9	0.595816834601666\\
10	0.467023461187599\\
11	0.397419159426973\\
12	0.416771205126263\\
13	0.348620600781625\\
14	0.277578053920266\\
15	0.29559215672161\\
16	0.229842444702444\\
17	0.199431077564544\\
18	0.220013393845836\\
19	0.142711967707022\\
20	0.148180244965856\\
21	0.112992556488889\\
22	0.0834710483552989\\
23	0.0856879522075152\\
24	0.0610445329357527\\
25	0.0497710733379413\\
26	0.0538788345873244\\
27	0.031927677680155\\
28	0.0316601758700189\\
29	0.0226113497283592\\
30	0.0163672405548336\\
31	0.0177793879347858\\
32	0.0105943621498518\\
33	0.00859707130651554\\
34	0.0082867262135975\\
35	0.00487023032331246\\
};

\addplot [color=mycolor3,dashed,line width=\tikzLineWidth]
  table[row sep=crcr]{%
1	0.91085768367668\\
2	0.830858632081634\\
3	0.755840131887677\\
4	0.687709449857894\\
5	0.621406422140489\\
6	0.566816481634193\\
7	0.509671526522443\\
8	0.463873033976162\\
9	0.420581055199293\\
10	0.373128360900897\\
11	0.330427468104809\\
12	0.295369202111756\\
13	0.258751296791437\\
14	0.226556075638889\\
15	0.199522963760888\\
16	0.170780696514923\\
17	0.14995712632414\\
18	0.129644084117887\\
19	0.107542988947012\\
20	0.0938794501309968\\
21	0.0788409895605112\\
22	0.0651813936953892\\
23	0.0551058506206821\\
24	0.0450451776244294\\
25	0.0378403342350466\\
26	0.0314692769876372\\
27	0.0246757363893513\\
28	0.0206319345764415\\
29	0.0167481605310792\\
30	0.0131161914670331\\
31	0.01088072615451\\
32	0.00826432296242198\\
33	0.00680961123913661\\
34	0.00539024108621366\\
35	0.00398683875728283\\
};

\addplot [color=mycolor3,solid,line width=\tikzLineWidth]
  table[row sep=crcr]{%
1	0.951091593913338\\
2	0.968531268233114\\
3	0.821439002052244\\
4	1.01951839007955\\
5	0.681602386907816\\
6	0.728592275719246\\
7	0.55721224433734\\
8	0.525074969017237\\
9	0.584343911083423\\
10	0.444998670418648\\
11	0.376278104265883\\
12	0.389070308152803\\
13	0.319984999916782\\
14	0.263081150801019\\
15	0.283556283455571\\
16	0.212542911984088\\
17	0.185754822698291\\
18	0.207978812118095\\
19	0.13693090209214\\
20	0.139515551777179\\
21	0.107522739097109\\
22	0.0838040154074598\\
23	0.0859218654185885\\
24	0.061943163189415\\
25	0.0504099729914538\\
26	0.0547223114019383\\
27	0.0335949959806593\\
28	0.0331233961538724\\
29	0.0245379563097291\\
30	0.0173700304153012\\
31	0.0185241677435348\\
32	0.0117231544359658\\
33	0.00924439991016118\\
34	0.00896721133995809\\
35	0.00536188341833267\\
};

\addplot [color=mycolor4,dashed,line width=\tikzLineWidth]
  table[row sep=crcr]{%
1	0.911075084113728\\
2	0.831075124999298\\
3	0.756088215396223\\
4	0.687950513919181\\
5	0.621678519946742\\
6	0.5671081882598\\
7	0.509953325957439\\
8	0.464172429340163\\
9	0.420876174943849\\
10	0.373443697428095\\
11	0.330721365337574\\
12	0.295661749309639\\
13	0.259059847221284\\
14	0.226812211546463\\
15	0.199801690041569\\
16	0.171058165612875\\
17	0.150199253229078\\
18	0.129866779412429\\
19	0.107786403037003\\
20	0.094101487599658\\
21	0.0790713917460964\\
22	0.0653833229838999\\
23	0.0552959516596692\\
24	0.0452313754705128\\
25	0.0380022477097394\\
26	0.0316260820252397\\
27	0.024811835290794\\
28	0.020758645176565\\
29	0.0168599775863922\\
30	0.0132137269590915\\
31	0.0109661666443149\\
32	0.00833859044879242\\
33	0.00687576059918203\\
34	0.00544658713162226\\
35	0.00403390546924495\\
};

\addplot [color=mycolor4,solid,line width=\tikzLineWidth]
  table[row sep=crcr]{%
1	0.943056391691711\\
2	0.960395760104761\\
3	0.817883786388785\\
4	1.01143182612718\\
5	0.677281606691862\\
6	0.732530997668582\\
7	0.555777318603815\\
8	0.519142404656802\\
9	0.584082690439309\\
10	0.439405846571939\\
11	0.381810273759914\\
12	0.39754592690773\\
13	0.325663792105553\\
14	0.281438915555871\\
15	0.295770684894168\\
16	0.235525198137146\\
17	0.200441563888915\\
18	0.222598938842025\\
19	0.158369659383044\\
20	0.156086689351049\\
21	0.13283720335107\\
22	0.0966142455623539\\
23	0.0977433243134553\\
24	0.077373800923692\\
25	0.0570204654619089\\
26	0.0665465225029345\\
27	0.0410111045963215\\
28	0.0396086705564296\\
29	0.0323872320672266\\
30	0.0207823541466605\\
31	0.0213259583715062\\
32	0.0153029589272204\\
33	0.0110537160336895\\
34	0.0110843698953056\\
35	0.00677622512414458\\
};

\end{axis}
\end{tikzpicture}%
\\[.2em]
	\ref{legPODDiscretization} 
	\caption{Advection diffusion equation with non-periodic boundary conditions -- Decay of the relative error for the shifted POD \gls{ROM} (left) and the POD \gls{ROM} (right). The solid lines represent the online error, while the dashed lines represent the offline error.}
	\label{fig:errorDecay}
\end{figure}
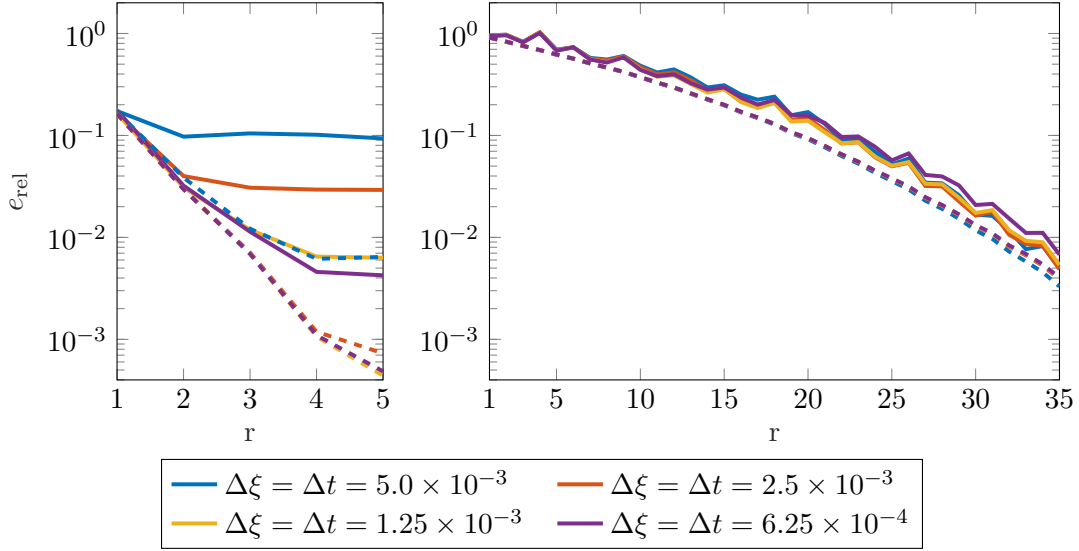	

\subsection{Linear Wave Equation}
\label{subsec:waveEquation}

We consider the one-dimensional acoustic wave equation with periodic boundary conditions as discussed in \Cref{ex:waveEquation}. For the offline phase, we evaluate the analytical solution \eqref{eq:linWaveSolution} on an equally-distributed space-time grid with $200$ points in space and time, respectively. We use the algorithm from \cite{ReiSSM18} to compute $\dimROM=2$ modes with a relative offline approximation error of $\num{7e-16}$. Based on the two dominant modes determined in the offline stage, the \gls{ROM} \eqref{eq:ROMcoupled} yields an online error of $e_{\rel} \approx \num{3.2e-10}$ and thus accurately represents the original equation.

Let us emphasize that our \gls{ROM} not only achieves a considerable reduction in the dimension of the spatial variable, but also benefits the time integration. To detail this, we compare the numerical solution of the \gls{FOM}, a \gls{POD} reduced model and our approach with the \textsc{Matlab} solvers \texttt{ode45} and \texttt{ode23}. Both solvers choose an adaptive step size and the corresponding numbers for the three different models are presented in \Cref{tab:waveEquationTimeSteps} and \Cref{fig:odeSolver}. 
\begin{table}[ht]
	\centering
	\caption{Number of adaptive timesteps for the \gls{FOM}, the \gls{POD} reduced model, and the \gls{sPOD} reduced model}
	\label{tab:waveEquationTimeSteps}
	
	\begin{tabular}{lccc}
		\toprule
		& \gls{FOM} & \gls{POD} & \gls{sPOD} \\\midrule
		\texttt{ode45} & 202 & 51 (\SI{25}{\percent}) & 15 (\SI{7}{\percent}) \\
		\texttt{ode23} & 189 & 194 (\SI{103}{\percent}) & 16 (\SI{8}{\percent}) \\\bottomrule
	\end{tabular}
\end{table}
\begin{figure}
	\captionsetup[subfigure]{width=0.9\textwidth}
	\centering
    \begin{subfigure}{0.45\textwidth}
        \definecolor{myblue}{rgb}{0.30100,0.74500,0.93300}%
\definecolor{myred}{rgb}{0.85000,0.32500,0.09800}%
\definecolor{mygreen}{rgb}{0.09800,0.85000, 0.32500}%
\begin{tikzpicture}
        \begin{axis}[
            xlabel=$t$,
            width = 6.8cm,
            height = 4.8cm,
            xmin = 0,
            xmax = 1,
            ymin = 0,
            ymax = 250,
            ylabel={number of time steps},
			legend entries = {FOM, shifted POD, POD},
			legend pos = north west,
			legend cell align={left},
			legend style={column sep=3pt}
            ]
            
            \addplot+[color=myblue, only marks, mark options={fill=myblue}, mesh/ordering=y varies]
              table[row sep=crcr]{%
0	1\\
2.34296745301183e-05	2\\
0.00014057804718071	3\\
0.000726319910433667	4\\
0.00365502922669845	5\\
0.0182985758080224	6\\
0.0380243751953939	7\\
0.0603569118516186	8\\
0.0741951511126449	9\\
0.0880333903736712	10\\
0.0963430747962767	11\\
0.104652759218882	12\\
0.113200069910291	13\\
0.1210879192762	14\\
0.128631944541123	15\\
0.136503783937339	16\\
0.143812103278498	17\\
0.151303654432105	18\\
0.158257730466903	19\\
0.165457875716229	20\\
0.172335126782281	21\\
0.179372398697541	22\\
0.18608351067289	23\\
0.19285321684096	24\\
0.199801373488182	25\\
0.206283164632595	26\\
0.21303851691158	27\\
0.219719488329683	28\\
0.226062747392603	29\\
0.232731962125451	30\\
0.239280052276379	31\\
0.245659300131262	32\\
0.252145566683731	33\\
0.258657504328329	34\\
0.266297604252848	35\\
0.272476507921227	36\\
0.278961858546461	37\\
0.286002237643461	38\\
0.293042616740461	39\\
0.300519599631059	40\\
0.308635781774798	41\\
0.316751963918538	42\\
0.324825833156915	43\\
0.333084946971379	44\\
0.342404681487127	45\\
0.351047295408205	46\\
0.359689909329282	47\\
0.368219602513591	48\\
0.376307711016645	49\\
0.384363940757639	50\\
0.393096023128906	51\\
0.400882061157675	52\\
0.408668099186444	53\\
0.416467711630223	54\\
0.423996159004309	55\\
0.432040649014777	56\\
0.43899001724828	57\\
0.445939385481784	58\\
0.452896517400531	59\\
0.460443575365851	60\\
0.466846415393488	61\\
0.473249255421124	62\\
0.479646133649235	63\\
0.485130565087384	64\\
0.490614996525533	65\\
0.496451658528039	66\\
0.500525340876191	67\\
0.504599023224342	68\\
0.509589980519093	69\\
0.514969838903639	70\\
0.5205841002853	71\\
0.526740861325346	72\\
0.53325832660179	73\\
0.539413587087312	74\\
0.54575593391476	75\\
0.552098280742208	76\\
0.558643771785314	77\\
0.564941149927239	78\\
0.571238528069164	79\\
0.577928710694308	80\\
0.584114932778536	81\\
0.590301154862764	82\\
0.597058785919989	83\\
0.604072162946646	84\\
0.61006656107042	85\\
0.616331842750818	86\\
0.623101424729318	87\\
0.628731593959663	88\\
0.63483766448397	89\\
0.640565537983144	90\\
0.645875636879603	91\\
0.650779502293467	92\\
0.655683367707332	93\\
0.660431916587598	94\\
0.665180465467864	95\\
0.669676720628999	96\\
0.674172975790135	97\\
0.678592629271885	98\\
0.683012282753634	99\\
0.688162455470161	100\\
0.692487572163397	101\\
0.696982329240772	102\\
0.701088423455996	103\\
0.70521591161875	104\\
0.708682889879324	105\\
0.712149868139898	106\\
0.715850112352783	107\\
0.719550356565667	108\\
0.723097586396091	109\\
0.726240154945931	110\\
0.729382723495772	111\\
0.732424007096681	112\\
0.735625295881972	113\\
0.738826584667262	114\\
0.742012440093052	115\\
0.745566376734474	116\\
0.748558350317707	117\\
0.751546386929463	118\\
0.754940467745236	119\\
0.757985240392548	120\\
0.760919953734317	121\\
0.763869164757368	122\\
0.766898775737482	123\\
0.769762733608781	124\\
0.772768642975562	125\\
0.775704923958332	126\\
0.778611239800412	127\\
0.781678521035178	128\\
0.784534434900302	129\\
0.787474748767915	130\\
0.790391924722437	131\\
0.7932714981379	132\\
0.796276669189934	133\\
0.799147720765215	134\\
0.802125431152955	135\\
0.805014613926733	136\\
0.807915474555853	137\\
0.810912554724385	138\\
0.813781380763196	139\\
0.816776888843615	140\\
0.819673047767486	141\\
0.822590245833747	142\\
0.825547139057891	143\\
0.828432282005649	144\\
0.831433695335654	145\\
0.834315316280378	146\\
0.837256060053218	147\\
0.840212358422677	148\\
0.843101679792475	149\\
0.846114046700319	150\\
0.849006952823278	151\\
0.851960247172008	152\\
0.854890356643117	153\\
0.857796470295356	154\\
0.860848297579002	155\\
0.863734130764473	156\\
0.866708924454355	157\\
0.869648117834785	158\\
0.872560107690408	159\\
0.875592119995512	160\\
0.878491808779681	161\\
0.881468314822221	162\\
0.884391897455914	163\\
0.887318576670886	164\\
0.890357984299303	165\\
0.89325477744458	166\\
0.896247016899707	167\\
0.899182355719339	168\\
0.902115785023046	169\\
0.905130865539729	170\\
0.908041033902803	171\\
0.911040910314651	172\\
0.913964785354663	173\\
0.91691068223196	174\\
0.919938318467552	175\\
0.92284694631966	176\\
0.925852337356077	177\\
0.928790135121817	178\\
0.931742104115314	179\\
0.934752404578226	180\\
0.937673412169376	181\\
0.940692009828013	182\\
0.943621530581047	183\\
0.946582439547039	184\\
0.949610863092485	185\\
0.952530441940939	186\\
0.955545993452585	187\\
0.958491438400819	188\\
0.961455848030483	189\\
0.964473477795565	190\\
0.967404351314625	191\\
0.970432929837764	192\\
0.97337267209438	193\\
0.976342384611744	194\\
0.979383881313634	195\\
0.982312721186325	196\\
0.985333149802107	197\\
0.988291129894004	198\\
0.991260888353798	199\\
0.994298080333055	200\\
0.997237633327454	201\\
1	202\\
};
            \addplot+[color=red, only marks, mark options={fill=red}, mesh/ordering=y varies]
  table[row sep=crcr]{%
0	1\\
0.000201075859380892	2\\
0.00120645515628535	3\\
0.00623335164080764	4\\
0.0313678340634191	5\\
0.131367834063419	6\\
0.231367834063419	7\\
0.331367834063419	8\\
0.431367834063419	9\\
0.531367834063419	10\\
0.631367834063419	11\\
0.731367834063419	12\\
0.831367834063419	13\\
0.931367834063419	14\\
1	15\\
};
            \addplot+[mygreen, color=mygreen, only marks, mark=diamond*, mark options={fill=mygreen}, mesh/ordering=y varies]
             table[row sep=crcr]{%
0	1\\
0.000510566136167275	2\\
0.00306339681700365	3\\
0.0158275502211855	4\\
0.0363311067750301	5\\
0.0588676859995937	6\\
0.0797269176651853	7\\
0.102310853401465	8\\
0.123376182955344	9\\
0.146107974394122	10\\
0.167328943658313	11\\
0.189430464364618	12\\
0.21078832652398	13\\
0.232173636561374	14\\
0.253585786539226	15\\
0.274447268079206	16\\
0.295929299646651	17\\
0.316537336796379	18\\
0.338045381931398	19\\
0.358430487301783	20\\
0.379977082662026	21\\
0.400267560134332	22\\
0.421821678411442	23\\
0.4420154352197	24\\
0.463589848660595	25\\
0.483740498631899	26\\
0.505315909892181	27\\
0.525430285711517	28\\
0.547017689823406	29\\
0.567104834171283	30\\
0.588690185023562	31\\
0.608777287234278	32\\
0.630368702995129	33\\
0.650421691665185	34\\
0.672011885873753	35\\
0.692088759234685	36\\
0.713682424188005	37\\
0.733714569587661	38\\
0.7553071471073	39\\
0.775385933757358	40\\
0.796978664906405	41\\
0.816990122537594	42\\
0.83858523184113	43\\
0.858665861027583	44\\
0.880261414161089	45\\
0.900248023916233	46\\
0.92186379165163	47\\
0.94194677934989	48\\
0.963586032443442	49\\
0.983551374061858	50\\
1	51\\
};
        \end{axis}
    \end{tikzpicture}%
        \caption{\textsc{Matlab} solver \texttt{ode45}}
        \label{fig:ode45}
    \end{subfigure}\qquad
    \begin{subfigure}{0.45\textwidth}
        \input{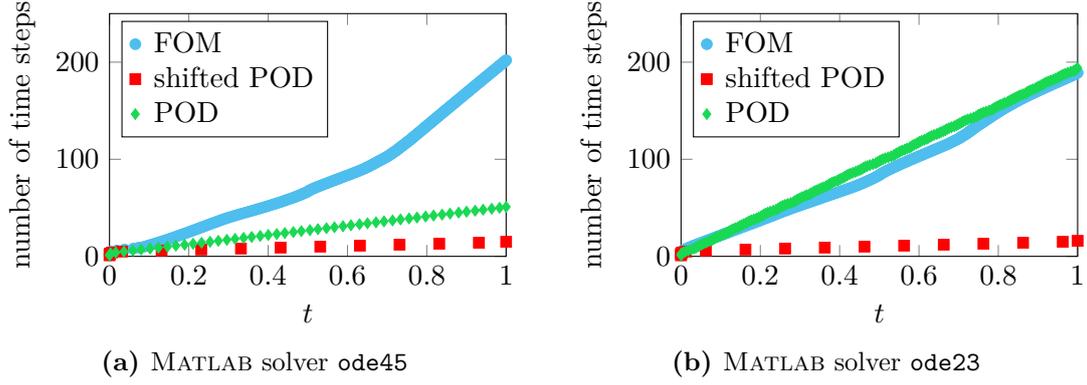}
        \caption{\textsc{Matlab} solver \texttt{ode23}}
        \label{fig:ode23}
    \end{subfigure}
    \caption{Linear wave equation -- Number of time steps chosen by an adaptive time integrator in the computation of the \gls{FOM} (blue circles), the \gls{sPOD} reduced model (red squares), and \gls{POD} reduced model (green diamonds).}
    \label{fig:odeSolver}
\end{figure}
Clearly, our approach requires considerably less time steps for the numerical integration compared to the \gls{FOM} and the \gls{POD} reduced model. The main reason for this behavior is the fact that the reduced state $\SolROM$ in our framework is changing only slightly, i.e., $\|\dot{\SolROM}(t)\|_2$ is small. 
In view of the different phase conditions (cf.~\Cref{tab:phaseConditions}) this is interesting, since our \gls{ROM} is based on the minimization of the residual.

\subsection{Burgers' Equation}
\label{subsec:burgers}
We consider the one-dimensional (viscous) Burgers' equation with periodic boundary conditions, given by
\begin{equation} 
	\label{eq:burgersEquation}
    \left\{
        \begin{aligned}
            \ParDer{t}{\Sol} \left( t, \xi \right) &= \mu \partial_{\xi}^{2} \Sol \left( t, \xi \right) - \Sol \left( t, \xi \right)\ParDer{\xi}{\Sol} \left( t, \xi \right), & \left( t, \xi \right) \in \left( 0, 1 \right) \times \left( 0, 1 \right), \\
            \Sol \left( t, 0 \right) &= \Sol \left( t, 1 \right), & t \in \left( 0, 1 \right), \\
            \Sol \left( 0, \xi \right) &= \exp \left( \left( \tfrac{\xi - 0.5}{0.1} \right)^2 \right), & \xi \in \left( 0, 1 \right),
        \end{aligned}
    \right.
\end{equation}
as a nonlinear test case for the methodology presented within this paper. For our experiments we use the viscosity parameter $\mu = \num{2e-3}$.
The solution is presented in \Cref{fig:burgers:FOM}.
As for the other examples, we use the shift operator $\Transformation \left( \Path \right) \Sol = \Sol \left( \cdot - \Path \right)$ from \Cref{ex:advectionEquationSPOD} as transformation operator, and transform each mode with the same transformation, i.e., we use the approximation \eqref{eq:referenceFrameApproach}. 
Let us emphasize that the right-hand side of Burgers' equation is equivariant with respect to the shift operator (cf.~\Cref{ass:equivarianceGroupAction}) and hence the \gls{ROM} simplifies as described in \Cref{rem:pathIndependentROM}. Note that we can exploit the polynomial form of the nonlinear right-hand side to eliminate the dependency of the original space by computing the reduced right-hand side via
{\small
\begin{multline*}
    \FROMstate(\SolROM) = -
    \begin{bmatrix}
        \ipX{ \Mode_1, \Mode_1 \ParDer{\xi}{\Mode_1} } &
        \ldots &
        \ipX{ \Mode_1, \Mode_1 \ParDer{\xi}{\Mode_{\dimROM}} } &
        \ipX{ \Mode_1, \Mode_2 \ParDer{\xi}{\Mode_1} } &
        \ldots &
        \ipX{ \Mode_1, \Mode_{\dimROM} \ParDer{\xi}{\Mode_{\dimROM}} } \\
        \vdots &
        &
        \vdots &
        \vdots &
        &
        \vdots \\
        \ipX{ \Mode_{\dimROM}, \Mode_1 \ParDer{\xi}{\Mode_1} } &
        \ldots &
        \ipX{ \Mode_{\dimROM}, \Mode_1 \ParDer{\xi}{\Mode_{\dimROM}} } &
        \ipX{ \Mode_{\dimROM}, \Mode_2 \ParDer{\xi}{\Mode_1} } &
        \ldots &
        \ipX{ \Mode_{\dimROM}, \Mode_{\dimROM} \ParDer{\xi}{\Mode_{\dimROM}} }
    \end{bmatrix}
    \left( \Coefficient \otimes \Coefficient \right) \\
   - \mu
    \begin{bmatrix}
        \ipX{ \partial_{\xi}\Mode_1, \partial_{\xi} \Mode_1 } &
        \ldots &
        \ipX{ \partial_{\xi}\Mode_1, \partial_{\xi} \Mode_{\dimROM} } \\
        \vdots &
        &
        \vdots \\
        \ipX{ \partial_{\xi}\Mode_{\dimROM}, \partial_{\xi} \Mode_1 } &
        \ldots &
        \ipX{ \partial_{\xi}\Mode_{\dimROM}, \partial_{\xi} \Mode_{\dimROM} } \\
    \end{bmatrix}
    \Coefficient.
\end{multline*}}%
For our \gls{ROM}, we identify $\dimROM=7$ modes with a relative offline error of $\num{4.4e-3}$. We compare our \gls{ROM} with a \gls{POD} reduced model with $\dimROM = 7$ and $\dimROM = 32$. The results are presented in \Cref{tab:burgersEquation} and \Cref{fig:burgers}. Clearly, our approach outperforms \gls{POD} with the same number of modes. To achieve a comparable error, \gls{POD} requires more than four times as many modes as our framework. 
\begin{table}[ht]
	\centering
	\caption{Approximation quality of \glspl{ROM} for Burgers' equation}
	\label{tab:burgersEquation}
	
	\begin{tabular}{lrrr}
		\toprule
		& \gls{sPOD} $\dimROM=7$ & \gls{POD} $\dimROM = 7$ & \gls{POD} $\dimROM = 32$\\\midrule
		relative offline error & \num{4.4e-3} & \num{7.2e-2} & \num{1.7e-3} \\
		relative online error & \num{3.8e-3} & \num{2.1e-1} & \num{3.5e-3}\\\bottomrule
	\end{tabular}
\end{table}
Comparing the error plots in \Cref{fig:burgers} we observe, in agreement with \Cref{ex:advectionEquationPOD}, strong oscillations in the case of the \gls{POD} approximation with $\dimROM=7$ modes, and they remain present even for the approximation with $\dimROM=32$ modes. 
In contrast, the \gls{sPOD} absolute error does not exhibit this behaviour, the error is dominated by the region where the verge of the transport front lies. 
We remark that the shift $\Path \left( t \right)$ computed from the \gls{ROM} is indeed nonlinear (see \Cref{fig:burgersShift}), as is expected, since the transport velocity for the nonlinear test case of the Burgers' equation depends on the solution itself. 

\begin{figure}
	\captionsetup[subfigure]{width=0.9\textwidth}
	\centering
	\begin{subfigure}{.48\linewidth}
		\centering
        \includegraphics{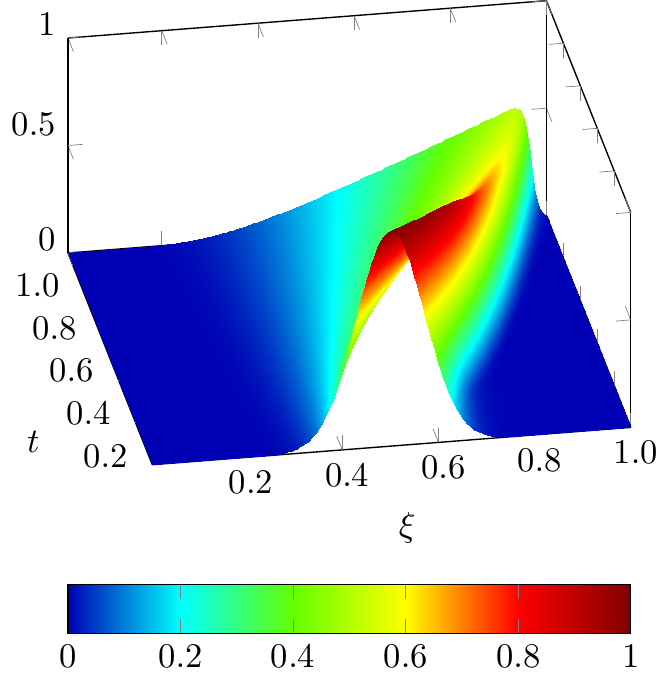}
        \caption{\gls{FOM} with parameter $\mu=\num{2e-3}$ and $200$ degrees of freedom.}
        \label{fig:burgers:FOM}
    \end{subfigure}\quad
    \begin{subfigure}{.48\linewidth}
    	\centering
        \includegraphics{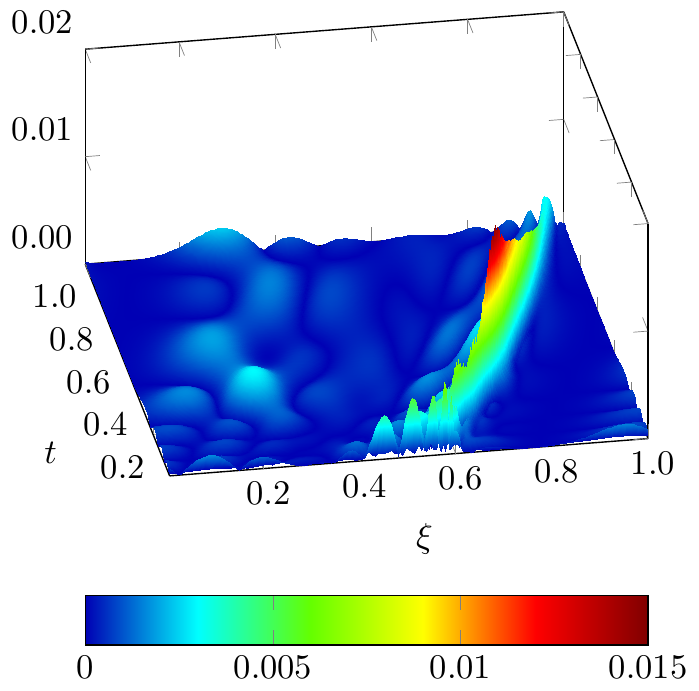}
        \caption{Absolute error for the shifted POD approximation with $\dimROM=7$ modes.}
    \end{subfigure}\\[.7em]
    \begin{subfigure}{.48\linewidth}
    	\centering
        \includegraphics{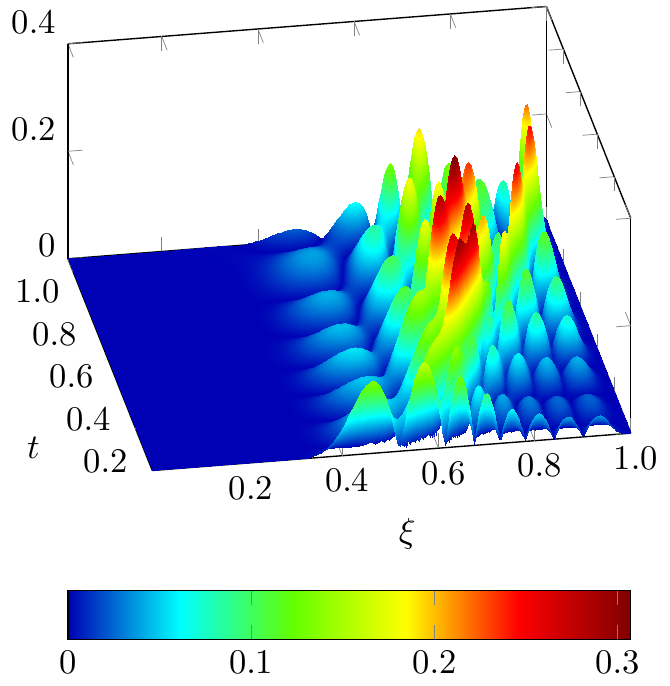}
        \caption{Absolute error for the POD approximation with $\dimROM=7$ modes.}
    \end{subfigure}\quad
    \begin{subfigure}{.48\linewidth}
    	\centering
        \includegraphics{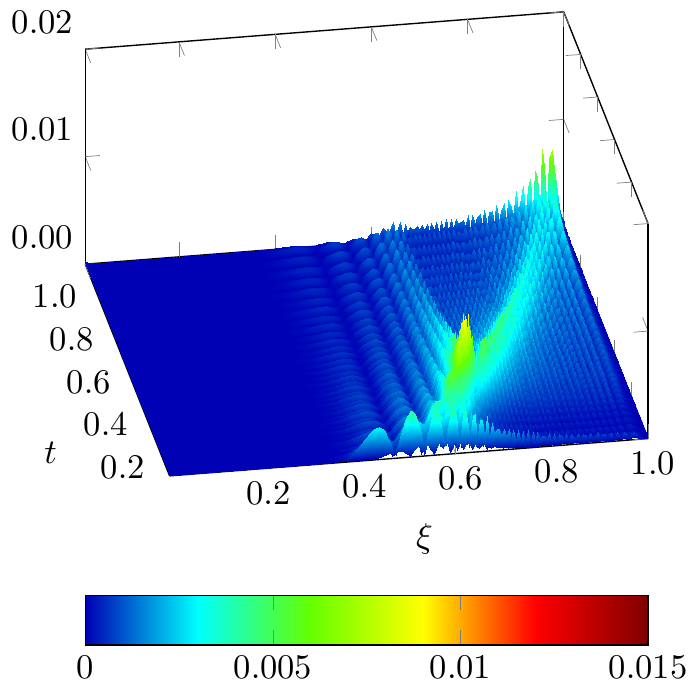}
        \caption{Absolute error for the POD approximation with $\dimROM=32$ modes.}
    \end{subfigure}
    \caption{Burgers' equation -- Solution of the \gls{FOM} and absolute errors for different \glspl{ROM}.}
    \label{fig:burgers}
\end{figure}

\begin{figure}
    \centering
    \includegraphics{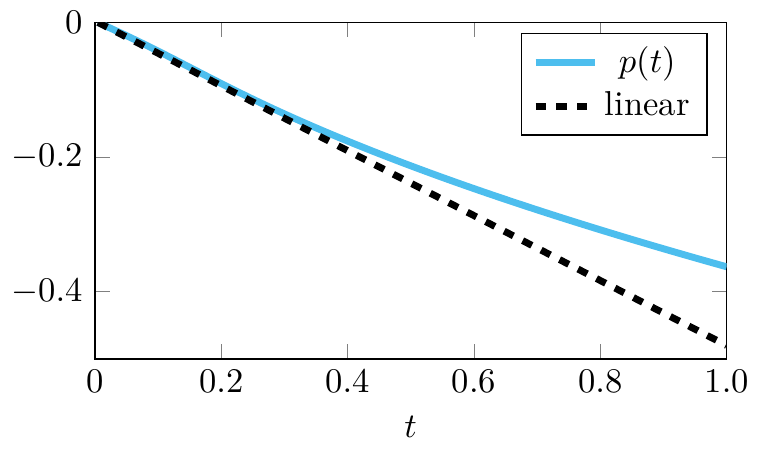}
    \caption{Burgers' equation -- The shift variable is nonlinear.}
    \label{fig:burgersShift}
\end{figure}

\section{Conclusions}

In this paper, we introduce a new framework for constructing reduced order models based on the approximation ansatz \eqref{eq:multiframeAnsatz}, which features multiple time-dependent transformation operators.
This ansatz allows obtaining accurate low-dimensional surrogate models even for systems whose dynamics are dominated by multiple transport modes with potentially large gradients.
The construction of the \gls{ROM} is based on residual minimization and extends the ideas of the moving finite element method to model reduction.
Furthermore, we provide a residual-based a-posteriori error bound.
For the particular case that only one isometric transformation operator is employed, we show a connection between our method and the symmetry reduction framework, c.f. \cite{RowM00, BeyT04}.
Further contributions include a thorough literature review of related approaches and analysis for the identification of optimal basis functions on the infinite-dimensional level.
We illustrate our theoretical findings with several analytical and numerical examples.

The problem of identifying optimal basis functions is currently solved by a first-discretize-then-optimize approach. For future work, it is interesting to analyze the first-optimize-then-discretize approach and compare the results with the current strategy. 
Furthermore, we plan to investigate the efficient implementation of the \gls{ROM}. Notably, the combination with hyper reduction techniques is a promising research direction to obtain an efficient offline/online decomposition.

\subsection*{Acknowledgements} 

We are grateful for helpful discussions with Mathieu Rosi\`ere, Fredi Tröltzsch, and Christoph Zimmer from TU Berlin. 

\bibliographystyle{plain}
\bibliography{literature}    

\vfill
\printglossary[type=acronym,title=List of Abbreviations,toctitle=List of Abbreviations]

\end{document}